\documentclass[reqno,oneside,10pt]{amsart}

\usepackage{graphicx}
\usepackage{hyperref}
\usepackage{enumerate}
\usepackage[all]{xy}
\usepackage{amsmath,amsfonts,amssymb}
\usepackage{color}
\usepackage[latin9]{inputenc}

\newcommand{\prho}{\overline{\rho}} % partial coaction 

\newtheorem{prop}{Proposition}[section]
\newtheorem{cor}[prop]{Corollary} 
\newtheorem{corollary}[prop]{Corollary} 
\newtheorem{teo}[prop]{Theorem}
\newtheorem{thm}[prop]{Theorem}
\newtheorem{lemma}[prop]{Lemma}

\theoremstyle{definition}
\newtheorem{defi}[prop]{Definition}
\newtheorem{exmp}[prop]{Example}
\newtheorem{exmps}[prop]{Examples}
\newtheorem{example}[prop]{Example}

\theoremstyle{remark}
\newtheorem{remark}[prop]{Remark}

\newcommand{\cosmash}{>\!\!\blacktriangleleft}

\newcommand{\um }{1_A}

\newcommand{\benu}{\begin{enumerate}}
\newcommand{\enu}{\end{enumerate}}

\newcommand{\beqna}{\begin{eqnarray}}
\newcommand{\eqna}{\end{eqnarray}}
\newcommand{\beqnast}{\begin{eqnarray*}}
\newcommand{\eqnast}{\end{eqnarray*}}
\newcommand{\beqn}{\begin{equation}}
\newcommand{\eqn}{\end{equation}}
\newcommand{\beqnst}{\begin{equation*}}
\newcommand{\eqnst}{\end{equation*}}

\usepackage{latexsym}

\usepackage{xspace}
\usepackage{amscd}
\usepackage{amssymb}
\usepackage{amsfonts}

\newcommand{\TDelta}{\widetilde{\Delta}}
\newcommand{\Tepsilon}{\widetilde{\epsilon}}
\newcommand{\TS}{\widetilde{S}}
\newcommand{\HDelta}{\widehat{\Delta}}
\newcommand{\Hepsilon}{\widehat{\epsilon}}

\makeatother

\newcommand{\ud}{\underline{\Delta}}
\newcommand{\ue}{\underline{\epsilon}}
\newcommand{\te}{\tilde{\varepsilon}}
\newcommand{\td}{\tilde{\Delta}}

\newcommand{\bema}{\left ( \begin{array}}
\newcommand{\ema}{\end{array} \right )}

\newcommand{\Hom}{\operatorname{Hom}}
\newcommand{\End}{\operatorname{End}}

\newcommand{\llangle}{\langle \! \langle}
\newcommand{\rrangle}{\rangle \! \rangle}
\newcommand{\lllangle}{\langle \! \! \langle \! \! \langle}
\newcommand{\rrrangle}{\rangle \! \! \rangle \! \! \rangle}

\newcommand{\ot}{\otimes}

\def\Oo{{\mathcal O}}
\def\ol{\overline}
   
\def\Fun{{\sf Fun}}
\def\Alg{{\sf Alg}}

\def\ParAct{{\sf ParAct}}
\def\ParCoAct{{\sf ParCoAct}}
\def\ul{\underline}

\newcommand{\thlabel}[1]{\label{th:#1}}
\newcommand{\thref}[1]{Theorem~\ref{th:#1}}
\newcommand{\selabel}[1]{\label{se:#1}}
\newcommand{\seref}[1]{Section~\ref{se:#1}}
\newcommand{\lelabel}[1]{\label{le:#1}}
\newcommand{\leref}[1]{Lemma~\ref{le:#1}}
\newcommand{\prlabel}[1]{\label{pr:#1}}
\newcommand{\prref}[1]{Proposition~\ref{pr:#1}}

\newcommand{\exlabel}[1]{\label{ex:#1}}

\newcommand{\delabel}[1]{\label{de:#1}}
\newcommand{\deref}[1]{Definition~\ref{de:#1}}
\newcommand{\eqlabel}[1]{\label{eq:#1}}
\newcommand{\equref}[1]{(\ref{eq:#1})}
\def\Cc{{\mathcal C}}

\def\Mm{{\mathcal M}}
\def\Hh{{\mathcal H}}
\def\bul{\bullet}
\newcommand\bk[1]{\left<#1\right>}

\title[Dual Constructions for Partial Actions]{Dual Constructions for Partial Actions of Hopf Algebras}

\author[E. Batista]{Eliezer Batista}
\address{Departamento de Matem\'atica, Universidade Federal de Santa Catarina, Brazil}
\email{ebatista@mtm.ufsc.br}
\author[J. Vercruysse]{Joost Vercruysse}
\address{D\'epartement de Math\'ematiques, Universit\'e Libre de Bruxelles, Belgium}
\email{jvercruy@ulb.ac.be}

\begin{document}

\begin{abstract}
The duality between partial actions (partial $H$-module algebras) and co-actions (partial $H$-comodule algebras) of a Hopf algebra $H$ is fully explored in this work. A connection between partial (co)actions and Hopf algebroids is established under certain commutativity conditions. Moreover, we continue this duality study, introducing also partial $H$-module coalgebras and their associated $C$-rings, partial $H$-comodule coalgebras and their associated cosmash coproducts, as well as the mutual interrelations between these structures.
\end{abstract} 

\maketitle

\setcounter{tocdepth}{2}
\tableofcontents

\section{Introduction}

The theory of Hopf algebras traditionally has been strongly inspired by group theory. This is due to two dual constructions leading to mayor classes of examples of Hopf algebras and the (co)algebras upon which they (co)act. First of all, given any group $G$ and a commutative base ring $k$, the associated group algebra $H=kG$ is a Hopf $k$-algebra. Actions of the group $G$ are in correspondence with module coalgebras over $H$ and representations of $G$ are exactly modules over $H$. In a dual way, supposing now that $G$ is moreover finite, the algebra of $k$-valued functions on $G$ is a again a Hopf algebra, say $K$. Actions of $G$ now correspond to comodule algebras over $K$ and representations of $G$ are precisely the $K$-comodules. If $G$ is more generally an algebraic group, one could take $K$ to be the algebra of rational functions on $G$ and obtain the same results. The Hopf algebras $H$ and $K$, as well as their respective actions and coactions, are related by dual pairings. 

The theory  of partial actions of groups first appeared in 1994 with the seminal work of Ruy Exel in the attempt to better describe some $C^*$ algebras which are graded by the additive group of integers but nevertheless they cannot be written as ordinary crossed products \cite{Ruy1, Ruy2}. Still in the context of operator algebras, several developments were made, such as the theory of globalization, Takai duality \cite{Aba1}, the characterization of partial crossed products \cite{QuiggRae}, the K-theory of partial crossed products \cite{McCla} and a connection between partial actions and groupoids \cite{Aba2}. In recent years, a purely algebraic theory of partial group actions has been developed along the lines of the theory existing in the context of operator algebras \cite{dok}. The globalization problem has been studied for a long time in the context of operator algebras \cite{Aba1}. In the case of partial actions of groups on algebras the existence of globalization is restricted to partial actions in which the ideals defined by the partial action are unital ideals \cite{dok}. For partial actions of groups on algebras without units, the globalization can be achieved modulo Morita equivalence \cite{AbaRuyMishaSimon}.

The Galois theory for partial group actions \cite{DFP} motivated the introduction of partial actions for Hopf algebras \cite{caen06}, which on its turn triggered a new branch of research in Hopf algebra theory. Partial actions and coactions of Hopf algebras were verified to have nice properties with relation to globalization, that is, every partial action (resp. coaction) of a Hopf algebra $H$ on an algebra $A$ is coming from a restriction of an action (resp. coaction) of $H$ on a bigger algebra $B$ such that $A$ is isomorphic to a unital ideal of $B$ \cite{AB,AB3}. This can be better understood by the fact that partial actions of the group algebra $kG$ are in correspondence with partial actions of the group $G$ such that the ideals defined by the partial action are unital \cite{caen06}. 

As in the global case, the category of partial modules over a Hopf algebra forms a monoidal category, in which the algebra objects are exactly the partial actions \cite{ABCV}. The notion of a partial module over a Hopf algebra $H$ is closely related to that of a partial representation of $H$, and in case where $H$ equals the group algebra $kG$, the latter coincide with the partial representations of the group $G$.

In geometric terms, partial group actions describe no global symmetries of a space, but only ``local'' symmetries of certain subspaces. We make this viewpoint more explicit in \seref{geometry}, where we introduce the notion of a partial action of an algebraic group on an affine space. As it is the case for usual actions, such a partial action allows in a natural way to construct examples of partial coactions on the coordinate algebras of these spaces. 
Moreover, this construction allows us to interpret partial coactions of Hopf algebras as partial symmetries in noncommutative geometry. 

This observation is in fact the starting point and motivation to develop a complete theory of dual constructions for partial actions of Hopf algebras, together with their mutual dualities. Moreover, we show that the internal algebraic structures associated to partial (co)actions are sometimes richer than first expected.

In particular, given a partial action of a Hopf algebra $H$ on an algebra $A$, one can construct the {\em partial smash product} $A\# H$ \cite{caen06}, which is an $A$-ring. In case where $H$ is cocommutative and $A$ is commutative, we show in \seref{commutative} that $A\# H$ has even the structure of a $A$-Hopf algebroid. 

Secondly, given a commutative Hopf algebra $H$ and right partial coaction on a commutative algebra $A$, one can construct a commutative $A$-Hopf algebroid out of these data, called the {\em partial split Hopf algebroid}, denoted by  $A\underline{\otimes} H$. The interconnection between partial coactions of commutative Hopf algebras on commutative algebras and Hopf algebroids also works in the opposite direction. More precisely, given a commutative $A$-Hopf algebroid $\mathcal{H}$ and a Hopf algebroid morphism $F:H\rightarrow \mathcal{H}$ such that $\mathcal{H}$ can be totally defined from the image of $F$ and the image of the source map $s:A\rightarrow \mathcal{H}$, then there is a natural partial coaction of $H$ on the base algebra $A$ constructed from these data. This is a dual version of a theorem which stablishes an interconnection between groupoids and partial group actions, this connection first appeared in the operator algebraic context in \cite{Aba2}, but we follow a purely algebraic formulation due to Kellendonk and Lawson \cite{KL}, which gives an equivalence, in the categorical sense, between partial actions of a group $G$ and star injective functors between groupoids $\mathcal{G}$ and the group $G$.

The partial smash product Hopf algebroid and the partial split Hopf algebroid are in duality. More precisely, consider a co-commutative Hopf algebra $H$, and a commutative Hopf algebra $K$ such that there exists a pairing between them. Let $A$ be a commutative algebra which is a right partial $K$-comodule algebra. Then $A$ is also a left partial $H$-module algebra and there exists a skew pairing, in the sense of Schauenburg \cite{S}, between  $\underline{A\# H}$ and $A\underline{\otimes} H$, considering them as $\times_A$ bialgebras.
All this is done in \seref{comodulealgebras}.
 
In \seref{modulecoalgebras}, we study partial $H$-module coalgebras. For the Hopf algebra being the group algebra $kG$, for some group $G$, the partial $kG$-module coalgebras correspond to partial actions of the group $G$ on coalgebras. We show that a left $H$-module coalgebra over a bialgebra $H$ gives rise to a $C$-ring structure over a subspace $\underline{H\otimes C}\subseteq H\otimes C$.

Partial $H$ comodule coalgebras were first defined in \cite{wang}. In \seref{comodulecoalgebras}, we use this notion to construct partial cosmash coproducts, which are naturally coalgebras over the base field. Given two Hopf algebras $H$ and $K$ such that there is a dual pair between them, there is again a duality between partial $H$-module algebras and partial $K$-comodule algebras. Moreover, if $A$ is a left partial $H$-module algebra and $C$ is a left partial $K$-comodule coalgebra such that there is a dual pairing between them, then there exists a dual pairing between the partial smash product $\underline{A\# H}$ and the partial cosmash coproduct $\underline{C\cosmash K}$.

Finally, in \seref{conclusions}, the global picture of all dualities described in this paper is pointed out and possible directions for further developments of the theory are outlined.

In the appendix, some additional proofs that were omitted in the final journal version are included.

\section{Preliminaries: Partial actions and partial representations}\selabel{preliminaries}

\subsection{Algebraic structures}\selabel{algstruc}

\subsubsection*{Hopf algebroids}

Throughout this note, $k$ denotes a field, however, often a commutative ring is sufficient. An algebra, coalgebra or Hopf algebra means the respective structure with respect to the base $k$, unadorned tensor products are tensor products over $k$.

Let $A$ be a $k$-algebra. An {\em $A$-coring} is a coalgebra object in the category of $A$-bimodules, i.e. it is a triple $(\Cc,\Delta,\epsilon)$ where $\Cc$ is an $A$-bimodule, $\Delta:\Cc\to \Cc\ot_A\Cc$ and $\epsilon:\Cc\to A$ are  $A$-bimodule maps satisfying the usual coassociativity and counit conditions. 

Given a $k$-algebra $A$, a left (resp. right) {\em bialgebroid} over $A$ is given by the data $(\mathcal{H}, A, s_l ,t_l ,\ud_l , \ue_l )$ (resp. $(\mathcal{H}, A, s_r , t_r , \ud_r , \ue_r )$) such that:
\begin{enumerate}
\item $\mathcal{H}$ is a $k$ algebra.
\item The map $s_l$ (resp. $s_r$) is a morphism of algebras from $A$ to $\mathcal{H}$, and the map $t_l$ (resp. $t_r$) is an anti-morphism of algebras from $A$ to $\mathcal{H}$. Their images commute, that is, for every $a,b\in A$ we have 
$s_l (a)t_l (b)=t_l (b)s_l (a)$ (resp. 
$s_r (a)t_r (b)=t_r (b)s_r (a)$).
By the maps $s_l ,t_l $ (resp. $s_r , t_r$) the algebra $\mathcal{H}$ inherits a structure of $A$ bimodule given by $a\triangleright h \triangleleft b =s_l (a)t_l (b)h$ (resp. $a\triangleright h \triangleleft b =hs_r (b)t_r (a)$).
\item The triple $(\mathcal{H},\ud_l , \ue_l )$ (resp. $(\mathcal{H}, \ud_r , \ue_r )$) is an $A$-coring relative to the structure of $A$-bimodule defined by $s_l $ and $t_l$ (resp. $s_r $, and $t_r$).
\item The image of $\ud_l$ (resp. $\ud_r $) lies on the Takeuchi subalgebra
\[
\mathcal{H}\times_A \mathcal{H} =\{ \sum_i h_i \otimes k_i \in \mathcal{H}\otimes_A \mathcal{H} \, |\, \sum_i h_i t_l (a) \otimes k_i =\sum_i h_i \otimes k_i s_l (a) \, \; \forall a\in A \} ,
\]
resp.
\[
\mathcal{H} {_A\times} \mathcal{H} =\{ \sum_i h_i \otimes k_i \in \mathcal{H}\otimes_A \mathcal{H} \, |\, \sum_i
s_r (a) h_i  \otimes k_i =\sum_i h_i \otimes t_r (a) k_i \, \; \forall a\in A \} ,
\]
and it is an algebra morphism.
\item For every $h,k\in \mathcal{H}$, we have
\[
\ue_l (hk)=\ue_l (hs_l (\ue_l (k)))=\ue_l (ht_l (\ue_l (k))) ,
\]
resp.
\[
\ue_r (hk) =\ue_r (s_r (\ue_r (h))k)=\ue_r (t_r (\ue_r (h))k) .
\]
\end{enumerate}

Given two anti-isomorphic algebras $A$ and $\tilde{A}$ (ie, $A\cong \tilde{A}^{op}$) and an algebra $\Hh$ that is endowed with at the same time a left $A$-bialgebroid structure $(\mathcal{H}, A, s_l ,t_l ,\ud_l , \ue_l )$ and a right $\tilde{A}$-bialgebroid structure $(\mathcal{H}, \tilde{A}, s_r, t_r , \ud_r , \ue_r )$, we say that $\Hh$ is a {\em Hopf algebroid}\label{defHopfalgebroid} if it is equipped with an antipode, that is, an anti algebra homomorphism $\mathcal{S}:\mathcal{H}\rightarrow \mathcal{H}$ such that
\begin{enumerate}
\item[(i)] $s_l \circ \ue_l \circ t_r =t_r$, $t_l \circ \ue_l \circ s_r =s_r$, $s_r \circ \ue_r \circ t_l =t_l$ and $t_r\circ \ue_r \circ s_l =s_l$;
\item[(ii)] $(\ud_l \otimes_{\tilde A} I)\circ \ud_r =(I\otimes_A \ud_r )\circ \ud_l$ and
 $\quad (I\otimes_{\tilde A} \ud_l )\circ \ud_r =(\ud_r \otimes_A I )\circ \ud_l$;
\item[(iii)] $\mathcal{S}(t_l (a)ht_r (b'))=s_r( b')\mathcal{S}(h) s_l (a)$,
for all $a\in A$, $b'\in \tilde{A}$ and $h\in \mathcal{H}$;
\item[(iv)] $\mu_{\mathcal{H}} \circ (\mathcal{S} \otimes I)\circ \ud_l =s_r \circ \ue_r$ and $\mu_{\mathcal{H}} \circ (I\otimes \mathcal S)\circ \ud_r =s_l \circ \ue_l$.
\end{enumerate} 
An example of a Hopf algebroid is the so-called {\em split Hopf algebroid}: Let $H$ be a commutative Hopf algebra and $A$ a commutative $H$-comodule algebra. Then $A\ot H$ is an $A$-Hopf algebroid with structure maps $t_r(a)=s_l(a)=s(a)=a\ot 1_H$, $s_r(a)=t_l(a)=a^{[0]}\ot a^{[1]}$, $\Delta_r(a\ot h)=\Delta_l(a\ot h)=a\ot h_{(1)}\ot_A 1\ot h_{(2)}$, $\epsilon_r(a\ot h)=\epsilon_l(a\ot h)=a\epsilon_H(h)$, $\mathcal S(a\ot h)=a^{[0]}\ot a^{[1]} S_H(h)$.

\subsubsection*{Dual pairings}

A {\em dual pairing} between a $k$-algebra $A$ and a $k$-coalgebra $C$ is a $k$-linear map
$$\bk{-,-}:A\ot C\to k$$
satisfying the following conditions
$$\bk{ab,c}=\bk{a,c_{(1)}}\bk{b,c_{(2)}},\qquad \bk{1_A,c}=\epsilon_C(c),$$
for all $a\in A$, $c\in C$.
A dual pairing induces two $k$-linear morphisms 
\begin{eqnarray*}
\phi:A\to C^*,&& \phi(a)(c)=\bk{a,c}\\
\psi:C\to A^*,&& \phi(c)(a)=\bk{a,c}
\end{eqnarray*}
for all $a\in A$, $c\in C$. Clearly, a linear map $\bk{-,-}$ is a dual pairing if and only if the induced map $\phi$ is an algebra morphism. Hence, a dual pairing induces a functor
$$\Phi:\Mm^C\to{_A\Mm}$$
that sends a $C$-comodule $M$ to an $A$-module with the same underlying $k$-module $M$ and $A$-action given by
$$a\cdot m=m^{[0]}\bk{a,m^{[1]}}.$$
The following is equivalent for a pairing $\bk{-,-}$ between $A$ and $C$
\begin{enumerate}[(i)]
\item $\phi$ is injective;
\item the image of $\psi$ is dense in $A^*$ with respect to the finite topology;
\item $\bk{-,-}$ is left non-degenerate, i.e.\ if $\bk{a,c}=0$ for all $c\in C$ then $a=0$.
\end{enumerate}
We call $\bk{-,-}$ {\em non-degenerate} or a {\em rational pairing} if it is both left and right non-degenerate, i.e. if and only if $\phi$ is injective and has a dense image with respect to the finite topology. If the pairing is non-degenerate, then it is known that the functor $\Phi$ can be co-restricted to an isomorphism of categories between the category of $C$-comodules and the category of rational $A$-modules, which are the $A$-modules $M$ such that each cyclic submodule $Am\subset M$ is finitely generated as a $k$-module.

A dual pairing between bialgebras $H$ and $K$ is a dual pairing between the underlying coalgebras and algebras in both ways. One observes that a map $\bk{-,-}:H\ot K\to k$ is a dual pairing of Hopf algebras if and only if the associated morphisms $\phi:H\to K^*$ and $\psi:K\to H^*$ induce bialgebra morphisms $\phi:H\to K^\circ$ and $\psi:K\to H^\circ$, where $B^\circ$ denotes the finite dual of a bialgebra $B$. A dual pairing between Hopf algebras $H$ and $K$ is just a dual pairing between the underlying bialgebras. As any bialgebra morphism between Hopf algebras preserves the antipode, a dual pairing between Hopf algebras respects the antipode in the following way
$$\bk{h,S_K(x)}=\bk{S_H(h),x}, \forall h\in H, x\in K.$$ 
In case of a (non-degenerate) pairing between bialgebras, the functor $\Phi$ is a monoidal functor, hence it sends algebras to algebras in the respective categories.

\subsection{Partial actions}

The following definition, of partial actions of Hopf algebras, first appeared in \cite{caen06} and was motivated by examples of partial actions of groups on algebras. 

\begin{defi} A {\em left partial action} of a Hopf algebra $H$ over a unital algebra $A$ is a linear map 
\[
\begin{array}{rccc} \cdot : & H\otimes A & \rightarrow & A\\
                    \,      & h\otimes a & \mapsto     & h\cdot a
\end{array}
\]
such that
\begin{enumerate}
\item[(PLA1)] For every $a\in A$, $1_h \cdot a =a$.
\item[(PLA2)] For every $h\in H$ and $a,b\in A$, $h\cdot (ab)=
(h_{(1)} \cdot a)(h_{(2)}\cdot b)$.
\item[(PLA3)] For every $h,k\in H$ and $a\in A$, $h\cdot (k\cdot a)=
(h_{(1)}\cdot 1_A)(h_{(2)}k\cdot a)$
\end{enumerate}
The partial action is \underline{\em symmetric} if in addition
\begin{enumerate} 
\item[(PLA3')] For every $h,k\in H$ and $a\in A$, $h\cdot (k\cdot a)=
(h_{(1)}k\cdot a)(h_{(2)}\cdot 1_A)$
\end{enumerate}
In this case, the algebra $A$ is said to be a partial left $H$ module algebra. Given two partial left $H$-module algebras $A$ and $B$ , a morphism of partial actions is an algebra morphism $f:A\rightarrow B$ such that, for every $h\in H$ and $a\in A$ we have $h\cdot_B f(a)=f(h\cdot_A a)$. The category of left partial actions of $H$ will be denoted by ${}_H\ParAct$
\end{defi}

\begin{remark} Throughout this paper, all partial actions will be considered to be symmetric. It is easy to see that $H$-module algebras are partial $H$-module algebras, in fact, one can prove that a partial action is global if and only if for every $h\in H$ we have $h\cdot 1_A =\epsilon (h)1_A$.
\end{remark}

Given a partial action of a Hopf algebra $H$ on a unital algebra $A$, one can define an associative product on $A\otimes H$, given by
\[
(a\otimes h)(b\otimes k)=a(h_{(1)}\cdot b)\otimes h_{(2)}k,
\]
for all $a,b\in A$ and $h,k\in H$.
Then, a new unital algebra is constructed as
\[
\underline{A\# H}=(A\otimes H)(\um \otimes 1_H ) .
\]
This unital algebra is called partial smash product \cite{caen06}. This algebra is generated by typical elements of the form
\[
a\# h= a(h_{(1)}\cdot \um )\otimes h_{(2)} .
\]
One then proves that
\begin{equation}\label{propsmash}
a\# h= a(h_{(1)}\cdot \um )\# h_{(2)} ,
\end{equation}
and that
\begin{equation} \label{propsmash2}
(a\# h)(b\# k)= a(h_{(1)} \cdot b ) \# h_{(2)} k .
\end{equation}

\begin{exmps} \exlabel{partact}
\begin{enumerate}
\item 
\cite{Ruy1,dok} A group $G$ acts partially on a set $X$, if there exists a family of subsets $\{X_g\}_{g\in G}$ and a family of bijections $\alpha_g:X_{g^{-1}}\to X_g$ such that
\begin{enumerate}
\item $X_e =X$, and $\alpha_e =\mbox{Id}_X$;
\item $\alpha_g (X_{g^{-1}} \cap X_h )=X_g \cap X_{gh}$;
\item If $x\in X_{h^{-1}}\cap X_{(gh)^{-1}}$, then $\alpha_g \circ \alpha_h(x) =\alpha_{gh}(x)$.
\end{enumerate}
One says moreover that a group $G$ acts partially on an algebra $A$, if $G$ acts partially on the underlying set of the algebra $A$, such that each $A_g$ is an ideal of $A$ and each $\alpha_g$ is multiplicative. If $G$ acts partially by $\alpha$ on the set $X$, then $G$ acts partially by $\theta$ on the algebra $A=\Fun(X,k)$, where $A_g=\Fun(X_g,k)$ and $\theta_g(f)(x)=f(\alpha_{g^{-1}}(x))$, where $f\in A_{g^{-1}}$ and $x\in X_g$. Remark that in this example, the ideals $A_g=\Fun(X_g,k)$ are unital algebras.
Partial actions of the group algebra $kG$ over any unital algebra $A$ are one-to-one correspondence with partial actions of the group $G$ on $A$ in which the domains $A_g$ are unital ideals, that is, there exists a central idempotent $1_g \in A$ such that $A_g =1_g .A$, and the partially defined isomorphisms $\alpha_g :A_{g^{-1}}\rightarrow A_g$ are unital algebra isomorphisms for each $g\in G$.

\item Recall from \cite{AB} that if $B$ is an $H$-module algebra and $\mathbf{e}$ is a central idempotent in $B$, then there exist a partial action of $H$ on the ideal $A=\mathbf{e}B$, given by
\[
h\cdot a =\mathbf{e} (h\triangleright a),
\]
where $h\in H$, $a\in A$ and $\triangleright:H\ot B\to B$ is the (global) action of $H$ on $B$.
\end{enumerate}
\end{exmps}

Of course, one can introduce in a symmetric way the category $\ParAct_H$ of {\em right} partial actions of $H$, also called right partial $H$-module algebras, and morphisms between them. Furthermore, if $A$ is a right partial $H$-module algebra, then one can construct the right partial smash product $\ul{H\# A}=(1_H\ot 1_A)(H\ot A)$.

\subsection{Partial representations}\selabel{representations}

Closely related to partial actions, there is the concept of a partial representations. Partial representations were first defined for groups. For any partial action of a group $G$ on a unital algebra we construct a partial representation associated to it. On the other hand, if we have a partial representation, one can define a partial action in order to relate this original partial representation with a partial crossed product, see \cite{Ruy1} \cite{QuiggRae} \cite{dok0}. 

\begin{defi} \label{partialrep} \cite{ABCV}
Let $H$ be a Hopf $k$-algebra, and let $B$ be a unital $k$-algebra. A \emph{partial representation} of $H$ on $B$ is a linear map $\pi: H \rightarrow B$ such that 
\begin{enumerate}[{(PR1)}]
\item $\pi (1_H)  =  1_B$; \label{partialrep1}
\item $\pi (h) \pi (k_{(1)}) \pi (S(k_{(2)}))  =   \pi (hk_{(1)}) \pi (S(k_{(2)})) $;
\label{partialrep2}
\item $\pi (h_{(1)}) \pi (S(h_{(2)})) \pi (k)  =   \pi (h_{(1)}) \pi (S(h_{(2)})k)$; \label{partialrep3} 
\item $\pi (h) \pi (S(k_{(1)})) \pi (k_{(2)}) = \pi (hS(k_{(1)})) \pi (k_{(2)})$; \label{partialrep4}
\item $\pi (S(h_{(1)}))\pi (h_{(2)}) \pi (k) = \pi (S(h_{(1)}))\pi (h_{(2)} k)$. \label{partialrep5}
\end{enumerate}
A left (resp.\ right) {\em partial $H$-module} is a pair $(M,\pi )$ in which $M$ is a $k$-vector space and $\pi$ is a partial representation of $H$ on the algebra $\mbox{End}_k (M)$ (resp.\ $\mbox{End}_k (M)^{op}$).  A morphism between two partial $H$-modules $(M,\pi )$ and $(N, \phi )$ is a linear function $f:M\rightarrow N$ such that for each $h\in H$ one has, $f\circ \pi (h) =\phi (h)\circ f$. The category of left (resp.\ right) partial $H$-modules is denoted by ${}_H \mathcal{M}^{par}$ (resp.\ $\mathcal{M}_H^{par}$).
\end{defi}

Let left $(M,\pi)$ be a partial $H$-module For $h\in H$, we will denote $h\bullet \underline{\; }=\pi(h)$ in the algebra $\End_k (M)$.

It was shown in \cite{ABCV} that the category of partial modules is equivalent to a category of modules over a Hopf algebroid $H_{par}$, which can be explicitly constructed out of $H$ by means of a universal property (hence it is a quotient algebra of the free tensor algebra over $H$), in particular, there is a partial representation $[\underline{\;}]:H\to H_{par}$. We denote the action of $H_{par}$ on a partial $H$-module $M$ as 
$x\triangleright m$, for $x\in H_{par}$ and $m\in M$. In particular, for $h\in H$ and $x=[h]\in H_{par}$, we have
$[h]\triangleright m=h\bul m$. 

As a consequence of the Hopf algebroid structure, the category of partial $H$-modules is closed monoidal and allows a monoidal forgetful functor to the category of $A$-bimodules that preserves internal Homs. Here the base algebra $A$ of the Hopf algebroid $H_{par}$ can be identified with the subalgebra of $H_{par}$ generated by elements of the form $\epsilon_h=[h_{(1)}][S(h_{(2)})]$. For details we refer to \cite{ABCV}. Using this monoidal structure, it was shown in \cite{ABCV} that the category of (left) $H$-module algebras ${}_H\ParAct$ coincides with the category $\Alg(_H\Mm^{par})$ of algebra objects in ${}_H\Mm^{par}$, what justifies the name ``partial module algebra''. Let us now state explicitly the internal Hom structure of the category of partial modules.

\begin{prop}\prlabel{internalHoms}
Consider the objects $M,N,P\in {}_H \mathcal{M}^{par}$, then 
\begin{enumerate}[(i)]
\item $\Hom_A (M,N)$ is an object in ${}_H \mathcal{M}^{par}$, where
\[
(x\triangleright F) (m) =x_{(1)} \triangleright_N F(\mathcal{S}(x_{(2)}) \triangleright_M m) 
\]
for all $F\in \mbox{Hom}_A (M,N)$ and $x\in H_{par}$;
\item the $k$-linear map
\begin{eqnarray*}
{}_{H_{par}}\mbox{Hom}(M\otimes_A N ,P)\to {}_{H_{par}}\mbox{Hom}(M,\mbox{Hom}_A (N,P)),\ F\mapsto \widehat F,
\end{eqnarray*}
with $\widehat F(m)(n)=F(m\otimes_A n)$, for all $m\in M$, $n\in N$
is an isomorphism, natural in $M$ and $P$;
\item $\End_A(M)$ is a left partial $H$-module algebra.
\end{enumerate}
\end{prop}

We have also another duality theorem involving the dual vector space of a partial $H$-module. This reveals that the algebra $H_{par}$ should posses a richer structure than just that of an $A$-Hopf algebroid, which is a subject of future investigation.

\begin{prop} \label{dualandhoms} Given an object $M\in {}_H \mathcal{M}^{par}$, its dual vector space $M^* =\mbox{Hom}_k (M,k)$ has both left and right partial $H$ module-structures.
\end{prop}

\begin{proof} Given $\phi \in M^*$ and $h\in H$, we define the functionals $h\bullet_l \phi $ and $\phi \bullet_r h$ by
\[
(h\bullet_l \phi) (m)=([h]\triangleright \phi ) (m) =\phi ([S(h)]\triangleright_M m) , \qquad (\phi \bullet_r h)(m)=(\phi \triangleleft [h])(m) =\phi ([h]\triangleright_M m) .
\]
It is easy to see that these are, respectively, a left and a right partial $H$-module structures on $M^*$.
\end{proof}

\section{Partial actions of co-commutative Hopf algebras.}\selabel{commutative}

In \cite{ABCV} it was shown that the universal partial Hopf algebra $H_{par}$ (see  \seref{representations}), which has a structure of a Hopf algebroid, is moreover isomorphic to a partial smash product $\ul{A\# H}$, where $A$ is a particular subalgebra of $H_{par}$. In this section, we will show that a general partial smash product $\underline{A\# H}$ of a {\em co-commutative} Hopf algebra with a commutative left partial $H$ module algebra $A$ also has the structure of a Hopf algebroid, however, in this commutative case several axioms from the general definition of a Hopf algebroid simplify drastically.

First, we define the left and right source and target maps all to be the same map
\[
\begin{array}{rccl} s_l = t_l = s_r = t_r :& A & \rightarrow & \underline{A\# H}\\
\, & a & \mapsto & a\# 1_H \end{array}
\]
which might be denoted as $s$ or $t$ below. Clearly, $s$ is an algebra morphism and as $A$ is commutative, then the target map $t=s$ can be viewed as an anti morphism as well. Again by the commutativity of $A$, the images of $s$ and $t$ obviously commute in $\mathcal{H}=\underline{A\# H}$. Nevertheless, the images of $s$ and $t$ lie not necessarily in the center of $A\# H$, hence the pairs $(s_l,t_l)$ and $(s_r,t_r)$ induce different $A$-(bi)module structures on $A\# H$. Explicitly, the ``left handed'' $A$-bimodule structure in $\mathcal{H}$ is given by
\begin{equation}\label{bimoduleleft}
a\triangleright (b\# h)\triangleleft c= s_l(a)t_l (c)(b\# h)=abc\# h,
\end{equation}
and the ``right handed'' $A$ bimodule structure by
\begin{equation}\label{bimoduleright}
a\blacktriangleright (b\# h)\blacktriangleleft c= (b\# h) s_r (c) t_r (a)=b(h_{(1)}\cdot (ac))\# h_{(2)} .
\end{equation}
By the commutativity of $A$, both bimodule structures are in fact central, i.e. they are just $A$-module structures or
$$a\triangleright (b\# h)=(b\# h)\triangleleft a, \qquad a\blacktriangleright (b\# h)=(b\# h)\blacktriangleleft a.$$
Hence, different from what is done in a general Hopf algebroid, it is in our situation not needed to keep in mind the two bimodule structures on $\Hh$, but just the two module structures. In particular, we will show in lemma's \ref{le:coringleft} and \ref{le:coringright} below that if $\Hh$ is endowed with one of the $A$-module structures described above, it can moreover be endowed with a coalgebra structure. Of course these coalgebra structures can be viewed as coring structures to fit the definition of a Hopf algebroid.

In what follows, we will denote the $A$ tensor product with respect to the (bi)module structure \eqref{bimoduleleft} as $\ot_A^l$ and the tensorproduct with respect to \eqref{bimoduleright} as $\ot_A^r$.

\begin{lemma}\lelabel{coringleft} Let $H$ be a cocommutative Hopf algebra and $A$ be a commutative left partial $H$ module algebra. Then the partial smash product 
$\mathcal{H}= \underline{A\# H}$ has the structure of an $A$-coalgebra with the $A$-(bi)module structure \eqref{bimoduleleft} and comultiplication and counit given by
\begin{eqnarray*}
\underline{\Delta}_l (a\# h)&=&(a\# h_{(1)}) \otimes^l_A (\um \# h_{(2)}), \\
\underline{\epsilon}_l (a\# h) &=&a(h\cdot \um ).
\end{eqnarray*}
\end{lemma}

\begin{proof} 
By the definitions of $\underline{\Delta}$ and of $\underline{\epsilon}$, one can verify immediately that the comultiplication and the counit are $A$-linear maps
and that $\ul\Delta$ is coassociative. 
The right counit axiom reads
\beqnast
(I\otimes \underline{\epsilon}_l )\circ \underline{\Delta}_l (a\# h) & = & (a\# h_{(1)})\triangleleft (h_{(2)}\cdot \um ) 
 =  a(h_{(2)} \cdot \um )\# h_{(1)}\\
& = & a(h_{(1)} \cdot \um )\# h_{(2)}
 =  a\# h,
\eqnast
where we used the co-commutativity of $H$ and \eqref{propsmash}. Similarly, we check the left counit axiom
\beqnast
(\underline{\epsilon}_l \otimes I)\circ \underline{\Delta}_l (a\# h) & = & (a(h_{(1)}\cdot \um ))\triangleright (\um \# h_{(2)}) \\
& = & a(h_{(1)} \cdot \um )\# h_{(2)}.
 =  a\# h
\eqnast
Therefore, $(\underline{A\# H} , \ud ,\ue )$ is an $A$-coalgebra.
\end{proof}

\begin{lemma}\lelabel{coringright} Let $A$, $H$ be as in \leref{coringleft}. 
Then the partial smash product $\Hh=\underline{A\# H}$ has the structure of an $A$-coalgebra with $A$-(bi)module structure given by \eqref{bimoduleright},
and comultiplication and counit given by
\begin{eqnarray*}
\underline{\Delta}_r (a\# h) &=& (a\# h_{(1)}) \otimes^r_A (\um \# h_{(2)}), \\
\underline{\epsilon}_r (a\# h) &=&S(h)\cdot a. 
\end{eqnarray*}
\end{lemma}

\begin{proof} Let us begin with the $A$-linearity of the comultiplication. 
\beqnast 
\ud_r (a\blacktriangleright (b\# h )) & = & 
\ud_r (b(h_{(1)}\cdot a)\# h_{(2)}) 
 =  b(h_{(1)}\cdot a)\# h_{(2)} \otimes^r_A 1_A \# h_{(3)} \\
& = & a\blacktriangleright (b\# h_{(1)}) \otimes^r_A (1_A \# h_{(2)}) = a\blacktriangleright \ud_r (b\# h).
\eqnast
For the $A$-linearity of the counit, we have
\beqnast
\ue_r (a\blacktriangleright (b\# h )) & = & \ue_r (b(h_{(1)}\cdot a)\# h_{(2)}) 
 =  S(h_{(2)})\cdot (b(h_{(1)}\cdot a)) \\
& = & (S(h_{(3)})\cdot b)(S(h_{(2)})\cdot (h_{(1)}\cdot a))
 =  (S(h_{(3)})\cdot b)(S(h_{(1)})h_{(2)}\cdot a)\\
& = & (S(h)\cdot b)a 
 =  a\ue_r (b\# h)  .
\eqnast
The coassociativity is straightforward, then it remains to verify the counit axiom:
\beqnast
(I\otimes \ue_r )\circ \ud_r (a\# h) & = & (a\# h_{(1)})\blacktriangleleft (S(h_{(2)})\cdot 1_A) 
=a (h_{(1)}\cdot (S(h_{(3)})\cdot 1_A))\# h_{(2)} \\
& = & a (h_{(1)}\cdot 1_A ) (h_{(2)}S(h_{(4)})\cdot 1_A)\# h_{(3)} \\
& = & a (h_{(1)}\cdot 1_A ) (h_{(2)}S(h_{(3)})\cdot 1_A)\# h_{(4)} \\
& = & a (h_{(1)}\cdot 1_A )\# h_{(2)} 
 =  a\# h ,
\eqnast
and
\beqnast
(\ue_r \otimes I)\circ \ud_r (a\# h) & = & (S(h_{(1)})\cdot a) \blacktriangleright (1_A \# h_{(2)}) 
=(h_{(2)}\cdot (S(h_{(1)})\cdot a))\# h_{(3)} \\
& = & (h_{(2)}S(h_{(1)})\cdot a)(h_{(3)}\cdot 1_A) \# h_{(4)} 
 =  (h_{(1)}S(h_{(2)})\cdot a)(h_{(3)}\cdot 1_A) \# h_{(4)} \\
& = & a(h_{(1)}\cdot 1_A) \# h_{(2)} 
 =  a\# h .
\eqnast
Therefore, $\underline{A\# H}$ is an $A$-coalgebra with respect to the module structure \eqref{bimoduleright}.
\end{proof}

Although in the previous Lemma's we found coalgebra structures rather than coring structures on $\Hh$, there are still two different module structures on $\Hh$, so it still makes sense to consider the Takeuchi product. We obtain the following result.

\begin{lemma}\lelabel{Takprod} With notation as in \leref{coringleft} and \leref{coringright}.
\begin{enumerate}
\item The map $\ud_l$ is an algebra map, considered as a map to the left-Takeuchi tensor product 
\[
\underline{A\# H}\times_A \underline{A\# H} =\left\{ \sum x_i \otimes^l_A y_i \in \underline{A\# H}\otimes^l_A \underline{A\# H} |\sum a\blacktriangleright x_i \otimes^l_A y_i = \sum  x_i \otimes^l_A y_i \blacktriangleleft a\,, \; \forall a\in A \right\} .
\]
\item
The map $\ud_r $ is an algebra map considered as a map to the right-Takeuchi tensor product
\[
\underline{A\# H} {}_A \times \underline{A\# H} =\left\{ \sum x_i \otimes^r_A y_i \in \underline{A\# H}\otimes^r_A \underline{A\# H} |\sum a \triangleright x_i \otimes^r_A y_i =\sum x_i \otimes^r_A y_i\triangleleft a \,, \; \forall a\in A \right\} .
\]
\end{enumerate}
\end{lemma}

\begin{proof} We only prove \ul{(1)}, the proof on part \ul{(2)} is left to the reader. First, we verify that the image of $ud_l$ lies in the Takeuchi product. Consider $a\# h \in \underline{A\# H}$ and $b\in A$, then
\beqnast
b \blacktriangleright (a\# h_{(1)})\otimes^l_A (\um \# h_{(2)}) 
& = & (a(h_{(1)}\cdot b)\# h_{(2)})\otimes^l_A (\um \# h_{(3)}) \\
&&\hspace{-3cm} =  (a\# h_{(2)})\triangleleft (h_{(1)}\cdot b)\otimes^l_A (\um \# h_{(3)}) 
 =  (a\# h_{(1)})\triangleleft (h_{(2)}\cdot b)\otimes^l_A (\um \# h_{(3)}) \\
&&\hspace{-3cm} =  (a\# h_{(1)})\otimes^l_A (h_{(2)}\cdot b)\triangleright (\um \# h_{(3)}) 
 =  (a\# h_{(1)})\otimes^l_A ((h_{(2)}\cdot b)\# h_{(3)}) \\
&&\hspace{-3cm} =  (a\# h_{(1)})\otimes^l_A (\um \# h_{(2)})\blacktriangleleft b.
\eqnast
Next, let us verify that the co-restriction $\underline{\Delta}_l:\mathcal{H}\rightarrow \mathcal{H}\times_A \mathcal{H}$ is an algebra morphism. On one hand we have
\beqnast
\underline{\Delta}_l((a\# h)(b\# k)) & = & \underline{\Delta}_l (a(h_{(1)}\cdot b)\# h_{(2)}k)
 =  (a(h_{(1)}\cdot b)\# h_{(2)}k_{(1)})\otimes^l_A (\um \# h_{(3)}k_{(2)});
\eqnast
on the other hand, we have
\beqnast
\underline{\Delta}_l(a\# h)\underline{\Delta}_l(b\# k) & = & [(a\# h_{(1)})\otimes^l_A (\um \# h_{(2)})][(b\# k_{(1)})\otimes^l_A (\um \# k_{(2)})]\\
& = & (a\# h_{(1)})(b\# k_{(1)})\otimes^l_A (\um \# h_{(2)})(\um \# k_{(2)}) \\
& = & (a(h_{(1)}\cdot b)\# h_{(2)}k_{(1)})\otimes^l_A ((h_{(3)} \cdot \um ) \# h_{(4)}k_{(2)})\\
& = & (a(h_{(1)}\cdot b)\# h_{(2)}k_{(1)})\otimes^l_A (h_{(3)} \cdot \um )\triangleright (\um \# h_{(4)}k_{(2)})\\
& = & (a(h_{(1)}\cdot b)\# h_{(2)}k_{(1)}) \triangleleft (h_{(3)} \cdot \um )  \otimes^l_A (\um \# h_{(4)}k_{(2)})\\
& = & (a(h_{(1)}\cdot b)\# h_{(2)}k_{(1)})\otimes^l_A (\um \# h_{(3)}k_{(2)}) .
\eqnast
Therefore, $\ud_l$ is an algebra morphism.
\end{proof}

\begin{lemma}\lelabel{counitprops} With notation as in \leref{coringleft} and \leref{coringright}, the following identities hold for all $x,y\in\underline{A\# H}$.
\begin{enumerate}
\item $\underline{\epsilon}_l(1_{\mathcal{H}}) =1_A = \underline{\epsilon}_r (1_{\mathcal{H}}) $;
\item $\underline{\epsilon}_l(xy)= \ue_l(x \blacktriangleleft\ue_l(y)) =\underline{\epsilon}_l(xs(\underline{\epsilon}_l(y)))$ ;
\item $\underline{\epsilon}_r(xy)=\ue_r(\ue_r(x)\triangleright y) =\underline{\epsilon}_r(s (\underline{\epsilon}_r(x))y)$.
\end{enumerate}
\end{lemma}

\begin{proof} The identity \ul{(1)} is straightforward. For \ul{(2)}, 
take $x=a\# h$ and $y=b\# k$. Then, on one hand, we have
\beqnast
\underline{\epsilon}_l((a\# h)(b\# k)) & = & 
\underline{\epsilon}_l(a(h_{(1)}\cdot b )\#h_{(2)}k ) 
 =  a(h_{(1)} \cdot b )(h_{(2)}k\cdot \um ) \\
& = & a(h_{(1)} \cdot b )(h_{(2)} \cdot \um )(h_{(3)}k\cdot \um ) 
 =  a(h_{(1)} \cdot b )(h_{(2)} \cdot (k\cdot \um )) \\
& = & a(h\cdot (b(k\cdot \um ))) . 
\eqnast
On the other hand,
\beqnast
\underline{\epsilon}_l\big((a\# h) \blacktriangleleft \underline{\epsilon}_l(b\# k)\big) & = & 
\underline{\epsilon}_l ((a\# h)\blacktriangleleft b(k\cdot \um )) 
 = 
\underline{\epsilon}_l( a(h_{(1)}\cdot (b(k\cdot \um )))\# h_{(2)}) \\
& =&  a(h_{(1)}\cdot (b(k\cdot \um )))(h_{(2)} \cdot \um ) 
 =  a(h\cdot (b(k\cdot \um ))) . 
\eqnast
For \ul{(3)} take again $x=a\# h$ and $y=b\# k$. Then,
\beqnast
\underline{\epsilon}_r((a\# h)(b\# k)) & = & 
\underline{\epsilon}_r(a(h_{(1)}\cdot b )\#h_{(2)}k ) 
 =  S(h_{(2)}k)\cdot (a(h_{(1)}\cdot b )) \\
&&\hspace{-3cm} =  (S(h_{(3)}k_{(2)})\cdot a )(S(h_{(2)}k_{(1)})\cdot (h_{(1)}\cdot b )) 
 =  (S(h_{(3)}k_{(2)})\cdot a )(S(k_{(1)})S(h_{(1)})h_{(2)}\cdot b ) \\
&&\hspace{-3cm} =  (S(hk_{(2)})\cdot a )(S(k_{(1)})\cdot b ) 
 =  (S(k_{(2)})S(h)\cdot a )(S(k_{(1)})\cdot b ) \\
&&\hspace{-3cm} =  S(k)\cdot ((S(h)\cdot a ) b) 
 =  \ue_r ((S(h)\cdot a ) b\# k) \\
&&\hspace{-3cm} =  \ue_r ((S(h)\cdot a )\triangleright ( b\# k)) 
 =  \ue_r (\ue_r  (a\# h)\triangleright ( b\# k)). 
\eqnast
\end{proof}

Combining the results of Lemma's \ref{le:coringleft}, \ref{le:coringright}, \ref{le:Takprod} and \ref{le:counitprops} so far we have obtained that $\mathcal{H}=\underline{A\# H}$ has the structures of left and right $A$-bialgebroid. In order to prove that it is a Hopf algebroid we need to define the antipode map
\[
\mathcal{S}:\Hh\to\Hh,\quad {\mathcal S}(a\# h)=(S(h_{(2)})\cdot a)\# S(h_{(1)}), \forall a\# h\in \Hh .
\]

\begin{thm}\thlabel{Hoidpartialaction} 
Using notation as in Lemma's \ref{le:coringleft} and \ref{le:coringright}, the data 
$$(\underline{A\# H}, A, s_l,t_l, s_r ,t_r \ud_l ,\ue_l , \ud_r ,\ue_r ,\mathcal{S})$$ 
define a structure of a Hopf algebroid over the base algebra $A$.
\end{thm}

\begin{proof} 
As $(\underline{A\# H} , A, s_l ,t_l , \ud_l ,\ue_l )$ is a left $A$ bialgebroid, and $(\underline{A\# H} , A, s_r ,t_r , \ud_r ,\ue_r )$  is a right $A$ bialgebroid, we need only to verify that $\mathcal{S}$ is an anti-algebra morphism and that the Hopf algebroid axioms (i)-(iv) recalled on page \pageref{defHopfalgebroid} are satisfied. We restrict ourselves here to axioms (iii) and (iv).

For the Hopf algebroid axiom (iii), take $a,c\in A$ and $b\# h\in \underline{A\# H}$, then 
\beqnast
\mathcal{S}(t_l (a)(b\# h)t_r (c)) & = & \mathcal{S} ((a\# 1_H)(b\# h)(c\# 1_H))
 =  \mathcal{S} (ab(h_{(1)}\cdot c)\# h_{(2)})\\
&&\hspace{-3cm} =  (S(h_{(3)})\cdot (ab(h_{(1)}\cdot c)))\# S(h_{(2)})
 =  (S(h_{(5)})\cdot a)(S(h_{(4)})\cdot b)(S(h_{(3)})\cdot 
(h_{(1)}\cdot c))\# S(h_{(2)})\\
&&\hspace{-3cm} =  (S(h_{(5)})\cdot a)(S(h_{(4)})\cdot b)(S(h_{(2)})h_{(3)}\cdot c)\# S(h_{(1)})
 =  (S(h_{(3)})\cdot a)(S(h_{(2)})\cdot b) c\# S(h_{(1)})\\
&&\hspace{-3cm} =  (c\# 1_H) ((S(h_{(3)})\cdot b)(S(h_{(2)})\cdot a) \# S(h_{(1)}))
 =  (c\# 1_H) ((S(h_{(2)})\cdot b)\# S(h_{(1)}))(a\# 1_H )\\
&&\hspace{-3cm} =  s_r( c)\mathcal{S}(b\# h) s_l (a) .
\eqnast
Finally, for the Hopf algebroid axiom (iv), take $a\# h \in \underline{A\# H}$, then
\beqnast
\mu \circ (\mathcal{S} \otimes I)\circ \ud_l (a\# h) & = & \mathcal{S} (a\# h_{(1)})
(1_A \# h_{(2)}) 
 =  ((S(h_{(2)})\cdot a)\# S(h_{(1)}))(1_A \# h_{(3)})\\
&&\hspace{-3cm} =  (S(h_{(3)})\cdot a)(S(h_{(2)})\cdot 1_A ) \# S(h_{(1)})h_{(4)}
 =  (S(h_{(3)})\cdot a)\# S(h_{(1)})h_{(2)}\\
&&\hspace{-3cm} =  (S(h)\cdot a) \# 1_H 
 =  s_r \circ \ue_r (a\# h) ,
\eqnast
and 
\beqnast
\mu \circ (I\otimes S)\circ \ud_r (a\# h) & = & (a\# h_{(1)})\mathcal{S}(1_A \# h_{(2)}) 
 =  (a\# h_{(1)})((S(h_{(3)})\cdot 1_A )\# S(h_{(2)})) \\
&&\hspace{-3cm} =  a(h_{(1)}\cdot (S(h_{(4)})\cdot 1_A ))\# h_{(2)} S(h_{(3)}) 
 =  a(h_{(1)}\cdot (S(h_{(2)})\cdot 1_A ))\# 1_H \\
&&\hspace{-3cm} =  a(h_{(1)}\cdot 1_A)(h_{(2)}S(h_{(3)})\cdot 1_A )\# 1_H 
 =  a(h\cdot 1_A)\# 1_H \\
&&\hspace{-3cm} =  s_l \circ \ue_l (a\# h).
\eqnast

Therefore, $\underline{A\# H}$ is a Hopf algebroid.
\end{proof}

\section{Partial comodule algebras}\selabel{comodulealgebras}

\subsection{Symmetric partial comodule algebras}

\subsubsection{Definitions  and examples}

In \cite{caen06} the notion of a partial comodule algebra was introduced. We recall this notion here and complete it with a symmetry axiom.

\begin{defi} Let $H$ be a Hopf algebra. A unital algebra $A$ is said to be a right partial $H$-comodule algebra, or is said to possess a partial coaction of $H$, if there exists a linear map
\[
\begin{array}{rccl} \prho : & A & \rightarrow & A\otimes H \\
                    \,     & a & \mapsto     & \prho(a)= a^{[0]}\otimes a^{[1]}
\end{array}\]
that satisfies the following identities:
\begin{enumerate}
\item[(PRHCA1)] for every $a,b \in A$, $\prho (ab)=\prho (a) \prho (b)$;
\item[(PRHCA2)] for every $a\in A$, $(I\otimes \epsilon)\prho (a) =a$;
\item[(PRHCA3)] for every $a\in A$, $(\prho \otimes I)\prho (a)=[(I\otimes \Delta)\prho (a)](\prho (\um )\otimes1_H)$. 
\end{enumerate} 
A partial coaction is moreover called {\em symmetric} if 
\begin{enumerate}
\item[(PRHCA4)] for every $a\in A$, $(\prho \otimes I)\rho (a)=(\prho (\um )\otimes1_H)[(I\otimes \Delta)\prho (a)]$.
\end{enumerate}
Let $A$ and $B$ be two right partial $H$ comodule algebras. We say that $f:A\to B$ is a morphism of partial comodule algebras if it is an algebra morphism such that $\prho_B \circ f = (f\otimes I)\circ \prho_A$. 
The category of right partial $H$ comodule algebras and their morphisms is denoted by $\ParCoAct_H$.
\end{defi}
Denoting the partial coaction in a Sweedler-type notation,
\[
\prho (a) = a^{[0]}\otimes a^{[1]} ,
\]
we can rewrite the axioms for partial coactions in the following manner:
\begin{enumerate}
\item[(PRHCA1)] $(ab)^{[0]}\otimes (ab)^{[1]} =a^{[0]}b^{[0]} \otimes a^{[1]}b^{[1]}$;
\item[(PRHCA2)] $a^{[0]}\epsilon(a^{[1]})=a$;
\item[(PRHCA3)] $ a^{[0][0]}\otimes a^{[0][1]}\otimes a^{[1]} =a^{[0]}\um^{[0]}\otimes {a^{[1]}}_{(1)}\um^{[1]}\otimes {a^{[1]}}_{(2)}$;
\item[(PRHCA4)] $ a^{[0][0]}\otimes a^{[0][1]}\otimes a^{[1]} =\um^{[0]}a^{[0]}\otimes \um^{[1]}{a^{[1]}}_{(1)}\otimes {a^{[1]}}_{(2)}$.
\end{enumerate}

Symmetrically, one can define also the notion of a {\em left} partial $H$-comodule algebras, but throughout this text, we deal basically with right partial comodule algebras. It is important to note that, as the partial coaction $\prho :A\rightarrow A\otimes H$ is a morphism of algebras, $\prho (1_A )$ is an idempotent in the algebra $A\otimes H$, and for any $a\in A$ we have $\prho (a)=\prho (a) \prho (1_A )=\prho(1_A)\prho(a)$. Remark however, that $\prho(1_A)$ is only central in the image of $\prho$, and not in the whole of $A\ot H$. We obtain that the image of the coaction is contained in the unitary ideal $A\underline{\otimes} H=(A\otimes H)\prho (1_A )$ and the projection $\pi :A\otimes H \rightarrow A\underline{\otimes} H$ is given simply by the multiplication by $\prho (1_A )=1^{[0]} \otimes 1^{[1]}$. A typical element in $A\underline{\otimes} H$ can be written as
\[
x= \sum_i a^i 1^{[0]} \otimes h^i 1^{[1]} , \quad \mbox{ for  } \, a^i \in A , \mbox{  and  } \, h^i \in H .
\]

Let us recall the following basic example from \cite{AB}.
\begin{example}
Let $H$ be a Hopf algebra and $B$ a right $H$-comodule algebra with coaction $\rho:B\to B\ot H$. Suppose that $A\subset B$ is a unital ideal in $B$, then $A$ is right partial $H$-comodule algebra, with coaction $\prho:A\to A\ot H,\ \prho(a)=(1_A\ot 1_H)\rho(a)$.
\end{example}

\subsubsection{Duality between partial actions and partial coactions}

There is a natural duality between partial actions and partial coactions, as was observed in \cite{AB} and \cite{caen06}, which moreover holds in the symmetric case, as we will point out now.

\begin{prop}\prlabel{coactiontoaction}
Consider a dual pairing of Hopf algebras $\bk{-,-}:H\ot K\to k$ and let $A$ be a symmetric right partial $K$-comodule algebra.
Then the map
\[
\begin{array}{rccc}
\cdot : & H \otimes A & \rightarrow & A\\
\, & f\otimes a & \mapsto & f\cdot a=\sum a^{[0]} 
\langle f,a^{[1]} \rangle
\end{array}
\]
is a symmetric left partial action of $H$ on $A$. This construction yields a functor
$$\Phi:\ParCoAct_K\to {_H\ParAct}.$$
\end{prop}

\begin{proof} Let us just check the duality between the symmetry of the coaction and the action. For the remaining part, we refer to \cite{AB}  where the statement was proven in case $k$ is a field, but the proof generalizes without any problem to any commutative base ring $k$. Consider $a\in A$ and $h,k\in H$, then 
\beqnast
h\cdot (k\cdot a) & = & h\cdot (a^{[0]} \langle k ,a^{[1]} \rangle =a^{[0][0]} \langle h, a^{[0][1]} \rangle \langle k ,a^{[1]} \rangle \\
& = & a^{[0]} 1^{[0]} \langle h, {a^{[1]}}_{(1)} 1^{[1]} \rangle \langle k ,{a^{[1]}}_{(2)} \rangle \\
& = & a^{[0]} 1^{[0]} \langle h_{(1)}, {a^{[1]}}_{(1)} \rangle \langle h_{(2)} , 1^{[1]} \rangle  \langle k ,{a^{[1]}}_{(2)} \rangle \\
& = & a^{[0]} 1^{[0]} \langle h_{(1)}k, a^{[1]} \rangle \langle h_{(2)} , 1^{[1]} \rangle  \\ 
& = & (h_{(1)}k \cdot a)(h_{(2)} \cdot \um ),
\eqnast
where we used (PRHCA3) in the third equality. On the other hand, if one applies (PRHCA4) at the thrid equality in stead and continues by a similar calculation, then one obtains
\[
h\cdot (k\cdot a) =(h_{(1)}\cdot \um ) (h_{(2)}k\cdot a).
\]
Therefore, the partial action of $H$ on $A$ is symmetric.
\end{proof}

Let us use the same notation as in the statement of \prref{coactiontoaction}. Take any $a\in A$ and write $a^{[0]}\ot a^{[1]}=\sum_{i=1}^n a_i\ot x_i$ for certain $a_i\in A$ and $x_i\in K$. Then for all $h\in H$, we find that $a\cdot h=\sum_{i=1}^na_i\bk{h_i,x_i} $. This leads us to the following definition.

\begin{defi}
Let $H$ be a Hopf algebra and $A$ be a left partial $H$-module algebra. We say that the partial action of $H$ on $A$ is {\em rational} if for every $a\in A$ there exists $n=n(a) \in \mathbb{N}$ and a finite set $\{ a_i , \varphi^i \}_{i=1}^n$, with $a_i \in A$ and $\varphi^i \in H^*$ such that, for every $h\in H$ we have 
\[
h\cdot a =\sum_{i=1}^n \varphi^i (h) a_i .
\] 
The full subcategory of ${_H\ParAct}$ consisting of all rational left partial actions is denoted by ${_H\ParAct^r}$.
\end{defi}

Recall that a $k$-module $M$ is said to be locally projective (in the sense of Zimmermann-Huisgen) if for all $m\in M$, there exists a finite dual base $\{(e_i,f_i)\}_{i=1,\ldots,n}\in M\ot M^*$ i.e.\ $m=\sum_{i=1}^ne_if_i(m)$. 
Moreover the fact that $M$ is locally projective over $k$ and a submodule $D\subset M^*$ is dense with respect to the finite topology is equivalent with $M$ satisfying the $D$-relative $\alpha$-condition, which states that for any $k$-module $N$ the canonical map
$$\alpha_{N,D}:N\ot M\to {\Hom}(D,N),\quad \alpha_{N,D}(n\ot m)(d)=d(m)n$$
is injective, for details see e.g. \cite{Ver:locunits}.

\begin{teo}  \thlabel{actiontocoaction}
Consider a non-degenerate dual pairing of Hopf algebras $\bk{-,-}:H\ot K\to k$ where $K$ is locally projective over $k$ and let $A$ be a rational symmetric left partial $H$-module algebra. 
Then $A$ can be endowed with the structure of a symmetric right partial $K$-comodule algebra such that $\Phi(A)$ is the initial left partial $H$-module algebra. 
This construction, to together with \prref{coactiontoaction}, yields an isomorphism of categories
$${_H\ParAct^r}\cong \ParCoAct_K.$$
\end{teo}

\begin{proof} 
Consider the following diagram.
\[
\xymatrix{
A \ar@{.>}[rr] \ar[drr]_-\beta  && A\ot K \ar[d]^-{\alpha_{A,H}}\\
&& \Hom(H,A)
}
\]
where we define $\beta(a)(h)=h\cdot a$. Then, since $A$ is rational and $K$ is dense in $H^*$ by the non-degeneracy of the pairing, the image of $\beta$ is contained in the image of $\alpha_{A,H}$. Moreover, since $K$ is locally projective over $k$ and $H$ is dense in $K^*$, the map $\alpha_{A,H}$ is injective. Hence there exists a well-defined map $\prho:A\to A\ot K$ that renders the diagram commutative, i.e. 
\[
h\cdot a = (I\otimes \langle h, \underline{\; } \rangle )\prho (a).
\]
Similarly to \cite{AB}, where the proof was made considering $k$ a field, one can now show the stated isomorphism. Let us just check the duality between the symmetry of the coaction and the action.

We have to verify the axiom (PRHCA3), to this, consider $h,k\in H$ then
\beqnast
& \, & (I\otimes \langle h, \underline{\; }\rangle \otimes \langle k , \underline{\; } \rangle ) ((\prho \otimes I)\prho (a)) = (I\otimes \langle h, \underline{\; }\rangle ) \sum_{i=1}^n \prho (a_i )\langle k, x^i \rangle =(I\otimes \langle h, \underline{\; }\rangle ) \prho (k\cdot a) \\
& = & h\cdot (k\cdot a) =(h_{(1)}k \cdot a)(h_{(2)} \cdot \um ) = [(I\otimes \langle h_{(1)}k, \underline{\; }\rangle )\prho (a)][(I\otimes \langle h_{(2)}, \underline{\; }\rangle )\prho (\um )] \\
& = & [(I\otimes \langle h_{(1)}, \underline{\; } \rangle  \otimes \langle k, \underline{\; }\rangle ) (I\otimes \Delta ) \prho (a)][((I\otimes \langle h_{(2)}, \underline{\; }\rangle )\prho (\um )) \otimes 1_H ] \\
& = & (I\otimes \langle h, \underline{\; }\rangle \otimes \langle k , \underline{\; } \rangle ) (((I\otimes \Delta )\prho (a))(\prho (\um )\otimes 1_H )) .
\eqnast
By the nondegeneracy of the pairing, one can conclude that 
\[
(\prho \otimes I)\prho (a)=((I\otimes \Delta )\prho (a))(\prho (\um )\otimes 1_H ) .
\]
\end{proof}

The following result can implicitly be found in \cite{caen06}, hence we omit an explicit proof and give only the structure maps. 

\begin{lemma} \lelabel{splitcoring}
Let $K$ ba a Hopf algebra and $A$ be a right partial $K$ comodule algebra, then the reduced tensor product $A\underline{\otimes} K =(A\otimes K)\prho (1_A )$ has a structure of an $A$-coring with bimodule structure, comultiplication and counit given by
\begin{eqnarray*}
b\cdot (a1^{[0]}\ot x1^{[0]})\cdot b'&=& bab'^{[0]}\ot x b'^{[1]};\\
\widetilde{\Delta} (a1^{[0]} \otimes x1^{[1]} ) & = & a 1^{[0]} \otimes x_{(1)} 1^{[1]} \otimes_A 1^{[0']} \otimes x_{(2)}1^{[1']};\\
\widetilde{\epsilon} (a1^{[0]} \otimes x1^{[1]} ) & = & a\epsilon (x).
\end{eqnarray*}
Moreover, if there is a dual pairing $\bk{-,-}:H\ot K\to k$ between $K$ and a second Hopf algebra $H$, then there is an algebra isomorphism between the left dual ring ${}^* (A\underline{\otimes} K)$ and the smash product $(A^{op}\# H^{cop})^{op}$.
\end{lemma}

\subsection{Partial coactions of commutative Hopf algebras}

In this section we restrict to partial actions of commutative Hopf algebras on commutative algebras. Remark that in this case, all partial actions are automatically symmetric.

\subsubsection{Partial coactions from algebraic geometry}\selabel{geometry}

In this section we extend the well-known correspondence between actions of affine algebraic groups on affine algebraic sets and comodule algebras over affine Hopf algebras to the partial setting. 
Throughout this subsection,
$k$ will denote an algebraically closed field and $\mathbb{A}^n$ the $n$-dimensional  affine space over $k$. By an affine algebraic set we mean a subset of point in $\mathbb{A}^n$ which are zeros of a finite set of polynomials $p_1,\ldots,p_k$ in $k[x_1 ,\ldots ,x_n ]$ (i.e. these are not-necessarily irreducible affine algebraic varieties).

\begin{defi} Let $G$ be an affine algebraic group and $M$ an affine algebraic set. A partial action $(\{ M_g \}_{g\in G} , \{ \alpha_g \}_{g\in G} )$ of $G$ on the underlying set $M$ is said to be {\em algebraic} if 
\begin{enumerate}
\item for all $g\in G$, $M_g$ and its complement $M'_g=M\setminus M_g$ are affine algebraic sets;
\item for all $g\in G$, the maps $\alpha_g :M_{g^{-1}}\rightarrow M_g$ are polynomial;
\item the set of ``compatible couples'' $G\bullet M:=\{(g,x)\in G\times M ~|~x\in M_{g^{-1}}\}\subset G\times M$ is an algebraic set and the map $\alpha:G\bullet M\to M, \alpha(g,x)=\alpha_g(x)$ is polynomial.
\end{enumerate}
\end{defi}

\begin{remark} If a partial action of an affine algebraic group $G$ on an affine algebraic set $M$ is algebraic, then each domain $M_g$ is a disjoint union of a finite number of connected components.
\end{remark}

\begin{remark} In what follows, we write the set of "compatible couples" $G\bullet M$ in an equivalent way as
\[
G\bullet M =\{ (x,g)\in M\times G \; | \; x\in M_g \} .
\]
This is equivalent because one can view the set of "compatible couples" as the set of triples 
\[
A= \{ (\alpha_g (x) , g, x)\in M\times G \times M \; | \; x\in M_{g^{-1}} \},
\]
in which there is an obvious redundancy. As we shall see later, this set has a structure of a groupoid \cite{KL}.
\end{remark}

\begin{exmp} Take the algebraic set $M$ which is the union of two horizontal circles of radius 1, one centered at $(0,0,1)$ and the other at $(0,0,0)$. This is an affine algebraic set, whose algebra of coordinate functions is given by
\[
A=k[x,y,z]/\langle x^2+y^2-1 , z^2 -z \rangle
\]

There is a partial action of the affine algebraic group $G= \mathbb{S}^1 \rtimes \mathbb{Z}_2$. Geometrically, the group $G$ can be thought as the union of two disjoint circles in the three dimensional affine space: The circle $G_1$, whose elements are of the form $g=(x_1 ,x_2 ,1 )$ and $G_2$ whose elements are of the form $g=(x_1 ,x_2 ,-1)$ and the group operation is given by
\[
(x_1, x_2 , \lambda )(y_1 , y_2 , \mu) =(x_1 y_1 -\lambda x_2 y_2 , y_1 x_2 +\lambda x_1 y_2, \lambda \mu ) ,
\]
where $\lambda$ and $\mu$ are equal to $+1$ or $-1$. The  algebra of coordinate functions for the group $G$ can be written as
\[
H=k[x_1 ,x_2 ,x_3 ]/\langle x_1^2 +x_2^2 -1 , x_3^2 -1 \rangle.
\]
 For $g\in G_1$, we have $M_g =M$ and the action is given by 
\[
\alpha_{(x_1 , x_2  , 1)} (x,y,z) =(xx_1 -yx_2 , xx_2 +yx_1 , z) \qquad \mbox{ for  }\, z=0,1 .
\]
For $g\in G_2$, the domain $M_g$ is only the unit circle centered at $(0,0,0)$, and the action is given by
\[
\alpha_{(x_1 , x_2 , -1)} (x,y,z) =(-xx_1 -yx_2 , -xx_2 +yx_1 , -z) \qquad 
\mbox{ for  }\, z=0.
\]
This partial action is clearly algebraic.
\end{exmp}

Algebraic partial group actions give rise to partial coactions of commutative Hopf algebras. Let us first state the following useful lemma.

\begin{lemma}\lelabel{decomp}
Let $M$ be an affine algebraic set and $A=\Oo(M)$ its coordinate algebra. Then $M$ is the disjoint union of two algebraic subsets, $M=N\sqcup N'$ if and only if the unit of $A$ can be decomposed in a sum of two orthogonal idempotents $1=e+f$. In this case $\Oo(N)=e\Oo(M)$ is a unital ideal in $\Oo(M)$ and $e$ can be viewed as the characteristic function on $N$. 
\end{lemma}

\begin{proof}
Suppose that $M$ is an algebraic subset of the affine space $\mathbb{A}^n$ and that $M=N\sqcup N'$. Then the coordinate algebras $A=\mathcal{O}(M)$, $B=\mathcal{O}(N)$ and  $B'=\mathcal{O}(N')$ are quotients of the polynomial algebra $k[x_1 ,\ldots ,x_n ]$. Denote the respective ideals by $I$, $J$ and $J'$. Since $N\cap N'=\emptyset$, we have that $J+J'=k[x_1,\ldots,x_n]$, i.e. $J$ and $J'$ are comaximal. Take $u \in J$ and $v \in J'$ such that $u +v =1$. Since also $M=N\cup N'$, we have that $I=J\cap J'=JJ'$ (the last equality follows by the comaximality).
Denote by $e$ and $f$, respectively the classes of $v$ and $u$ modulo $JJ'$. Then one can easily check that $e$ and $f$ are orthogonal idempotents in $\Oo(M)$. The converse is easier.
\end{proof}

\begin{prop} Let $G$ be an affine algebraic group and $M$ be an affine algebraic set. Then each algebraic partial action of $G$ on $M$ defines a (symmetric) partial coaction of the commutative Hopf algebra $H=\mathcal{O}(G)$ on the commutative algebra $A=\mathcal{O}(M)$.
\end{prop}

\begin{proof} 
Denote $G\circ M=\{(x,g)\in M\times G~|~x\in M'_{g}\}$. Then we clearly have that the algebraic set $M\times G$ can be written as a disjoint union $M\times G=(G\bullet M)\sqcup (G\circ M)$. By \leref{decomp}, we then now that there exists an idempotent $e\in \Oo(M\times G)\cong \Oo(M)\ot \Oo(G)$ such that $\Oo(G\bullet M)=e\Oo(M\times G)$. The algebraic partial action provides us with a polynomial map $\alpha:G\bullet M\to M$, which induces an algebra morphism $\rho:\Oo(M)\to \Oo(G\bullet M)$. We now define the partial coaction of the Hopf algebra $H=\Oo(G)$ on the algebra $A=\Oo(M)$ as the morphism
\[
\ol\rho:\Oo(M)\to \Oo(M)\ot \Oo(G),\ \ol\rho(f)=e\rho(f) ,
\]
which will be denoted in a Sweedler like notation as $\overline{\rho} (f)=f^{[0]}\ot f^{[1]}$.
As $e$ can be understood as the characteristic function on $G\bullet M$, $\ol\rho$ can be rewritten as
\begin{equation}
\label{coactionfromalgebraicaction}
\prho (f) (p,g) =\left\{ \begin{array}{lcr} f(\alpha_{g^{-1}}(p)) & \mbox{ if } & p\in M_g \\
0 & \,  & \mbox{otherwise} 
\end{array} \right.
\end{equation}
It is easy to see that $\prho$ above defined is multiplicative, i.e.\ it satisfies (PRHCA1). Now, consider $f\in \mathcal{O}(M)$ and $p\in M$ then
\[
\big( (I\otimes \epsilon )\prho (f)\big) (p)  =  \prho (f) (p,e_G) =f(\alpha_{e_G} (p))=f(p), 
\]
since  $p\in M_{e_G} =M$.
Therefore, $\prho$ satisfies (PRHCA2). Finally, for $f\in \mathcal{O}(M)$, $p\in M$ and $g,h\in G$,
\beqnast
(\prho \otimes I)\prho (f)(p,g,h) & = & (\prho \otimes I)(f^{[0]}\otimes f^{[1]})(p,g,h) \\
& = & \left\{ \begin{array}{lcl} f^{[0]}(\alpha_{g^{-1}}(p))f^{[1]}(h) && \mbox{if }    p\in M_g \\
0 & \, & \mbox{otherwise} \end{array} \right. \\
& = & \left\{ \begin{array}{lcl} f(\alpha_{h^{-1}} (\alpha_{g^{-1}}(p))) && \mbox{if }    p\in M_g , \mbox{ and } \alpha_{g^{-1}}(p))\in M_h \\
0 & \, & \mbox{otherwise} \end{array} \right. \\
& = & \left\{ \begin{array}{lcl} f(\alpha_{h^{-1}} (\alpha_{g^{-1}}(p))) && \mbox{if }    p\in \alpha_g (M_{g^{-1}} \cap M_h ) \\
0 & \, & \mbox{otherwise} \end{array} \right. \\
& = & \left\{ \begin{array}{lcl} f(\alpha_{(gh)^{-1}} (p)) && \mbox{if }    p\in M_g \cap M_{gh}  \\
0 & \, & \mbox{otherwise} \end{array} \right. 
\eqnast
On the other hand,
\beqnast
[(I\otimes \Delta )\prho (f)(\prho (1_{\mathcal{O}(M)})\otimes 1_{\mathcal{O}(G)})](p,g,h)  & = &  (I\otimes \Delta )\prho (f)(p,g,h) (\prho (1_{\mathcal{O}(M)})\otimes 1_{\mathcal{O}(G)})(p,g,h) \\
& = & \prho (f) (p,gh)  \prho (1_{\mathcal{O}(M)} )(p,g) \\
& = & \left\{ \begin{array}{lcl} \prho (f) (p,gh)  && \mbox{if }    p\in M_g  \\
0 & \, & \mbox{otherwise} \end{array} \right. \\
& = & \left\{ \begin{array}{lcl}  f(\alpha_{(gh)^{-1}} (p)) && \mbox{if }    p\in M_g \cap M_{gh} \\
0 & \, & \mbox{otherwise} \end{array} \right. 
\eqnast
Therefore, $\prho$ satisfies (PRHCA3), which proves that this is a partial coaction of the commutative Hopf algebra $H=\mathcal{O}(G)$ over the commutative algebra $A=\mathcal{O}(M)$. 
\end{proof}

Conversely, given a partial coaction of a commutative Hopf algebra on a commutative algebra, one can construct a partial action of the affine algebraic group $\Hom_{\underline{Alg}} (H,k)$ on the algebraic set $\mbox{Hom}_{\underline{Alg}} (A,k)$. 

\begin{prop} Let $H$ be a commutative Hopf algebra and $A$ be a commutative right partial $H$ comodule algebra. Then there is a (symmetric) partial action of the affine algebraic group $G=\Hom_{\underline{Alg}} (H,k)$ on the affine algebraic set  $M=\Hom_{\underline{Alg}} (A,k)$.
\end{prop}

\begin{proof} Denoting $\prho (1_A )=1^{[0]}\otimes 1^{[1]}$, define, for each $g\in G$, the element
\[
1_g =1^{[0]}g(1^{[1]}) \in A .
\]
It is easy to see that $1_g$ is an idempotent,
\[
1_g 1_g =  1^{[0]}1^{[0']} g(1^{[1]})g(1^{[1']}) =1^{[0]}1^{[0']} g(1^{[1]}1^{[1']}) =1^{[0]}g(1^{[1]})=1_g .
\]
Define, for each $g\in G$, the ideals $A_g =1_gA$. A typical element in $A_g$ is of the form $a1^{[0]}g(1^{[1]})$, for $a\in A$. Note that, since $\prho (a) =\prho (a1_A)=\prho(a)\prho(1_A)$, the elements of the form $a^{[0]}g(a^{[1]})$ are also in the ideal $A_g$. Define also the linear maps $\theta_g :A_{g^{-1}}\rightarrow A$, by 
$\theta_g =(I\otimes g(\underline{\; }))\circ \prho \vert_{A_{g^{-1}}}$. The map $\theta_g$ is an algebra isomorphism between $A_{g^{-1}}$ and $A_g$. Indeed, take $a\in A_{g^{-1}}$, then we have $a=a1^{[0]}g^{-1}(1^{[1]})$ and we find
\beqnast
\theta_g (a) & = & (I\otimes g)\circ \prho (a1^{[0]}g^{-1}(1^{[1]})) =
a^{[0]}1^{[0][0]} g(a^{[1]}1^{[0][1]})g^{-1}(1^{[1]}) \\
& = & a^{[0]}1^{[0]}1^{[0']}g(a^{[1]}{1^{[1]}}_{(1)} 1^{[1']})g^{-1} ({1^{[1]}}_{(2)})
=a^{[0]}g(a^{[1]})1^{[0]}g({1^{[1]}}_{(1)})g^{-1}({1^{[1]}}_{(2)})\\
& = & a^{[0]}g(a^{[1]})1^{[0]}\epsilon (1^{[1]}) 
 =  a^{[0]}g(a^{[1]}) \in A_g .
\eqnast
Also, we have, for $a\in A_{g^{-1}}$,
\beqnast
\theta_{g^{-1}} \circ \theta_g (a) & = & \theta_{g^{-1}} (a^{[0]}g(a^{[1]}))
 =  a^{[0][0]}g^{-1} (a^{[0][1]}) g(a^{[1]}) \\
& = & a^{[0]} 1^{[0]} g^{-1} ({a^{[1]}}_{(1)}1^{[1]}) g({a^{[1]}}_{(2)})
 =  a^{[0]} g^{-1}({a^{[1]}}_{(1)}) g({a^{[1]}}_{(2)})1^{[0]}g^{-1}(1^{[1]})\\
& = & a^{[0]}\epsilon (a^{[1]}) 1^{[0]}g^{-1}(1^{[1]})  
 =  a 1^{[0]}g^{-1}(1^{[1]}) =a .
\eqnast
Then $\theta_{g^{-1}} \circ \theta_g =\mbox{Id}_{A_{g^{-1}}}$. Analogously, we have $ \theta_g \circ \theta_{g^{-1}}=\mbox{Id}_{A_{g}}$. Then $\theta_g$ is bijective. Finally, for $a,b\in A_{g^{-1}}$,
\[
\theta_g (ab)  =  (ab)^{[0]} g((ab)^{[1]})=a^{[0]}b^{[0]}g(a^{[1]}b^{[1]}) =a^{[0]}g(a^{[1]})b^{[0]}g(b^{[1]}) =\theta_g (a) \theta_g (b) .
\]
Therefore, $\theta_g$ is an algebra isomorphism.

The data $(\{ A_g \}_{g\in G} , \{ \theta_g \}_{g\in G} )$ defines a partial action of the group $G$ on the algebra $A$. Indeed, it is easy to see that $A_e =A$ and $\theta_e =\mbox{Id}_A$. Take now an element $y \in A_h \cap A_{g^{-1}}$, then 
\beqnast
\theta_{h^{-1}} (y) & = & \theta_{h^{-1}} (y1^{[0]}1^{[0']}h(1^{[1]}) g^{-1} (1^{[1']}))\\
& = & y^{[0]}1^{[0][0]}1^{[0'][0]}h^{-1} (y^{[1]}1^{[0][1]}1^{[0'][1]}) h(1^{[1]}) g^{-1} (1^{[1']})\\
& = & y^{[0]}1^{[0]}1^{[0']}1^{[0'']}
h^{-1} (y^{[1]}{1^{[1]}}_{(1)}{1^{[1']}}_{(1)}1^{[1'']}) h({1^{[1]}}_{(2)}) 
g^{-1} ({1^{[1']}}_{(2)})\\
& = & y^{[0]}h^{-1} (y^{[1]})1^{[0]}
h^{-1} ({1^{[1]}}_{(1)}) h({1^{[1]}}_{(2)}) 1^{[0']}h^{-1}({1^{[1']}}_{(1)}) 
g^{-1} ({1^{[1']}}_{(2)})\\
& = & y^{[0]}h^{-1} (y^{[1]})1^{[0]} \epsilon (1^{[1]}) 1^{[0']}(gh)^{-1}(1^{[1']}) \\
& = & y^{[0]}h^{-1} (y^{[1]}) 1^{[0]}(gh)^{-1}(1^{[1]}) \in A_{h^{-1}}\cap A_{(gh)^{-1}}
\eqnast
Then $\theta_{h^{-1}}(A_h \cap A_{g^{-1}})\subseteq A_{h^{-1}}\cap A_{(gh)^{-1}}$. Finally, for any $x\in \theta_{h^{-1}}(A_h \cap A_{g^{-1}})$ we have
\beqnast
\theta_g \theta_h (x) & = & \theta_g (x^{[0]}h(x^{[1]}) 1^{[0]} g^{-1}( 1^{[1]})) \\
& = & x^{[0][0]}1^{[0][0]} g(x^{[0][1]}1^{[0][1]})h(x^{[1]})g^{-1}( 1^{[1]})\\
& = & x^{[0]}1^{[0]}1^{[0']} g({x^{[1]}}_{(1)}{1^{[1]}}_{(1)}1^{[1']})h({x^{[1]}}_{(2)})g^{-1}( {1^{[1]}}_{(2)})\\
& = & x^{[0]}g({x^{[1]}}_{(1)}) h({x^{[1]}}_{(2)}) 
1^{[0]}g({1^{[1]}}_{(1)})g^{-1}( {1^{[1]}}_{(2)}) 1^{[0']} g(1^{[1']}) \\
& = & x^{[0]}gh({x^{[1]}}) 1^{[0]}\epsilon ({1^{[1]}}) 1^{[0']} g(1^{[1']}) \\
& = & x^{[0]}gh({x^{[1]}}) 1^{[0]} g(1^{[1]}) ,
\eqnast
while, on the other hand, we have
\beqnast
\theta_{gh} (x) & = & \theta_{gh} (x1^{[0]}1^{[0']}h^{-1}(1^{[1]}) (gh)^{-1} (1^{[1']}))\\
& = & x^{[0]}1^{[0][0]}1^{[0'][0]}gh (x^{[1]}1^{[0][1]}1^{[0'][1]}) h^{-1}(1^{[1]}) (gh)^{-1} (1^{[1']})\\
& = & x^{[0]}1^{[0]}1^{[0']}1^{[0'']}
gh (x^{[1]}{1^{[1]}}_{(1)}{1^{[1']}}_{(1)}1^{[1'']}) h^{-1}({1^{[1]}}_{(2)}) 
(gh)^{-1} ({1^{[1']}}_{(2)})\\
& = & x^{[0]}gh (x^{[1]})1^{[0]} gh ({1^{[1]}}_{(1)}) h^{-1}({1^{[1]}}_{(2)}) 1^{[0']}
gh({1^{[1']}}_{(1)}) (gh)^{-1} ({1^{[1']}}_{(2)}) 1^{[0'']}gh(1^{[1'']}) \\
& = & x^{[0]}gh (x^{[1]})1^{[0]} g ({1^{[1]}})  1^{[0']} gh(1^{[1']}) \\
& = & x^{[0]}gh({x^{[1]}}) 1^{[0]} g(1^{[1]}) .
\eqnast
Then, $\theta_g \circ \theta_h (x) =\theta_{gh} (x)$, making the data $(\{ A_g \}_{g\in G} , \{ \theta_g \}_{g\in G} )$ a partial action of the group $G$ on the algebra $A$.

The last step is to make a partial action of $G$ on the set $M=\mbox{Hom}_{\underline{Alg}} (A, k)$, defining for each $g\in G$ the subsets $M_g =\mbox{Hom}_{\underline{Alg}} (A_g , k)$, and the maps $\alpha_g : M_{g^{-1}} \rightarrow M_g$ given by $\alpha_g (P)=P\circ \theta_{g^{-1}}$, for $P\in M_{g^{-1}}$. It is easy to see that this data defines a partial action of $G$ on $M$. Thanks to \leref{decomp}, we know moreover that $M$ decomposes in a disjoint union $M_g\sqcup M'_g$ for any $g\in G$. Hence the partial action is algebraic and this concludes our proof.
\end{proof}

The above constructions are clearly functorial. Hence we obtain the following result, generalizing the classical result for global actions

\begin{corollary}
Let $H$ be a commutative Hopf algebra and $G=\Hom_{\Alg}(H,k)$ the corresponding algebraic group, then there is an equivalence between the category of  commutative  right partial $H$-comodule algebras and the category of algebraic partial actions of the group $G$.
\end{corollary}

As finite groups provide (trivial) examples of algebraic groups, the above result applies in particular to this situation. Hence we obtain a bijective correspondence between partial actions of a finite group $G$ on a finite set $X$ and partial coactions of the dual group algebra $(kG)^*$ on the algebra $A=\Fun(X,k)$ of $k$ valued functions on $X$. Explicitly, given a partial action $\alpha$ of $G$ on $X$, the partial coaction of $(kG)^*$ on $A$ is given by
\begin{equation}
\label{parcoa}
\prho (f) =\sum_{g\in G} e_g (f\circ \alpha_{g^{-1}}) \otimes p_g .
\end{equation}
where $f\in \Fun(X,k)$, $e_g$ is the characteristic function of the domain $X_g$ and $p_g\in kG^*$ is given by $p_g(h)=\delta_{g,h}$ for all $g,h\in G$.

\subsubsection{The Hopf algebroid associated to a partial coaction}

As shown previously, from a left partial action of a co-commutative Hopf algebra $H$ on a commutative algebra $A$, one can endow the partial smash product $\underline{A\# H}$ with a structure of an $A$ Hopf algebroid. The aim of this section is to show that in the dual case, of a right partial coaction of a commutative Hopf algebra $H$ on a commutative algebra $A$, one can construct a Hopf algebroid as well. 

From now on, let $H$ denote a commutative Hopf algebra $H$ and $A$ a commutative right partial $H$-module algebra, with partial coaction $\prho :A\rightarrow A\otimes H$. We can define the left source and target maps $s=s_l ,t=t_l : A\rightarrow A\underline{\otimes} H$ as
\[
s(a) =a1^{[0]} \otimes 1^{[1]} , \qquad t(a)=\prho (a)=a^{[0]}\otimes a^{[1]} .
\]
The right source and target maps are defined as $s_r =t_l$ and $t_r =s_l$. It is easy to see that both maps are algebra morphisms, since $A$ and $H$ are commutative algebras, $t$ is an algebra anti-morphism and the images of $s_l$ and $t_l$ mutually commute. The $A$-bimodule structure on $A\underline{\otimes} H$, is given by
\[
a\cdot (b1^{[0]}\otimes h1^{[1]} )\cdot c =s_l (a) t_l (c) (b1^{[0]} \otimes h1^{[1]})=abc^{[0]}1^{[0]} \otimes c^{[1]}1^{[1]} =(b1^{[0]} \otimes h1^{[1]}) s_r (c) t_r (a) .
\]

The coring structure, introduced in \leref{splitcoring},
\beqnast
\TDelta (a1^{[0]} \otimes h1^{[1]}) & = & a1^{[0]}\otimes h_{(1)} 1^{[1]} \otimes_A 1^{[0']} \otimes h_{(2)} 1^{[1']} \\
\Tepsilon (a1^{[0]} \otimes h1^{[1]}) & = & a\epsilon (h) 
\eqnast
will be the same for both, the left and the right bialgebroid sructures, this is because of the commutativity of both $A$ and $H$. 

Due to the axiom (PRHCA3) of partial coactions, we find the following useful expressions for the comultiplication:
\begin{eqnarray*}
\TDelta (a1^{[0]} \otimes h1^{[1]} ) &=&a1^{[0]}1^{[0']} \otimes h_{(1)} {1^{[1]}}_{(1)} 1^{[1']} \otimes_A 1^{[0'']} \otimes h_{(2)} {1^{[1]}}_{(2)} 1^{[1'']}\\
& =&    a1^{[0][0]} 1^{[0']} \otimes h_{(1)} 1^{[0][1]} 1^{[1']} \otimes_A 1^{[0'']} \otimes h_{(2)} 1^{[1]}1^{[1'']}\\
& =&  (a1^{[0']} \otimes h_{(1)} 1^{[1']})\cdot 1^{[0]} \otimes_A 1^{[0'']} \otimes h_{(2)} 1^{[1]}1^{[1'']} \\
&=&  a1^{[0']} \otimes h_{(1)} 1^{[1']} \otimes_A 1^{[0]}\cdot ( 1^{[0'']} \otimes h_{(2)} 1^{[1]}1^{[1'']} )\\
&& =  a1^{[0']} \otimes h_{(1)} 1^{[1']} \otimes_A 1^{[0]}1^{[0'']} \otimes h_{(2)} 1^{[1]}1^{[1'']} \\
& =&  a1^{[0']} \otimes h_{(1)} 1^{[1']} \otimes_A 1^{[0]} \otimes h_{(2)} 1^{[1]} .
\end{eqnarray*}

\begin{lemma}\lelabel{commutativeleftbialgebroid} 
The data $(A\underline{\otimes} H , A, s_l ,t_l ,\TDelta ,\Tepsilon )$ define a structure of a left $A$-bialgebroid.
\end{lemma}

\begin{proof} We have already seen that $A\underline{\otimes} H$ is an $A$ coring and because of the commutativity of $A$, we have that the $A$ bimodule structure of this coring is compatible with both  source and target above defined. Again, as $A$ and $A\underline{\otimes} H$ are commutative, the image of the comultiplication is automatically contained in the Takeuchi's product, which in the commutative case is the ordinary tensor product over $A$
\[
 (A\underline{\otimes} H) \times_A (A\underline{\otimes} H) = (A\underline{\otimes} H) \otimes_A (A\underline{\otimes} H) .
\]
The comultiplication is also multiplicative, 
\begin{eqnarray*}
\TDelta (a1^{[0]}\otimes h1^{[1]})\TDelta (b1^{[0']}\otimes k 1^{[1']}) \\
&&\hspace{-3cm} =  (a1^{[0]}\otimes h_{(1)}1^{[1]} \otimes_A 1^{[0']}\otimes h_{(2)} 1^{[1']})(b1^{[0'']}\otimes k_{(1)}1^{[1'']} \otimes_A 1^{[0''']} \otimes k_{(2)}1^{[1''']} ) \\
&&\hspace{-3cm} =  (a1^{[0]}\otimes h_{(1)}1^{[1]})(b1^{[0'']}\otimes k_{(1)}1^{[1'']}) \otimes_A 
(1^{[0']}\otimes h_{(2)} 1^{[1']})(1^{[0''']} \otimes k_{(2)}1^{[1''']} ) \\
&&\hspace{-3cm} =  ab1^{[0]}\otimes h_{(1)}k_{(1)} 1^{[1]} \otimes_A 1^{[0']} \otimes h_{(2)}k_{(2)} 1^{[1']} 
 =  \TDelta (ab1^{[0]} \otimes hk1^{[1]}) \\
&&\hspace{-3cm} =  \TDelta ((a1^{[0]}\otimes h1^{[1]})(b1^{[0']}\otimes k1^{[1']})) .
\end{eqnarray*}
Let us verify the axiom of the counit for a bialgebroid:
\[
\Tepsilon (XY) =\Tepsilon (X s(\Tepsilon (Y)))=\Tepsilon (X t(\Tepsilon (Y))), \qquad \forall X,Y\in A\underline{\otimes} H .
\]
For the first equality, we have
\begin{eqnarray*}
\Tepsilon ((a1^{[0]}\otimes h1^{[1]})s(\Tepsilon (b1^{[0']}\otimes k1^{[1']}))) & = & \Tepsilon ((a1^{[0]}\otimes h1^{[1]})(b\epsilon (k) 1^{[0']}\otimes 1^{[1']})) \\
& = & \Tepsilon (ab\epsilon (k) 1^{[0]} \otimes h1^{[1]}) 
 =  ab\epsilon (h) \epsilon (k) \\
& = & \Tepsilon (ab1^{[0]}\otimes hk 1^{[1]}) 
 =  \Tepsilon ((a1^{[0]}\otimes h1^{[1]})(b1^{[0']}\otimes k1^{[1']})) .
\end{eqnarray*}
And for the second equality,
\begin{eqnarray*}
\Tepsilon ((a1^{[0]}\otimes h1^{[1]})t(\Tepsilon (b1^{[0']}\otimes k1^{[1']}))) & = & 
\Tepsilon ((a1^{[0]}\otimes h1^{[1]})(b^{[0]}\epsilon (k) \otimes b^{[1]})) \\
& =&  \Tepsilon (ab^{[0]}\epsilon (k) 1^{[0]}\otimes hb^{[1]}1^{[1]}) 
 =  ab^{[0]}\epsilon (b^{[1]}) \epsilon (h) \epsilon (k) \\
& = & ab\epsilon (h) \epsilon (k) 
 =  \Tepsilon (ab1^{[0]}\otimes hk 1^{[1]}) \\
& = & \Tepsilon ((a1^{[0]}\otimes h1^{[1]})(b1^{[0']}\otimes k1^{[1']})) .
\end{eqnarray*}
Therefore, we have that $(A\underline{\otimes} H, s,t, \TDelta ,\Tepsilon )$ is an $A$ bialgebroid. 
\end{proof}

The structure of a right $A$ bialgebroid is completely analogous, due to the fact that $A$ is a commutative algebra. Then we have the following result.

\begin{lemma}\lelabel{commutativerightbialgebroid} 
The data $(A\underline{\otimes} H , A, s_r ,t_r ,\TDelta ,\Tepsilon )$ define a structure of a right $A$ bialgebroid.
\end{lemma}

Finally, we have the structure of a Hopf algebroid on $A\underline{\otimes}H$.

\begin{thm}\thlabel{Hoidpartialcoaction}
Let $H$ be a commutative Hopf algebra and $A$ be a commutative right partial $H$ comodule algebra with partial coaction $\prho$, as above. Then the algebra $A\underline{\otimes} H =\prho (1_A )(A\otimes H)$ is a commutative Hopf algebroid over the base algebra $A$, called the partial split Hopf algebroid. 
The partial coaction $\prho$ is global if, and only if, the Hopf algebroid $A\underline{\otimes} H$ coincides with the split Hopf algebroid $A\otimes H$. 
\end{thm}

\begin{proof}
By \leref{commutativeleftbialgebroid} and \leref{commutativerightbialgebroid}, we know that $A\ul\ot H$ is a bialgebroid. To obtain a Hopf algebroid structure, we define an antipode $\tilde S$ by 
\[
\TS (a1^{[0]} \otimes h1^{[1]} ) =a^{[0]}1^{[0]} \otimes a^{[1]} S(h) 1^{[1]} .
\]
which is verified to be an anti algebra morphism. The Hopf algebroid axioms (i) and (ii) (see \seref{algstruc}) are easy to check. For the axiom (iii), take $a,c\in A$ and $b1^{[0]}\otimes h1^{[1]}\in A\underline{\otimes} H$, then
\beqnast 
\TS (t_l (a)(b1^{[0]}\otimes h1^{[1]})t_r (c)) & = & \TS (a^{[0]} b c1^{[0]}\otimes a^{[1]}h1^{[1]}) \\
&&\hspace{-3cm} =  a^{[0][0]} b^{[0]}c^{[0]}1^{[0]} \otimes a^{[0][1]}b^{[1]}c^{[1]} 1^{[1]} S(h) S(a^{[1]}) \\
&&\hspace{-3cm} =  \left( c^{[0]}\otimes c^{[1]} \right) \left( b^{[0]}a^{[0][0]} 1^{[0]} \otimes b^{[1]} a^{[0][1]} 1^{[1]} S(a^{[1]}) S(h) \right) \\
&&\hspace{-3cm} =  s_r (c) \left( b^{[0]} 1^{[0]} a^{[0]} \otimes b^{[1]} S(h) 1^{[1]} {a^{[1]}}_{(1)} S({a^{[1]}}_{(2)}) \right) \\
&&\hspace{-3cm} =  s_r (c) \left( b^{[0]} 1^{[0]} a^{[0]} \epsilon (a^{[1]}) \otimes b^{[1]} S(h) 1^{[1]} \right) 
 =  s_r (c) \left( b^{[0]} 1^{[0]} a \otimes b^{[1]} S(h) 1^{[1]} \right) \\
&&\hspace{-3cm} =  s_r (c) ( b^{[0]} 1^{[0]}  \otimes b^{[1]} S(h) 1^{[1]})  s_l (a) 
 =  s_r(c)\TS (b1^{[0]}\otimes h1^{[1]}) s_l (a) .
\eqnast
Finally, for the item (iv), take $a1^{[0]}\otimes h1^{[1]}\in A\underline{\otimes} H$, then  
\beqnast
\mu \circ (\TS \otimes I)\circ \TDelta (a1^{[0]}\otimes h1^{[1]}) & = & \TS (a1^{[0]}\otimes h_{(1)} 1^{[1]} )(1^{[0']}\otimes h_{(2)} 1^{[1']} )\\
& = & (a^{[0]}1^{[0]}\otimes a^{[1]} S(h_{(1)}) 1^{[1]} )(1^{[0']}\otimes h_{(2)} 1^{[1']} )\\
& = & a^{[0]}1^{[0]}\otimes a^{[1]} S(h_{(1)}) h_{(2)} 1^{[1]} 
 =  a^{[0]} \epsilon (h) 1^{[0]}\otimes a^{[1]} 1^{[1]} \\
& = & s_r (a\epsilon (h)) 
 =  s_r \circ \ue_r (a1^{[0]}\otimes h1^{[1]}) ,
\eqnast
and 
\beqnast
\mu \circ (I\otimes \TS )\circ \TDelta (a1^{[0]}\otimes h1^{[1]}) & = & 
(a1^{[0]}\otimes h_{(1)} 1^{[1]} )\TS (1^{[0']}\otimes h_{(2)} 1^{[1']} )\\
& = & (a1^{[0]}\otimes h_{(1)} 1^{[1]} )\TS (1^{[0']}\otimes S(h_{(2)}) 1^{[1']} )\\
& = & a1^{[0]}\otimes h_{(1)} S(h_{(2)}) 1^{[1]} 
 =  a1^{[0]} \epsilon (h) \otimes 1^{[1]} )\\
& = & s_l (a\epsilon (h))
 =  s_l \circ \ue_l (a1^{[0]}\otimes h1^{[1]}).
\eqnast
Therefore, $A\underline{\otimes}H$ is a Hopf algebroid.

Now consider a global coaction $\prho: A\rightarrow A\otimes H$, then $\prho$ is a unital morphism, that is, $\prho (1_A )=1_A \otimes 1_H$. Therefore 
\[
A\underline{\otimes}H =\prho (1_A )(A\otimes H) =A\otimes H.
\]
As explained in \seref{algstruc} $A\otimes H$ can be endowed with the structure of a Hopf algebroid, called the split Hopf algebroid. One easily observes that thanks to the simplification $1^{[0]}\otimes 1^{[1]}=1_A \otimes 1_H$, the split Hopf algebroid coincides with the Hopf algebroid structure on $A\underline{\otimes}H$ defined above. On the other hand, if $A\underline{\otimes}H =A\otimes H$, then $\prho (1_A )=1_A \otimes 1_H$ which makes $\prho$ into a global coaction.
\end{proof}

\subsubsection{A dual version of a theorem by Kellendock and Lawson}

As a motivation for our next result, we will first recall the relationship between partial actions of groups on sets and groupoids as this was pointed out initially by Abadie \cite{Aba2} and shown to lead to an equivalence of categories by Kellendonk and Lawson \cite{KL}. 

Consider a partial action of a group $G$ on a set 
$X$, given by the data $( \{ X_g \subseteq X \}_{g\in G} , \{ \alpha_g : X_{g^{-1}}\rightarrow X_g \}_{g\in G} )$. From these data it is possible to construct a groupoid
\[
\mathcal{G}(G,X,\alpha )=\{ (x,g)\in X\times G \, |\, x\in X_g \}.
\]
The objects in this groupoid are the elements of $X$. The groupoid structure is given by the source and target maps,
\[
s(x,g)= \alpha_{g^{-1}} (x) , \qquad t(x,g)=x ,
\]
the product
\[
(x,g)(y,h)=\left\{ \begin{array}{ll} (x, gh) & \mbox{if} \ y=\alpha_{g^{-1}}(x)\\
                                      \underline{\;} & \mbox{otherwise,}
                   \end{array} \right.
\]
and the inverse
\[
(x,g)^{-1} =(\alpha_{g^{-1}} (x) ,g^{-1}) .
\]

It is easy to see that the projection map 
\[
\begin{array}{rccc} \pi_2 : & \mathcal{G}(G,X,\alpha ) & \rightarrow & G \\
\, & (x,g) & \mapsto & g
\end{array}
\]
is a functor, considering $G$ as a groupoid with one single element. This functor is {\em star injective}, which is a strong version of faithfulness, in the sense of the following definition.
\begin{defi} \begin{enumerate}[(1)]\item Let $\mathcal{C}$ be a small category and $x$ be an object in $\mathcal{C}$. The star over $x$ is the set $\mathcal{S}(x)= \{ f\in \mbox{Hom}_{\mathcal{C}} (x,y) \, | \, y\in \mathcal{C} \}$.
\item
 A functor $F:\mathcal{C} \rightarrow \mathcal{D}$ between two small categories is said to be star injective (surjective) if for every $x\in \mathcal{C}$, the function $F|_{\mathcal{S}(x)}$ is injective (surjective).
\end{enumerate}
\end{defi}

Applied to the functor $\pi_2$, star injectivity simply means that given $x\in X$ and $g\in G$ there exists at most one element $\gamma\in \mathcal{G}(G,X,\alpha )$ such that $s(\gamma )=x$ and $\pi_2 (\gamma )=g$ (by the way, this element exists only when $x\in X_{g^{-1}}$ and it is written explicitly as $\gamma =(\alpha_g (x), g)$). The functor $\pi_2$ is star surjective if, and only if, the action of $G$ on $X$ is global \cite{KL}.

Our Hopf algebroid $A\underline{\otimes} H$ can be seen as a dual of this groupoid arising from partial group actions, at least for the case of an algebraic partial action $\alpha$ of an affine algebraic group on an affine set. Let us give the details in case of a finite group $G$ acting partially on a finite set $X$. Consider the right partial coaction of the Hopf algebra $H=(kG)^*$ on the algebra $A=\mbox{Fun}(X,k)$ as defined by the formula
\[
\prho (f) (p,g) =\left\{ \begin{array}{lcr} f(\alpha_{g^{-1}} (p)) & \mbox{ for } & p\in X_g \\
0 & \,  & \mbox{otherwise} \end{array} \right.
\]
denoting the basis of $(kG)^*$ by $\{ p_g \}_{g\in G}$, where $p_g (h)=\delta_{g,h}$ and denoting the characteristic functions of $X_g$ by $1_g$, then the partial coaction can be written as
\[
\prho (f) =\sum_{g\in G} 1_g .(f\circ \alpha_{g^{-1}}) \otimes p_g.
\]
The Hopf algebroid $A\underline{\otimes} H$ is exactly the Hopf algebroid of functions on the groupoid $\mathcal{G}(G,X,\alpha )$. Indeed, first note that
\[
\prho (1_A )= \sum_{g\in G} 1_g \otimes p_g,
\]
then, an element $(a\otimes h) (\prho (1_A ))\in A\underline{\otimes}H$ is a function defined on $\mathcal{G}(G,X,\alpha )$, because if evaluated on a pair $(x,g)$ it will be nonzero only if $x\in X_g$. On the other hand, 
\[
\mbox{Fun} (\mathcal{G}(G,X,\alpha ) ,k)\subseteq \mbox{Fun} (X\times G ,k)\cong A\otimes H .
\]
denoting by $\chi_x \in \mbox{Fun}(X,k)$ the characteristic function of the element $x\in X$, then any $\varphi \in \mbox{Fun} (\mathcal{G}(G,X,\alpha ), k)$ can be written as
\[
\varphi =\sum_{x\in X, g\in G} a_{x,g} \chi_x \otimes p_g ,
\]
with $a_{x,g} \in k$, for every $x\in X$ , $g\in G$. But, as $\varphi$ is only defined on the groupoid $\mathcal{G}(G,X,\alpha )$ then $a_{x,g} =0$ if $x\notin X_g$, therefore
\[
\varphi =\sum_{ g\in G} \left(\sum_{x\in X_g} a_{x,g}\chi_x  \right) 1_g  \otimes p_g \in (A\otimes H)\prho (1_A )=A\underline{\otimes }H .
\]

Note that we have a morphism of algebras
\[
\begin{array}{rccc} F: & H & \rightarrow & A\underline{\otimes} H \\
\, & h & \mapsto & 1^{[0]}\otimes h1^{[1]} 
\end{array}
\]
It is easy to show that this map also satisfies
\begin{enumerate}
\item $\Tepsilon \circ F =\eta_A \circ \epsilon $;
\item $\TDelta \circ F = \pi \circ (F\otimes F) \circ \Delta$;
\item $\TS\circ F=F\circ S$;
\end{enumerate}
where the map $\pi :(A\underline{\otimes} H) \otimes (A\underline{\otimes} H) \rightarrow  
(A\underline{\otimes} H) \otimes_A (A\underline{\otimes} H)$ is the natural projection. That is equivalent to say that the map $F$ is a morphism of Hopf algebroids with different base algebras, considering the Hopf algebra $H$ as a Hopf algebroid having as base algebra the base field $k$.

Finally, note that, the Hopf algebroid, $A\underline{\otimes} H$ is totally defined by the image of the source map and the image of the map $F$. Indeed, a general element of $A\underline{\otimes}H$ can be written as
\[
\sum_i a^i 1^{[0]} \otimes h^i 1^{[1]} = \sum_i s(a^i )F(h^i ) .
\]

Given a commutative Hopf algebra $H$ and a commutative $A$ Hopf algebroid $\mathcal{H}$ such that there exists an algebra morphism $F:H\rightarrow \mathcal{H}$, one can construct out of these data an algebra morphism
\[
\begin{array}{rccc}
\Pi :  & A\otimes H & \rightarrow & \mathcal{H} \\
\, & a\otimes h & \mapsto & s(a) F(h) 
\end{array}
\]
In the case of the Hopf algebroid $A\underline{\otimes} H$, the morphism $\Pi$ is surjective and splits by the canonical inclusion. This motivates the following definition:

\begin{defi} Given a commutative Hopf algebra $H$ and a commutative $A$ Hopf algebroid $\mathcal{H}$, an algebra morphism $F:H\rightarrow \mathcal{H}$ is said to be right dual star injective, if 
\begin{enumerate}
\item[(DSI1)] the map $F$ is a morphism of Hopf algebroids, considering $H$ as a Hopf algebroid over the base field $k$;
\item[(DSI2)] the Hopf algebroid $\mathcal{H}$ is generated, as algebra, by the image of the source map and the image of $F$; that is $\mathcal{H}=s(A)F(H)$;
\item[(DSI3)] The surjective algebra map $\Pi :A\otimes H \rightarrow \mathcal{H}$ given by $\Pi (a\otimes h) =s(a)F(h)$ splits as an algebra morphism.
\end{enumerate}
\end{defi}

Obviously, there is a left version, by modifying the item (DSI3) by taking the algebra map $\Pi': H\otimes A \rightarrow \mathcal{H}$. The origin of the name ``dual star injective'' is that these axioms 
arise as a dualization of the star injectivity conditions for a functor between a groupoid $\mathcal{G}$ and a group $G$, taking the algebras of functions on these as explained above.

Recall from \cite{KL} that given a groupoid $\mathcal{G}$ and a star injective functor $F:\mathcal{G}\rightarrow G$, one can construct a partial action of the group $G$ on the set $X$ of objects of $\mathcal{G}$. For this, one defines for each $g\in G$ the subset
\[
X_g =\{ x\in X \, | \, \exists \gamma \in \mathcal{G}, \; s(\gamma )=x, \; 
F(\gamma )=g^{-1} \} ;
\]
and bijections $\alpha_g : X_{g^{-1}}\rightarrow X_g$ given by
\[
\alpha_g (x) =t(\gamma ), \qquad \mbox{ such that } \; s(\gamma )=x , \mbox{ and } \; 
F(\gamma ) = g .
\]
Moreover, this defined partial action is global if, and only if, the functor $F$ is star surjective.

Supposse now that $F:\mathcal{G}\rightarrow G$ is a star injective functor, where $\mathcal{G}$ is a finite groupoid and $G$ is a finite group. Then the functor $F$ induces an algebra morphism $\hat{F}$ between the Hopf algebra $H=(kG)^*$ and the Hopf algebroid $\mathcal{H}=\mbox{Fun}(\mathcal{G},k)$ by posing
$\hat{F} (f) (\gamma )=f(F(\gamma))$, for all $f\in (kG)^*$ and $\gamma \in \mathcal{G}$. The functoriality of $F$ implies directly the item DSI1 of the definition of a dual star injective algebra morphism. 

For the item DSI2, note that, for a finite groupoid $\mathcal{G}$ its algebra of functions has a basis of characteristic functions $\{ \chi_{\gamma} | \gamma \in \mathcal{G} \}$. Let $\gamma \in \mathcal{G}$ denote by $x=s(\gamma )$ and 
$g=F(\gamma )$. Then, it is easy to see that the characteristic function $\chi_{\gamma}$ can be written as
\[
\chi_{\gamma} =s(\chi_x )\hat{F} (p_g) =\Pi (\chi_x \otimes p_g ).
\]
Finally, for the item DSI3, consider the map $\sigma :\mathcal{H}\rightarrow A\otimes H$ given by
\[
\sigma (f) =\sum_{\gamma \in \mathcal{G}} f( \gamma ) \chi_{t(\gamma )}\otimes 
p_{F(\gamma )}.
\]
It is easy to see that $\Pi \circ \sigma =\mbox{Id}_{\mathcal{H}}$. Therefore, the morphism $\hat{F}$ is dual star injective.

The next result shows a deeper connection between partial coactions of commutative  Hopf algebras and commutative Hopf algebroids, and can be viewed as a dual version of the result by Kellendonk-Lawson cited above.

\begin{thm} \label{dualkellendonk} Let $H$ be a commutative Hopf algebra, $\mathcal{H}$ be a commutative $A$ Hopf algebroid and $F: H\rightarrow \mathcal{H}$ be a right dual star injective algebra morphism. Then $A$ is a right partial $H$ comodule algebra. The partial coaction defined on $A$ is global if, and only if, the morphism $\Pi :A\otimes H \rightarrow \mathcal{H}$ is an isomorphism.
\end{thm}

\begin{proof} The map $F$ induces a right $H$-module structure on $\mathcal{H}$, given by $x\cdot h =xF(h)$. Let $\sigma :\mathcal{H} \rightarrow A\otimes H$ be the algebra map such that $\Pi \circ \sigma =\mbox{Id}_{\mathcal{H}}$. The fact that both $\Pi$ and $\sigma$ are morphisms of $A-H$ bimodules. Then for any element 
\[
x=\sum_i s(a^i) F(h^i) \in \mathcal{H} ,
\]
we have the following expression for $\sigma (x)$,
\[
\sigma (x) =\sum_i  \sigma (s(a^i)F(h^i)) =\sum_i \sigma (a^i \cdot 1_{\mathcal{H}} \cdot h^i) =\sum_i a^i \cdot \sigma (1_{\mathcal{H}} )\cdot h^i .
\]
Writing $\sigma (1_{\mathcal{H}})=1^{[0]} \otimes 1^{[1]}$, which is an idempotent in $A\otimes H$, we have that any element in the image of $\sigma$ is of the form
\[
\sigma (x) =\sum_i a^i 1^{[0]} \otimes h^i1^{[1]} .
\]
The projection $\sigma \circ \Pi : A\otimes H \rightarrow \sigma (\mathcal{H})$ is implemented by the multiplication by the idempotent 
$\sigma (1_{\mathcal{H}})=1^{[0]}\otimes 1^{[1]}$. 

For any $a\in A$ we have, obviously, $\sigma (s(a)) =a1^{[0]}\otimes 1^{[1]}$ and let us denote 
\[
\prho=t\circ\sigma:A\to A\ot H,\ \prho (a)= \sigma (t(a))=a^{[0]}\otimes a^{[1]} .
\]
The morphism $\prho$ is automatically an algebra morphism, because $\sigma$ and $t$ are morphisms of algebras. We shall prove that this morphism $\prho$ is a right partial coaction of the Hopf algebra $H$ on the algebra $A$. 

Let $\theta :\sigma (\mathcal{H}) \otimes_A \sigma (\mathcal{H}) \rightarrow \sigma (\mathcal{H}) \otimes H$ be the linear map given by
\[
\theta (a1^{[0]}\otimes h1^{[1]} \otimes_A b1^{[0']} \otimes k1^{[1']} )=ab^{[0]}1^{[0]} \otimes hb^{[1]}1^{[1]} \otimes k .
\]
One can easily prove that $\theta$ is also an algebra map.

In order to prove that $\prho$ is a partial coaction, the following two identities will be useful:
\begin{enumerate}
\item[(I)] $\epsilon_{\mathcal{H}} \circ \Pi =I\otimes \epsilon_H$.
\item[(II)] $\theta \circ (\sigma \otimes \sigma) \circ \Delta_{\mathcal{H}} \circ \Pi =((\sigma \circ \Pi ) \otimes I) \circ (I\otimes \Delta_H )$
\end{enumerate}

For the identity (I), take $a\otimes h\in A\otimes H$, then
\begin{eqnarray*}
\epsilon_{\mathcal{H}} \circ \Pi (a\otimes h ) & = & \epsilon_{\mathcal{H}} (s(a) F(h)) 
 =  a(\epsilon_{\mathcal{H}} (F(h))) \\
& = & a\epsilon_H (h) 
 =  (I\otimes \epsilon_H )(a\otimes h) .
\end{eqnarray*}

And for the identity (II),
\begin{eqnarray*}
\theta \circ (\sigma \otimes \sigma) \circ \Delta_{\mathcal{H}} \circ \Pi (a\otimes h) 
& = & \theta \circ (\sigma \otimes \sigma) \circ \Delta_{\mathcal{H}} (s(a)F(h)) 
 =  a\cdot (\theta \circ (\sigma \otimes \sigma ) \circ \Delta_{\mathcal{H}} (F(h))) \\
& = & a\cdot \theta (1^{[0]}\otimes 1^{[1]} h_{(1)} \otimes_A 1^{[0']} \otimes h_{(2)} 1^{[1']}) 
 =  a 1^{[0]}\otimes 1^{[1]} h_{(1)} \otimes h_{(2)} \\
& = & \sigma \circ \Pi (a\otimes h_{(1)}) \otimes h_{(2)} 
 =  ((\sigma \circ \Pi ) \otimes I) \circ (I \otimes \Delta_H ) (a\otimes h) .
\end{eqnarray*}

Then we are in position to verify the axioms of a right partial coaction for the map $\prho =\sigma \circ t :A\rightarrow \sigma (\mathcal{H})\subseteq A\otimes H$. The axiom (PRHCA1) is automatically satisfied because both $\sigma$ and $t$ are algebra maps, therefore $\prho$ is an algebra map. 

For the axiom (PRHCA2) take $a\in A$, then
\begin{eqnarray*}
(I\otimes \epsilon_H )\circ \prho (a) & = & (I \otimes \epsilon_H ) \circ \sigma (t(a)) 
 =  \epsilon_{\mathcal{H}} \circ \Pi \circ \sigma (t(a)) \\
& = & \epsilon_{\mathcal{H}} (t(a)) 
= a .
\end{eqnarray*}

For the axiom (PRHCA3) we first note that $(\prho \otimes I) \circ \prho =\theta \circ (\sigma \otimes \sigma ) \circ \Delta_{\mathcal{H}} \circ t$. Indeed, for $a\in A$
\begin{eqnarray*}
\theta \circ (\sigma \otimes \sigma ) \circ \Delta_{\mathcal{H}} (t(a)) 
& = & \theta \circ (\sigma \otimes \sigma )  (1_{\mathcal{H}}\otimes_A  t(a) ) 
 =  \theta (1^{[0]} \otimes 1^{[1]} \otimes_A a^{[0]} \otimes a^{[1]} ) \\
& = & 1^{[0]}a^{[0][0]} \otimes 1^{[1]}a^{[0][1]} \otimes a^{[1]} 
 =  a^{[0][0]} \otimes a^{[0][1]} \otimes a^{[1]} \\
& = & (\prho \otimes I) \circ \prho (a) .
\end{eqnarray*}
On the other hand,
\begin{eqnarray*}
\theta \circ (\sigma \otimes \sigma ) \circ \Delta_{\mathcal{H}} (t(a)) & = &  \theta \circ (\sigma \otimes \sigma ) \circ \Delta_{\mathcal{H}} \circ \Pi \circ \sigma (t(a)) 
 =  ((\sigma \circ \Pi ) \otimes I) \circ (I \otimes \Delta_H ) \circ \sigma (t(a)) \\
& = & ((\sigma \circ \Pi ) \otimes I) \circ (I \otimes \Delta_H ) \circ \prho (a) 
 =  (\sigma (1_{\mathcal{H}}) \otimes 1_H ) [(I\otimes \Delta_H)\circ \prho (a)] \\
& = & (\prho (1_A )\otimes 1_H )[(I\otimes \Delta_H)\circ \prho (a)].
\end{eqnarray*}
Therefore, $(\prho \otimes I) \circ \prho (a)=(\prho (1_A )\otimes 1_H )[(I\otimes \Delta_H)\circ \prho (a)]$ for any $a\in A$, and $\prho$ is a right partial coaction.

If the partial coaction $\prho =\sigma \circ t$ is global, then $\prho (1_A )=1^{[0]} \otimes 1^{[1]}=1_A \otimes 1_H$. Then we have for any $a\otimes h \in A\otimes H$
\[
\sigma \circ \pi (a\otimes h) =a1^{[0]}\otimes h1^{[1]} =a\otimes h.
\]
Therefore $\sigma =\Pi^{-1}$, which implies that $\mathcal{H} \cong A\otimes H$.

On the other hand, if $\Pi :A\otimes H\rightarrow \mathcal{H}$ is an isomorphism, then $\sigma \circ \Pi =\mbox{Id}_{A\otimes H}$ and
\[
(\prho \otimes I)\circ \prho (a) =\theta \circ (\sigma \otimes \sigma ) \circ \Delta_{\mathcal{H}} (t(a)) = ((\sigma \circ \Pi ) \otimes I) \circ (I \otimes \Delta_H ) \circ \prho (a)= (I \otimes \Delta_H ) \circ \prho (a).
\]
Therefore, the partial coaction $\prho$ is global.
\end{proof}

\subsubsection{A duality result}

Given two dually paired Hopf algebras, $H$ and $K$, a right partial comodule algebra $A$ is automatically a left partial $H$-module algebra. In the special case when $H$ is a  co-commutative Hopf algebra, and consequently $K$ is commutative, and $A$ is a commutative algebra, then the duality between partial actions and partial coactions can be viewed as a duality between bialgebroids. In order to make this statement more precise, let us define what is a skew pairing between left $A$ bialgebroids (for $A$ not necessarily commutative).

\begin{defi} (See \cite{S}) Let $A$ be a $k$-algebra, $(\Lambda , s,t,\td ,\te )$ and $(L, s,t, \ud , \ue)$ be two left $A$-bialgebroids. A skew paring between $\Lambda$ and $L$ is a $k$-linear map $\llangle \, | \, \rrangle :\Lambda \otimes_k L \rightarrow A$ satisfing
\begin{enumerate}
\item[(SP1)] $\llangle s(a)t(b) \xi s(c)t(d) | \ell \rrangle e = a \llangle \xi | s(c)t(e) \ell s(d) t(b) \rrangle$ for every $\xi \in \Lambda$, $\ell \in L$, and  $a,b,c,d,e\in A$;
\item[(SP2)] $\llangle \xi | \ell m \rrangle = \llangle \xi_{(1)} | \ell 
t( \llangle \xi_{(2)} |m \rrangle ) \rrangle$, for every $\xi\in \Lambda$ and $\ell ,m\in L$;
\item[(SP3)] $\llangle \xi \zeta | \ell \rrangle = \llangle \xi s( \llangle \zeta | \ell_{(1)} \rrangle ) | \ell_{(2)} \rrangle$, for every $\xi ,\zeta \in \Lambda$ and $\ell \in L$;
\item[(SP4)] $\llangle \xi | 1_L \rrangle =\te (\xi )$ for every $\xi \in \Lambda$;
\item[(SP5)] $\llangle 1_{\Lambda} | \ell \rrangle =\ue (\ell )$ for every $\ell \in L$.
\end{enumerate}
\end{defi}

Then we have the following duality result. 

\begin{thm} Let $H$ be a co-commutative Hopf algebra, $K$ be a commutative Hopf algebra with a dual pairing $\langle \, , \, \rangle :H\otimes K \rightarrow k$, and let $A$ be a commutative algebra which is, at the same time, a left partial $H$-module algebra and a right partial $K$-comodule algebra with the compatibility condition
$h\cdot a =a^{[0]} \langle h,a^{[1]} \rangle$. Then the map 
\[
\begin{array}{rccc} \llangle  \, , \, \rrangle : &  A\underline{\otimes} K  \otimes_k \underline{A\# H}& \rightarrow & A \\
\, &  (a1^{[0]}\otimes \xi 1^{[1]}) \otimes_k (b\# h) & \mapsto & ab(h_{(1)}\cdot 1_A)\langle h_{(2)},\xi \rangle 
\end{array}
\]
is a skew pairing between the $A$ bialgebroids $A\underline{\otimes} K$ and $\underline{A\# H}$. 
\end{thm}

\begin{proof} First, note that thanks to the compatibility condition between the left partial $H$-module structure and the right partial $K$-comodule structure on $A$, we have
\[
\llangle a1^{[0]} \otimes \xi 1^{[1]} | b\# h \rrangle =ab1^{[0]} \langle h, \xi 1^{[1]} \rangle 
\]
. It is easy to see that the map $\llangle \, , \, \rrangle$ is a well defined, indeed, as we know, in the partial smash product the equality $b\# h =b(h_{(1)}\cdot 1_A )\# h_{(2)}$ holds, then, for any $a1^{[0]}\otimes \xi 1^{[1]} \in A\underline{\otimes}K$ we have
\beqnast 
\llangle a1^{[0]} \otimes \xi 1^{[1]} | b(h_{(1)}\cdot 1_A )\# h_{(2)}\rrangle & = & ab(h_{(1)}\cdot 1_A )(h_{(2)} \cdot 1_A )\langle h_{(3)} ,\xi \rangle \\
&&\hspace{-3cm} =  ab(h_{(1)}\cdot 1_A)\langle h_{(2)},\xi \rangle 
 =  \llangle a1^{[0]} \otimes \xi 1^{[1]} | a\# h \rrangle .
\eqnast
Let us verify the axiom (SP1). As the bialgebroid $A\underline{\otimes} K$ is commutative, this enables a simplification, because, for each $a_1 , a_2 \in A$ and $x\in A\underline{\otimes} K$ we have $s(a_1 )t(a_2) x= xs(a_1 )t(a_2 )$. Then, taking $a_1 , a_2, a_3 \in A$, $b1^{[0]}\otimes \xi 1^{[1]} \in A\underline{\otimes} K$ and $c\# h\in \underline{A\# H}$, we have
\beqnast
\llangle s(a_1 )t(a_2)(b1^{[0]}\otimes \xi 1^{[1]})|c\# h \rrangle a_3 & = & 
\llangle (a_1 ba_2^{[0]}1^{[0]}\otimes  a_2^{[1]} \xi 1^{[1]})|c\# h \rrangle a_3 \\
&&\hspace{-4cm} =  a_1 bc a_3 (h_{(1)}\cdot 1_A ) a_2^{[0]} \langle h_{(2)} ,  a_2^{[1]} \xi \rangle 
 =  a_1 bc a_3 (h_{(1)}\cdot 1_A ) a_2^{[0]} \langle h_{(2)} ,  a_2^{[1]} \rangle \langle h_{(3)} ,\xi \rangle \\
&&\hspace{-4cm} =  a_1 bc a_3 (h_{(1)}\cdot 1_A ) (h_{(2)}\cdot a_2 ) \langle h_{(3)} ,\xi \rangle 
 =  a_1 bc a_3 (h_{(1)}\cdot a_2 ) \langle h_{(2)} ,\xi \rangle \\
&&\hspace{-4cm} =  a_1 \llangle b1^{[0]}\otimes \xi 1^{[1]} | a_3 c ((h_{(1)}\cdot a_2 ) \# h_{(2)} \rrangle 
 =  a_1 \llangle b1^{[0]}\otimes \xi 1^{[1]} | s( a_3 )(c\# h)(a_2 \# 1_H ) \rrangle \\
&&\hspace{-4cm} =  a_1 \llangle b1^{[0]}\otimes \xi 1^{[1]} | s( a_3 )(c\# h) t(a_2 ) \rrangle 
\eqnast
The last expression is also equal to $\llangle b1^{[0]}\otimes \xi 1^{[1]} | s(a_1) 
s( a_3 )(c\# h) t(a_2 ) \rrangle$.

For the axiom (SP2), take $a1^{[0]}\otimes \xi 1^{[1]} \in A\underline{\otimes} K$, and $b\# h , c\# k \in \underline{A\# H}$, then
\beqnast
\llangle a1^{[0]}\otimes \xi 1^{[1]} |(b\# h) (c\# k) \rrangle & = &  \llangle a1^{[0]}\otimes \xi 1^{[1]} |(b (h_{(1)}\cdot c)\# h_{(2)}k) \rrangle \\
&&\hspace{-4cm} =  ab(h_{(1)}\cdot c)(h_{(2)}k_{(1)}\cdot 1_A ) \langle h_{(2)}k_{(1)}, \xi \rangle 
 =  ab(h_{(1)}\cdot (c(k_{(1)}\cdot 1_A ))) \langle h_{(2)}k_{(1)}, \xi \rangle ,
\eqnast
while, on the other hand
\beqnast 
&&\hspace{-2cm} \llangle a1^{[0]}\otimes \xi_{(1)} 1^{[1]} |(b\# h) t( \llangle 1^{[0']} \otimes 
\xi_{(2)} 1^{[1']} |(c\# k) \rrangle ) \rrangle \\
& = & \llangle a1^{[0]}\otimes \xi_{(1)} 1^{[1]} |(b\# h) ( c(k_{(1)}\cdot 1_A) \# 1_H )  \rrangle \langle k_{(2)} ,\xi_{(2)} \rangle \\
& = & \llangle a1^{[0]}\otimes \xi_{(1)} 1^{[1]} |(b(h_{(1)}\cdot ( c(k_{(1)}\cdot 1_A))) \# h_{(2)} )  \rrangle \langle k_{(2)} ,\xi_{(2)} \rangle \\
& = &  ab(h_{(1)}\cdot ( c(k_{(1)}\cdot 1_A))) \langle h_{(2)} ,\xi_{(1)}  \rangle \langle k_{(2)} ,\xi_{(2)} \rangle \\
& = & ab(h_{(1)}\cdot (c(k_{(1)}\cdot 1_A ))) \langle h_{(2)}k_{(1)}, \xi \rangle .
\eqnast
For the axiom (SP3), take $a1^{[0]}\otimes \xi 1^{[1]} \, , \,  b1^{[0]}\otimes \zeta 1^{[1]}\in A\underline{\otimes} K$, and $c\# h\in \underline{A\# H}$, then
\beqnast
\llangle (a1^{[0]}\otimes \xi 1^{[1]})(b1^{[0']}\otimes \zeta 1^{[1']})|c\# h \rrangle 
& = & \llangle ab1^{[0]}\otimes \xi \zeta 1^{[1]}|c\# h \rrangle
 =  abc (h_{(1)}\cdot 1_A )\langle h_{(2)} ,\xi \zeta \rangle ,
\eqnast
while, on the other hand
\beqnast
&&\hspace{-.8cm}  \llangle a1^{[0]}\otimes \xi 1^{[1]} s( \llangle b1^{[0']}\otimes \zeta 1^{[1']} | c\# h_{(1)} \rrangle )| 1_A \# h_{(2)} \rrangle
 =  \llangle abc (h_{(1)} \cdot 1_A) 1^{[0]}  \otimes \xi 1^{[1]}  | 1_A \# h_{(3)} \rrangle  \langle h_{(2)} ,\zeta \rangle \\
&& =  \llangle abc (h_{(1)} \cdot 1_A) 1^{[0]}  \otimes \xi 1^{[1]}  | 1_A \# h_{(2)} \rrangle  \langle h_{(3)} ,\zeta \rangle 
 =  abc (h_{(1)} \cdot 1_A) \langle h_{(2)} ,\xi \rangle  \langle h_{(3)} ,\zeta \rangle \\
&& =  abc (h_{(1)}\cdot 1_A )\langle h_{(2)} ,\xi \zeta \rangle .
\eqnast
For (SP4), take $a1^{[0]}\otimes \xi 1^{[1]} \in A\underline{\otimes} K$, then
\[
\llangle a1^{[0]}\otimes \xi 1^{[1]} |1_A \# 1_H \rrangle = a\langle 1_H , \xi \rangle =a\epsilon (\xi ) =\te (a1^{[0]}\otimes \xi 1^{[1]}) .
\]
Finally, for (SP5) , take $a\# h \in \underline{A\# H}$, then
\[
\llangle 1^{[0]} \otimes 1^{[1]} |a\# h \rrangle =a(h_{(1)}\cdot 1_A )\langle h_{(2)} , 1_K \rangle =a(h_{(1)}\cdot 1_A)\epsilon (h_{(2)}) =a( h\cdot 1_A) =\ue (a\# h) .
\]

Therefore, the map $\llangle \, , \, \rrangle$ is a skew pairing between these two $A$-bialgebroids.
\end{proof}

\section{Partial module coalgebras}\selabel{modulecoalgebras}

\subsection{Definition and examples}

\begin{defi} Let $H$ be a bialgebra. A $k$ coalgebra $C$ is said to be a left partial $H$-module coalgebra if there is a linear map
\[
\begin{array}{rccc} \cdot : & H\otimes C & \rightarrow & C \\
\, & h\otimes c & \mapsto & h\cdot c 
\end{array}
\]
satisfying the following conditions:
\begin{enumerate}
\item[(PLHMC1)] For all $h\in H$ and $c\in C$, $\Delta (h\cdot c)= (h_{(1)} \cdot c_{(1)}) \otimes (h_{(2)} \cdot c_{(2)})$;
\item[(PLHMC2)] $1_H \cdot c =c$, for all $c\in C$;
\item[(PLHMC3)] For all $h,k\in H$ and $c\in C$, 
\[
h\cdot (k\cdot c) =  (hk_{(1)}\cdot c_{(1)}) \epsilon (k_{(2)}\cdot c_{(2)}) .
\]
\end{enumerate}
A partial module coalgebra is called symmetric, if, in addition, it satisfies
\begin{enumerate}
\item[(PLHMC3')] For all $h,k\in H$ and $c\in C$, 
\[
h\cdot (k\cdot c) =  \epsilon (k_{(1)}\cdot c_{(1)}) (hk_{(2)}\cdot c_{(2)}). 
\]
\end{enumerate}
\end{defi}

In a symmetric way, it is possible to define the notion of a right partial $H$-module coalgebra. We have the following immediate results.

\begin{prop} Let $H$ be a bialgebra and $C$ be a left partial $H$-module coalgebra, then, 
\begin{enumerate}[(i)]
\item for every $h\in H$ and $c\in C$ we have
\begin{equation}
\label{epsilon}
h\cdot c =\epsilon (h_{(1)}\cdot c_{(1)})(h_{(2)}\cdot c_{(2)}) =(h_{(1)}\cdot c_{(1)})\epsilon (h_{(2)}\cdot c_{(2)}) ;
\end{equation}
\item 
for every $h\in H$ and $c\in C$ we have
\begin{equation}
\label{epsilon2}
\epsilon (h\cdot c) =\epsilon (h_{(1)}\cdot c_{(1)})\epsilon (h_{(2)}\cdot c_{(2)});
\end{equation}
\item $C$ is a left $H$ module coalgebra if, and only if, for every $h\in H$ and $c\in C$ we have $\epsilon (h\cdot c )=\epsilon (h) \epsilon (c)$.
\end{enumerate}
\end{prop}

\begin{proof} \ul{(i)}. This follows imediatly from the axiom (PLHMC1). By applying $\epsilon \otimes I$, we obtain the first equality, and by applying $I\otimes \epsilon $ we obtain the second equality.\\
\ul{(ii)} follows from (i).\\
\ul{(iii)}.
Suppose that $\epsilon (h\cdot c )=\epsilon (h) \epsilon (c)$ for every $h\in H$ and $c\in C$, then, by (PLHMC3) we have
\beqnast
h\cdot (k\cdot c) & = &  (hk_{(1)}\cdot c_{(1)}) \epsilon (k_{(2)}\cdot c_{(2)})\\
& = &  (hk_{(1)}\cdot c_{(1)}) \epsilon (k_{(2)})\epsilon ( c_{(2)})\\ 
& = & hk \cdot c ,
\eqnast
therefore, $C$ is a left $H$-module coalgebra. The converse is trivial.
\end{proof}

\begin{remark}
It is important to notice that a different notion of ``partial module coalgebra'' was introduced earlier in the arXiv-version of \cite{caen06}. There, a coalgebra $C$ was said to be a (right) $H$-module coalgebra if and only if there exists a map $C\ot H\to H, c\ot h\mapsto c\cdot h$ such that the map $H\ot C\to C\to H, c\ot h\mapsto h_{(1)}\ot c\cdot h_{(2)}$ is a partial entwining structure. Clearly, this definition, as does ours, generalizes usual module coalgebras. However, the definition introduced in this note is motivated by the examples of partial actions of groups on coalgebras and by duality results with partial module algebras as illustrated below. 
\end{remark}

In the case of partial actions of a group on an algebra, we found that the partial actions of a group $G$ on an algebra $A$ that are in correspondence with partial actions of the Hopf algebra $kG$ on $A$, are those for which the ideals $A_g$ are generated by central idempotents $1_g$. The existence of a central idempotent ensures that there is
an algebra map $A\to A_g$ that is a left inverse for the inclusion map $A_g\to A$, which in turn is a multiplicative map; and ensures that $A\cong A_g\times A'_g$, for some other unitary algebra $A'_g$.
We will now prove a similar result for partial module coalgebras.

\begin{lemma}\lelabel{propertiesprojection}
For a coalgebra $C$ and $D$ with a coalgebra map $\iota:D\to C$, the following are equivalent:
\begin{enumerate}[(i)]
\item there exists a coalgebra $D'$ and a coalgebra map $\iota':D' \to C$ such that the universal morphism $J: D\coprod D'\to C$ is an isomorphism of coalgebras (where coproduct of $D$ and $D'$ is considered in the category of coalgebras);
\item there is a $k$-linear map $P:C\to D$ that satisfies 
\begin{itemize}
\item $P\circ \iota=id_D$,
\item  $\Delta \circ P =(P \otimes P )\circ \Delta$,
\item $\iota (P (c)) =c_{(1)} \epsilon_D (P (c_{(2)})) =\epsilon_D ( P (c_{(1)})) c_{(2)}$,  all $c\in C$
\end{itemize}
\item There is a $k$-linear map $P:C\to D$ such that
\begin{itemize}
\item $P\circ \iota=id_D$,
\item  $\Delta \circ P =(P \otimes P )\circ \Delta$,
\item $\Delta (\iota (P (c))) =c_{(1)}\ot \iota (P (c_{(2)})) =\iota (P (c_{(1)}))\ot  c_{(2)}=\iota (P(c_{(1)}))\ot \iota (P(c_{(2)}))$,  all $c\in C$
\end{itemize}
\end{enumerate}
\end{lemma}

\begin{proof}
$\ul{(i)\Rightarrow (ii)}$. Recall  (see e.g. \cite[Proposition 1.4.19]{DNR}) that the category of coalgebras has coproducts. In particular, if $D$ and $D'$ are two coalgebras, then the coproduct $D\coprod D'$ is computed by taking the direct sum of the underlying $k$-modules $D\oplus D'$ and endowing it with the following comultiplication and counit
$$\Delta(d,d')=(d_{(1)},0)\ot (d_{(2)},0)+(0,d'_{(1)})\ot (0,d'_{(2)})$$
$$\epsilon(d,d')=\epsilon_D(d)\epsilon_{D'}(d')$$
for all $d\in D, d'\in D'$. Remark that the comultiplication of the coproduct can also be written as $\Delta(d+d')=\Delta_D(d)+\Delta_{D'}(d')$
With this structure, it is easy to verify that the canonical projection $D\oplus D'\to D$ satisfies all stated conditions.\\
$\ul{(ii)\Rightarrow (iii)}$. We only have to prove the third condition. Take any $c\in C$ and use the fact that $\iota$ is a coalgebra morphism, then we find
\begin{eqnarray*}
\Delta (\iota (P (c))) &=& \iota( P(c)_{(1)}) \ot \iota (P(c)_{(2)}) = c_{(1)}\epsilon_D( P(c)_{(2)}) \ot \iota (P(c)_{(3)})\\
&=&c_{(1)} \ot \epsilon_D( P(c)_{(2)})\iota (P(c)_{(3)}) = c_{(1)}\ot \iota\circ P(c_{(2)}).\\
\end{eqnarray*}
$\ul{(iii)\Rightarrow (i)}$. Define $D'=\ker P$. Then clearly $C\cong D\oplus D'$ as $k$-modules. Moreover, using the second and third condition one can easily verify that $\ker P$ is a subcoalgebra of $C$. For any $c\in C$, we can write $c=P(c)+c-P(c)=d+d'$ where $d=P(c)$ and $d'=c-P(c)$. Then $\Delta(c)=\Delta(d)+\Delta(d')$, and since $D$ and $D'$ are subcoalgebras, we find that $C$ is indeed the coproduct of $D$ and $D'$. 
\end{proof}

We are now ready to define a partial action of a group $G$ on a coalgebra $C$. Clearly, such a partial action should in first place be a partial action of $G$ on the underlying set $C$. More precisely, we have the following definition.\\

\begin{defi}\delabel{partactionC} (Partial group actions on coalgebras) A partial action of a group $G$ on a coalgebra $C$ consists of a family of subcoalgebras $\{ C_{g} \}_{g\in G}$ of $C$ and coalgebra isomorphisms $\{ \theta_g :C_{g^{-1}} \rightarrow C_g \}_{g\in G}$, such that
\begin{enumerate}[(i)]
\item for every $g\in G$, the coalgebra $C_g$ is a subcoalgebra-direct summand of $C$, i.e. there exists a projection $P_g:C\to C_g$ satisfying the (equivalent) conditions of \leref{propertiesprojection}
\item $C_e=C$ and $\theta_e=P_e=id_C$, where $e$ is the unit of $C$;
\item For all $g,h\in G$, we the following equations hold
\begin{eqnarray}
P_h\circ P_g&=&P_g\circ P_h \eqlabel{Pcommute}\\
\theta_{h^{-1}} \circ P_{g^{-1}}\circ P_h &=& P_{(gh)^{-1}}\circ \theta_{h^{-1}} \circ P_{h} \eqlabel{Pthetacompatible}\\
\theta_g\circ \theta_h \circ P_{h^{-1}}\circ P_{(gh)^{-1}}&=&\theta_{gh}\circ P_{h^{-1}}\circ P_{(gh)^{-1}} \eqlabel{thetacomposition}
\end{eqnarray}
\end{enumerate}
\end{defi}

\begin{remark} The three equalities envolving the projections in \deref{partactionC} were motivated by dualizing a partial action of $G$ over an algebra $A$ in which every ideal $A_g$ is generated by a central idempotent in $A$, this is expressed in the equality \equref{Pcommute}. The equation \equref{Pthetacompatible} is basically saying that $\theta_{h^{-1}}(C_h \cap C_{g^{-1}})=C_{(gh)^{-1}}\cap C_{h^{-1}}$ which is the second axiom of partial action. The equality \equref{thetacomposition} reflects the fact that for any $x\in C_{(gh)^{-1}}\cap C_{h^{-1}}$, we have $\theta_g \circ \theta_h (x) =\theta_{gh} (x)$.
\end{remark}

\begin{thm}
Let $C$ be a $k$-coalgebra and $G$ a group. Then there is a bijective correspondence between partial actions of $G$ on the coalgebra $C$ and maps $kG\ot C\to C$ that turn $C$ into a symmetric partial $kG$-module coalgebra.
\end{thm}

\begin{proof}
Suppose that $C$ is a partial $kG$-module coalgebra. For any $g\in G$ we denote $\delta_g\in kG$ the corresponding base element. Then for any $c\in C$, we find
\begin{eqnarray*}
\delta_{g}\cdot (\delta_{g^{-1}}\cdot c)&=&\epsilon(\delta_{g^{-1}}\cdot c_{(1)}) \delta_g\delta_{g^{-1}}\cdot c_{(2)}= \epsilon(\delta_{g^{-1}}\cdot c_{(1)})  c_{(2)}=c_{(1)}\epsilon(\delta_{g^{-1}}\cdot c_{(2)})
\end{eqnarray*} 
where we used (PLHMC3) and (PLHMC3'). We then define $P_g:C\to C$, as 
\begin{equation}\label{pg}
 P_g(c)=\epsilon(\delta_{g^{-1}}\cdot c_{(1)})  c_{(2)}=c_{(1)}\epsilon (\delta_{g^{-1}} \cdot c_{(2)}),
\end{equation}
where we used the symmetry of the partial action to obtain the second expression. 
It is easy to see that these linear operators are projections, and we define $C_g={\rm Im}\, P_g$. Let us observe that $C_g$ is a subcoalgebra of $C$. 
\begin{eqnarray*}
\Delta (P_g(c)) & = & \Delta (\epsilon (\delta_g\cdot c_{(1)})  c_{(2)}) = 
\epsilon (\delta_g\cdot c_{(1)})  c_{(2)}\ot c_{(3)} 
= \epsilon(\delta_g\cdot c_{(1)})\epsilon(\delta_g\cdot c_{(2)})  c_{(3)}\ot  c_{(4)}\\
&=& \epsilon(\delta_g\cdot c_{(1)})c_{(2)}\epsilon(\delta_g\cdot c_{(3)})  \ot  c_{(4)}
= \epsilon(\delta_g\cdot c_{(1)})c_{(2)} \ot  
\epsilon(\delta_g\cdot c_{(3)}) c_{(4)}\\ 
&=& P_g(c_{(1)})\otimes P_g (c_{(2)}) \in C_g \otimes C_g .
\end{eqnarray*} 
For the counit condition, we find
\begin{eqnarray*}
P_g(c_{(1)})\epsilon (P_g (c_{(2)}))=\epsilon(\delta_g\cdot c_{(1)})  \epsilon(c_{(2)}) c_{(3)} = \epsilon (\delta_g\cdot c_{(1)}) c_{(2)} = P_g(c)
\end{eqnarray*}
and by symmetry we also have $\epsilon(P_g(c_{(1)}))P_g (c_{(2)})=P_g(c)$. 

Let us show that $P_g$ satisfies the conditions of \leref{propertiesprojection}(ii).
In order to see that 
$\Delta \circ P_g =(P_g \otimes P_g )\circ \Delta$ we use (\ref{pg}), indeed,
\beqnast
(P_g \otimes P_g )\circ \Delta (c) & = & P_g (c_{(1)}) \otimes P_g (C_{(2)}) =c_{(1)}\epsilon (\delta_{g^{-1}}\cdot c_{(2)})\otimes \epsilon (\delta_{g^{-1}} \cdot c_{(3)} ) c_{(4)} \\
& = & c_{(1)}\epsilon (\delta_{g^{-1}}\cdot c_{(2)})\otimes c_{(3)} =\epsilon (\delta_{g^{-1}}\cdot c_{(1)}) c_{(2)}\otimes c_{(3)} \\
& = & \Delta (\epsilon (\delta_{g^{-1}}\cdot c_{(1)}) c_{(2)}) =\Delta (P_g (c)) .
\eqnast
Finally, for all $g\in G$ and $c\in C$ we have
\[
\epsilon (P_g (c_{(1)})) c_{(2)}  =  \epsilon (\delta_{g^{-1}}\cdot c_{(1)})\epsilon (c_{(2)}) c_{(3)} =\epsilon (\delta_{g^{-1}}\cdot c_{(1)})c_{(2)}=P_g (c).
\]
Similarly, we can prove that $P_g (c) =c_{(1)}\epsilon (P_g (c_{(2)})) $. 

Let us check that the projections mutually commute, i.e. \equref{Pcommute} is satisfied. Take $c\in C$, then for all $g,h\in C$ we find, using \eqref{pg}
\begin{eqnarray*}
P_g\circ P_h(c) &=& P_g(c_{(1)}\epsilon(\delta_{h^{-1}}\cdot c_{(2)})) = \epsilon(\delta_{g^{-1}}\cdot c_{(1)(1)}) c_{(1)(2)}\epsilon(\delta_{h^{-1}}\cdot c_{(2)}) \\
&=& \epsilon(\delta_{g^{-1}}\cdot c_{(1)}) c_{(2)(1)}\epsilon(\delta_{h^{-1}}\cdot c_{(2)(2)})
= P_h\circ P_g (c).
\end{eqnarray*}

Now, define for any $g\in G$ a map $\theta_g :C_{g^{-1}} \rightarrow C_g$ by $\theta_g (P_{g^{-1}}(c)) =\delta_g \cdot P_{g^{-1}}(c)$. One can easily observe that $C_e=C$ and $\theta_e=id_C$.

Take $c\in C$ then,
\begin{eqnarray*}
P_{(gh)^{-1}}\circ \theta{h^{-1}} \circ P_h (c) & = & P_{(gh)^{-1}} (\delta_{h^{-1}} 
\cdot _l c_{(1)} ) \epsilon (\delta_{h^{-1}} \cdot c_{(2)}) =P_{(gh)^{-1}} (\delta_{h^{-1}} \cdot  c )\\
& = & (\delta_{h^{-1}}\cdot c_{(1)}) \epsilon (\delta_{gh} \cdot ( \delta_{h^{-1}} \cdot c_{(2)})) 
= (\delta_{h^{-1}}\cdot c_{(1)}) \epsilon (\delta_g \cdot c_{(2)}) 
\epsilon ( \delta_{h^{-1}} \cdot c_{(3)}) \\
& = & \epsilon (\delta_g \cdot c_{(1)})(\delta_{h^{-1}}\cdot c_{(2)}) 
\epsilon ( \delta_{h^{-1}} \cdot c_{(3)}) ,
\end{eqnarray*}
on the other hand
\beqnast
\theta_{h^{-1}} \circ P_{g^{-1}} \circ P_h (c) & = & \theta_{h^{-1}} \circ P_h \circ P_{g^{-1}} (c) =
\delta_{h^{-1}} \cdot ( P_h \circ P_{g^{-1}} (c))\\
& = & \delta_{h^{-1}} \cdot ( P_h (\epsilon (\delta_g \cdot c_{(1)} ) c_{(2)}))
 =  \delta_{h^{-1}} \cdot ( \epsilon (\delta_g \cdot c_{(1)} ) c_{(2)}  
\epsilon (\delta_{h^{-1}} \cdot c_{(3)} )) \\
& = & \epsilon (\delta_g \cdot c_{(1)} )(\delta_{h^{-1}} \cdot c_{(2)})
\epsilon (\delta_{h^{-1}} \cdot c_{(3)} ) .
\eqnast
So \equref{Pthetacompatible} is verified. Finally take any $c\in C$ and let us check \equref{thetacomposition}
\begin{eqnarray*}
\theta_g\circ \theta_h\circ P_{h^{-1}}\circ P_{(gh)^{-1}}(c)&=& 
\theta_g ( \delta_h \cdot (P_{h^{-1}} (P_{(gh)^{-1}} (c)))) = 
\theta_g ( \delta_h \cdot (P_{h^{-1}} (c_{(1)} \epsilon (\delta_{gh}\cdot c_{(2)}))))\\
& = & \theta_g (\delta_h \cdot (c_{(1)} \epsilon (\delta_h \cdot c_{(2)} ) 
\epsilon (\delta_{gh}\cdot c_{(3)}))) \\
& = & \theta_g (\delta_h \cdot (c_{(1)}  
\epsilon (\epsilon (\delta_h \cdot c_{(2)} )(\delta_{gh}\cdot c_{(3)})))) \\
& = & \theta_g ((\delta_h \cdot c_{(1)})  
\epsilon (\delta_g \cdot (\delta_h \cdot c_{(2)} ))) = \theta_g (P_{g^{-1}} (\delta_h \cdot c))\\
& = & \delta_g \cdot (P_{g^{-1}} (\delta_h \cdot c)) =(\delta_g \cdot (\delta_h \cdot c_{(1)})) \epsilon (\delta_g \cdot (\delta_h \cdot c_{(2)}))\\
& = & (\delta_{gh} \cdot c_{(1)}) \epsilon (\delta_h \cdot c_{(2)}) \epsilon (\delta_g \cdot (\delta_h \cdot c_{(3)})) \\
& = & (\delta_{gh} \cdot c_{(1)}) \epsilon (\delta_g \cdot (\delta_h \cdot c_{(2)})) = (\delta_{gh} \cdot c_{(1)})\epsilon (\delta_{gh} \cdot c_{(2)}) 
\epsilon (\delta_h \cdot c_{(3)}) \\
& = & \theta_{gh} (P_{(gh)^{-1}}(c_{(1)} \epsilon (\delta_h \cdot c_{(2)}))) = 
\theta_{gh} \circ P_{(gh)^{-1}} \circ P_{h^{-1}} (c) .
\end{eqnarray*}

Conversely, suppose that we have a partial action $( \{ C_g \}_{g\in G} , \{ \theta_g :C_{g^{-1}} \rightarrow C_g \}_{g\in G} )$ of $G$ on $C$. Then we define a linear map $\cdot: kG \ot C \to C$ by $\delta_g \cdot c =\theta_g (P_{g^{-1}}(c))$. \\
Thanks to axiom (ii) in \deref{partactionC}, we find for any $c\in C$
\[
1_{kG} \cdot c = \delta_e \cdot c =c .
\]
As both $\theta_g$ and $P_{g^{-1}}$ are comultiplicative, it follows that 
\[
\Delta (\delta_g \cdot c ) =\Delta(\theta_g (P_{g^{-1}}(c)))=\theta_g (P_{g^{-1}}(c_{(1)}))\ot \theta_g (P_{g^{-1}}(c_{(2)}))=\delta_g \cdot c_{(1)} \otimes \delta_g \cdot c_{(2)} .
\]
Next, let us first remark that 
\begin{equation}\epsilon(\delta_g\cdot c_{(1)})c_{(2)}=\epsilon( \theta_g \circ P_{g^{-1}}(c_{(1)}))c_{(2)}
 = \epsilon( P_{g^{-1}}( c_{(1)}))c_{(2)} = P_{g^{-1}}(c),\eqlabel{backandforth}
 \end{equation}
 where we used that $\theta_g$ is a coalgebra morphism in the second equality and the properties of the projection $P_{g^{-1}}$ (see \leref{propertiesprojection}) in the third equality.

For any $g,h\in G$ and $c\in C$ we have
\begin{eqnarray*}
\delta_g \cdot (\delta_h \cdot c) &=& \theta_g \circ P_{g^{-1}} \circ \theta_h \circ P_{h^{-1}} (c) 
 =  \theta_g \circ \theta_h \circ P_{(gh)^{-1}} \circ P_{h^{-1}} (c) \\ 
 &=&  \theta_{gh} \circ P_{(gh)^{-1}} \circ P_{h^{-1}} (c)
= \delta_{gh} \cdot (P_{h^{-1}}(c)) \\
&=& \delta_{gh} \cdot c_{(1)}\epsilon (\delta_h\cdot c_{(2)})
\end{eqnarray*}
Here we used \equref{Pthetacompatible} in the second equality, \equref{thetacomposition} in the third equality and \equref{backandforth} in the last equality. Therefore $C$ is a partial $kG$-module coalgebra.

We leave the verification that both constructions are mutual inverses to the reader.
\end{proof}

\begin{exmp} (Induced partial actions) Let $C$ be a left $H$-module coalgebra with the action $\triangleright :H\otimes C\rightarrow C$. Consider $D\subseteq C$ a subcoalgebra with a $k$-linear projection $P:C\to D$ that is comultiplicative and satisfies
\begin{equation}
\label{symmetryofprojection}
P(c) =c_{(1)}\epsilon (P(c_{(2)})) =\epsilon (P(c_{(1)}))c_{(2)},
\end{equation}
for all $c\in C$ (i.e.\ $D$ satisfies the equivalent condition of \leref{propertiesprojection}). Then the linear map defined by
\[
\begin{array}{rccc} \cdot : & H\otimes D & \rightarrow & D \\
\, & h\otimes d & \mapsto & h\cdot d = P(h\triangleright d)
\end{array}
\]
is a symmetric partial action, turning $D$ into a symmetric left partial $H$-module coalgebra.

Indeed, for proving the item (PLHMC1), consider $d\in D$, then
\[
1_H \cdot c =P(1_H \triangleright d )=P(d) =d .
\]
Now, for the item (PLHMC2), consider $h\in H$ and $d\in D$, then
\begin{eqnarray*}
\Delta (h\cdot d ) & = & \Delta (P(h\triangleright d))=(P\otimes P)\circ 
\Delta (h\triangleright d) \\
& = & (P\otimes P)((h_{(1)}\triangleright c_{(1)})\otimes (h_{(2)}\triangleright c_{(2)})) \\
& = & (h_{(1)}\cdot c_{(1)})\otimes (h_{(2)}\cdot c_{(2)}) .
\end{eqnarray*}
Finally, for the item (PLHMC3), consider $h,k\in H$ and $d\in D$, then
\begin{eqnarray*}
h\cdot (k\cdot d) & = & P(h\triangleright (P(k\triangleright d))) 
 =  P(h\triangleright (k_{(1)}\triangleright c_{(1)}))\epsilon (P(k_{(2)}\triangleright c_{(2)})) \\
& = & P(hk_{(1)}\triangleright c_{(1)}) \epsilon (P(k_{(2)}\triangleright c_{(2)})) 
 =  (hk_{(1)}\cdot c_{(1)}) \epsilon (k_{(2)}\cdot c_{(2)}) .
\end{eqnarray*}
where we used the definition of the partial action $\cdot$ in the first and last equality, \eqref{symmetryofprojection} in the second equality and the fact that $C$ is an $H$-module coalgebra in the third equality.

The symmetry of the action follows easily from the symmetry in the identity \eqref{symmetryofprojection}.
\end{exmp}

\subsection{The $C$-ring associated to a partial module coalgebra}

As we have seen, if $A$ is a right $H$ comodule algebra, then the space $A\underline{\otimes} H$ has a structure of an $A$ coring. In the same way if $C$ is a left partial $H$-module coalgebra, then a subspace of $H\otimes C$, can be endowed with a $C$-ring structure.

\begin{defi} Let $C$ be a coalgebra over a field, a $C$-ring is a monoid in the monoidal category of $C$-bicomodules
$( {}^C \mathcal{M}^C , \square^C , C)$. 
\end{defi}

\begin{prop}\prlabel{Cring} 
Let $H$ be a bialgebra and $C$ be a left symmetric partial $H$ module coalgebra, then the subspace
\[
\underline{H\otimes C} =\{ \underline{h\otimes c} =\epsilon (h_{(1)}\cdot c_{(1)}) h_{(2)} \otimes c_{(2)} \in H\otimes C \} ,
\]
has a structure of a $C$ ring.
\end{prop}

\begin{proof} First it is important to note that, for every $h\in H$ and $c\in C$ we have \begin{equation}
\label{redundancy}
\underline{h\otimes c}=\underline{\epsilon (h_{(1)}\cdot c_{(1)}) h_{(2)} 
\otimes c_{(2)}} ,
\end{equation} 
this follows easily from (\ref{epsilon2}). 

The left $C$-comodule structure on $\underline{H\otimes C}$ is given by
\[
\lambda: \ul{H\ot C}\to C\ot \ul{H\ot C},\quad \lambda (\underline{h\otimes c}) =h_{(1)}\cdot c_{(1)} \otimes \underline{ h_{(2)} \otimes c_{(2)}}.
\]
By the axiom (PLHMC1), we can see that $(I\otimes \lambda )\circ \lambda =(\Delta \otimes I)\circ \lambda$, and because of (\ref{redundancy}), we obtain $(\epsilon \otimes I)\circ \lambda =I$.

The right comodule structure, in its turn, is given by $\rho:\ul{H\ot C}\to \ul{H\ot C}\ot C,\ 
\rho (\underline{h\otimes c}) =\underline{h\otimes c_{(1)}}\otimes c_{(2)}$,
which satisfies trivially the equalities $(\rho \otimes I)\circ \rho =(I\otimes \Delta )\circ \rho$ and $(I\otimes \epsilon )\circ \rho =I$. Therefore $\underline{H\otimes C}\in {}^C \mathcal{M}^C$.

The cotensor product $\underline{H\otimes C} \square^C \underline{H\otimes C}$, being the equalizer of the morphisms $\rho \otimes I$ and $I\otimes \lambda$, can be characterized as the subspace of $\underline{H\otimes C} \otimes \underline{H\otimes C}$ spanned by elements
\[
\sum_i \underline{ h^i \otimes c^i }\otimes \underline{ k^i \otimes d^i },
\]
satisfying
\begin{equation}
\label{cotensor}
\sum_i \underline{ h^i \otimes c^i_{(1)} }\otimes c^i_{(2)} \otimes \underline{ k^i \otimes d^i } =
\sum_i \underline{ h^i \otimes c^i }\otimes k^i_{(1)}\cdot d^i_{(1)} \otimes 
\underline{ k^i_{(2)} \otimes d^i_{(2)} }.
\end{equation}
The multiplication map, $\mu : \underline{H\otimes C} \square^C \underline{H\otimes C} \rightarrow \underline{H\otimes C}$ is defined as
\[
\sum_i (\underline{ h^i \otimes c^i })(\underline{ k^i \otimes d^i })= 
\sum_i \underline {\epsilon (h^i_{(1)} \cdot c^i )
\epsilon (k^i_{(1)} \cdot d^i_{(1)} ) h^i_{(2)}k^i_{(2)} \otimes d^i_{(2)} }
\]
One can check that $\mu$ is an associative $C$-bimodule map.
The unit map $\eta :C\rightarrow \underline{H\otimes C}$ is given by $\eta (c) =\underline{1_H \otimes c}$. We can see that it is a bi-comodule map. For the left side, we have
\beqnast
\lambda \circ \eta (c) & = & \lambda (\underline{1_H \otimes c}) 
 =  1_H \cdot c_{(1)} \otimes \underline{1_H \otimes c_{(2)}} \\
& = & c_{(1)} \otimes \underline{1_H \otimes c_{(2)}} 
 =  c_{(1)} \otimes \eta ( c_{(2)}) \\
& = & (I\otimes \eta ) (c_{(1)} \otimes c_{(2)}) 
 =  (I\otimes \eta ) \circ \Delta (c) ,
\eqnast
and for the right side,
\beqnast
\rho \circ \eta (c) & = & \rho (\underline{1_H \otimes c}) 
 =  \underline{1_H \otimes c_{(1)}} \otimes c_{(2)}\\
& = & \eta (c_{(1)} )\otimes  c_{(2)} 
 =  (\eta \otimes I) (c_{(1)} \otimes c_{(2)}) \\
& = & (\eta \otimes I) \circ \Delta (c) .
\eqnast
The images of 
$(\eta \otimes I)\circ \lambda$ and $(I\otimes \eta )\circ \rho$ lie in the co-tensor product $\underline{H\otimes C} \square^C \underline{H\otimes C}$. Indeed, applying the bi-colinearity of $\eta$, we find 
\beqnast
(\rho \otimes I)\circ (\eta \otimes I)\circ \lambda 
&=& (\eta\ot I\ot I)\circ (\Delta\ot I)\circ \lambda = (\eta\ot I\ot I)\circ (I\ot \lambda)\circ \lambda\\
&=&(I\ot \lambda)\circ (\eta\ot I)\circ \lambda
\eqnast
with a similar proof for $(I\otimes \eta )\circ \rho$. It now makes sense to consider the unit axioms for the $C$-ring, i.e.
\begin{eqnarray*}
\mu \circ (\eta \otimes I)\circ \lambda = I= \mu \circ (I\otimes \eta ) \circ \rho =I
\end{eqnarray*}
whose verification is left to the reader.
\end{proof}

Finally, there exist a duality between the $C$-rings obtained from partial module coalgebras and $A$-corings defined from partial comodule algebras. More precisely, we have,

\begin{prop} Let $H$ and $K$ be two bialgebras with a dual pairing $\langle \, ,\, \rangle : K\otimes H \rightarrow k$. Consider a left $H$-comodule algebra $A$, with left partial coaction $\lambda :A\rightarrow H\otimes A$, and  a left partial $K$-module coalgebra $C$. Suppose that there is a dual pairing $\left( \, , \, \right) :C\otimes A\rightarrow k$, such that for every $x\in C$, $\xi \in K$ and $a\in A$ we have $\left( \xi \cdot x , a\right) =\langle \xi , a^{[-1]} \rangle \left( x, a^{[0]} \right)$. Then there is a dual pairing between the $A$-coring $H\underline{\otimes}A =\lambda (1_A )(H\otimes A)$ and the $C$-ring $\underline{K\otimes C}$.
\end{prop}

\subsection{Dualities}

\subsubsection*{Module coalgebras versus module algebras}

\begin{thm}\thlabel{modulevscomodule}
Suppose that we have two linear maps $\cdot_l :H\otimes C \rightarrow C$ and $\cdot_r :A\otimes H \rightarrow A$ related by the following equality
\begin{equation}\label{moduleversuscomodule}
(a\cdot_r h,c)=(a,h\cdot_l c) ,
\end{equation}
for all $a\in A$, $h\in H$ and $c\in C$. Then $(C,\cdot_l)$ is a (symmetric) left partial $H$-module coalgebra, if, and only if, $(A,\cdot_r)$ is a (symmetric) right partial $H$-module algebra. 

In particular, for any (symmetric) left $H$-module coalgebra $C$, the linear dual $A=C^*$ is a (symmetric) right partial $H$-module algebra, as (\ref{moduleversuscomodule}) then becomes a definition for the partial action.
\end{thm}

\begin{proof}
Let us check that the morphism $\cdot_r:A\ot H\to A$ defines a partial right module algebra structure on $A$ if and only if the morphism $\cdot_l:H\ot C\to C$ defines a partial left module coalgebra structure on $C$. To this end, it suffices to check that, taking arbitrary $a\in A$ and $c\in C$, the (right handed versions of) axioms (PLA1)-(PLA4) correspond to axioms (PLHMC1)-(PLHMC4). Let us check this for one axiom, leaving the others for the reader. Suppose that $A$ is right partial $H$ module algebra, then
\begin{eqnarray*}
(a,h\cdot_l (k\cdot_l c))&=& ((a\cdot_r h)\cdot_r k , c)= ((a\cdot_r hk_{(1)})(1_A\cdot_r k_{(2)}), c)\\
& =& ((a\cdot_r hk_{(1)}) , c_{(1)})((1_A\cdot_r k_{(2)}), c_{(2)}) =(a, hk_{(1)}\cdot_l c_{(1)})(1_A , k_{(2)}\cdot_l c_{(2)})\\
& = & (a, hk_{(1)}\cdot_l c_{(1)})\epsilon ( k_{(2)}\cdot_l c_{(2)}) = (a, (hk_{(1)}\cdot_l c_{(1)})\epsilon ( k_{(2)}\cdot_l c_{(2)}))
\end{eqnarray*}
for any $a\in A$ then $(k\cdot_l c)=(hk_{(1)}\cdot_l c_{(1)})\epsilon ( k_{(2)}\cdot_l c_{(2)})$, therefore, $C$ is a left partial $H$ module coalgebra.
\end{proof}

\subsubsection*{Module coalgebras versus comodule algebras}

\begin{thm}
Let $\bk{-,-}:H\ot K\to k$ be a dual pairing between the Hopf algebras $H$ and $K$ and $(-,-):A\ot C\to k$ be a non-degenerate dual pairing between the coalgebra $C$ and the algebra $A$, which is a symmetric left partial $K$-comodule algebra with partial coaction $\lambda :A\to K\ot A,\ \lambda (a)=a^{[-1]}\ot a^{[0]}$. Suppose that we have a linear map $\cdot : H\otimes C\rightarrow C$ satisfying  
\[
(a,h\cdot c)=\bk{h,a^{[-1]}}(a^{[0]},c) ,
\]
for all $a\in A$, $c\in C$ and $h\in H$. Then the map $\cdot$ endows $C$ with a structure of a symmetric partial left $H$-module coalgebra. 
\end{thm}

\begin{proof} 
Let us check that $\cdot$ satifies axiom (PHMC3'). Consider $a\in A$ , $c\in C$ and $h,k\in H$, then
\beqnast
& \, & \left( a , h\cdot (k\cdot c) \right)  =  \langle h, a^{[-1]} \rangle \left( a^{[0]} , k\cdot c \right) = \langle h, a^{[-1]} \rangle \langle k, a^{[0][-1]} \rangle \left( a^{[0][0]} , c \right) \\
& = & \langle h, {a^{[-1]}}_{(1)} \rangle \langle k, 1^{[-1]} {a^{[-1]}}_{(2)} \rangle \left( 1^{[0]}a^{[0]} , c \right) \\
& = & \langle h, {a^{[-1]}}_{(1)} \rangle \langle k_{(1)}, 1^{[-1]} \rangle \langle k_{(2)} , {a^{[-1]}}_{(2)} \rangle \left( 1^{[0]} c_{(1)} \right) \left( a^{[0]} , c_{(2)} \right) \\
& = & \left( 1, k_{(1)} \cdot  c_{(1)} \right) \langle hk_{(2)} , a^{[-1]} \rangle  \left( a^{[0]} , c_{(2)} \right) = \epsilon ( k_{(1)} \cdot  c_{(1)} ) \left( a , hk_{(2)} \cdot c_{(2)} \right)\\
& = & \left( a , \epsilon ( k_{(1)} \cdot  c_{(1)} ) (hk_{(2)} \cdot c_{(2)}) \right) ,
\eqnast
for every $a\in A$, then we have $h\cdot (k\cdot c) =\epsilon ( k_{(1)} \cdot  c_{(1)} ) (hk_{(2)} \cdot c_{(2)})$. The other axioms are verified similarly. Therefore, $C$ is a symmetric left partial $H$-module coalgebra.
\end{proof}

\section{Partial comodule coalgebras and partial co-smash coproducts}\selabel{comodulecoalgebras}

\subsection{Definition and examples}

\begin{defi}
A left partial coaction of a Hopf algebra $H$ on a $k$-coalgebra $C$ is a linear map 
\[
\begin{array}{rccc} \lambda : & C & \rightarrow & H\otimes C\\
\, & c & \mapsto & c^{[-1]}\otimes c^{[0]}
\end{array}
\] 
such that
\begin{enumerate}
\item[(PLHCC1)] $(I\otimes \Delta)\circ \lambda (c)={c_{(1)}}^{[-1]} {c_{(2)}}^{[-1]} \otimes {c_{(1)}}^{[0]} \otimes {c_{(2)}}^{[0]}$, for all $c\in C$;
\item[(PLHCC2)] $(\epsilon \otimes I)\circ \lambda (c) =c$, for all $c\in C$;
\item[(PLHCC3)] for all $c\in C$ we have 
\[
(I\otimes \lambda)\circ \lambda (c)  =  {c_{(1)}}^{[-1]} 
\epsilon ({c_{(1)}}^{[0]}) {{c_{(2)}}^{[-1]}}_{(1)} \otimes {{c_{(2)}}^{[-1]}}_{(2)} \otimes {c_{(2)}}^{[0]} .
\]
\end{enumerate}
The coalgebra $C$ is called a left partial $H$-comodule coalgebra. If, in addition, the partial coaction satisfies the condition,
\begin{enumerate} 
\item[(PLHCC3')] for all $c\in C$ we have 
\[
(I\otimes \lambda)\circ \lambda (c)  ={{c_{(1)}}^{[-1]}}_{(1)} {c_{(2)}}^{[-1]} 
\epsilon ({c_{(2)}}^{[0]})\otimes {{c_{(1)}}^{[-1]}}_{(2)} \otimes {c_{(1)}}^{[0]}.
\]
\end{enumerate}
then the partial coaction is said to be symmetric.
\end{defi}

It is easy to see that any left $H$-comodule coalgebra is a left partial $H$-comodule coalgebra. Indeed, the axioms (PLHCC1) and (PLHCC2) are the same as the classic case. For the axiom (PLHCC3), if 
\begin{equation}
\label{classicalcase4}
c^{[-1]}\epsilon (c^{[0]})=\epsilon (c)1_H
\end{equation}
for any $c\in C$ then, combining this with axiom (PLHCC3), we find
\beqnast
(I\otimes \lambda )\circ \lambda (c) &=&
   {c_{(1)}}^{[-1]} \epsilon ({c_{(1)}}^{[0]}) {{c_{(2)}}^{[-1]}}_{(1)} \otimes {{c_{(2)}}^{[-1]}}_{(2)} \otimes {c_{(2)}}^{[0]}  \\
& = & \epsilon (c_{(1)}){{c_{(2)}}^{[-1]}}_{(1)} \otimes {{c_{(2)}}^{[-1]}}_{(2)} \otimes {c_{(2)}}^{[0]}  
 =  {c^{[-1]}}_{(1)} \otimes {c^{[-1]}}_{(2)} \otimes c^{[0]}  \\
& =&  (\Delta \otimes I)\circ \lambda (c) 
\eqnast
This shows that a left partial $H$-comodule coalgebra satisfying 
(\ref{classicalcase4}) is in fact a (global) left $H$-comodule coalgebra.
 
In \cite{wang}, the author gave a definition of left partial $H$-comodule coalgebra, but in the last axiom, the author considered only the nonsymmetric version. As we have learned so far, the most interesting properties of partial actions can be found when we consider more symmetric versions of the actions and coactions. 

Analogously, we can have the definition of a right partial $H$-comodule coalgebra.

\begin{exmp} An immediate example of a left partial $H$ comodule coalgebra is $C=D/I$ where $D$ is a left $H$-comodule coalgebra with coaction $\delta :D\rightarrow H\otimes D$, denoted by $\delta (d)=d^{(-1)}\otimes d^{(0)}$, and $I$ is a right coideal of $D$ such that $C$ is a coalgebra. The partial coaction is given by $\lambda (\bar{d})  ={d_{(2)}}^{(-1)} \otimes \epsilon_C (\overline{d_{(1)}})\overline{{d_{(2)}}^{(0)}}$ as shown in \cite{wang}. 
\end{exmp}

\begin{lemma}\lelabel{comodulecoalgebralemma} Let $H$ be a bialgebra and $C$ be a left partial $H$-comodule coalgebra, with coaction $\lambda :C\rightarrow H\otimes C$ denoted by $\lambda (c)=c^{[-1]} \otimes c^{[0]}$. Then, for every $c\in C$ we have
\beqnast
{c_{(1)}}^{[-1]}\epsilon ({c_{(1)}}^{[0]}) {c_{(2)}}^{[-1]} \otimes {c_{(2)}}^{[0]} 
& = & {c_{(1)}}^{[-1]}{c_{(2)}}^{[-1]}\epsilon ({c_{(2)}}^{[0]}) \otimes {c_{(1)}}^{[0]} \\
& =& c^{[-1]} \otimes c^{[0]} .
\eqnast
\end{lemma}

\begin{proof} First, by the item (i) of the definition of a left partial $H$-comodule coalgebra, we have 
\begin{equation}
\label{coproductcoaction}
c^{[-1]} \otimes {c^{[0]}}_{(1)} \otimes {c^{[0]}}_{(2)} ={c_{(1)}}^{[-1]} {c_{(2)}}^{[-1]} \otimes {c_{(1)}}^{[0]} \otimes {c_{(2)}}^{[0]} .
\end{equation}
Applying $(I\otimes \epsilon \otimes I)$ on both sides of (\ref{coproductcoaction}), we have
\[
c^{[-1]} \otimes \epsilon ({c^{[0]}}_{(1)}) {c^{[0]}}_{(2)} ={c_{(1)}}^{[-1]} {c_{(2)}}^{[-1]} \otimes \epsilon ({c_{(1)}}^{[0]} ) {c_{(2)}}^{[0]} ,
\]
therefore
\[
c^{[-1]} \otimes c^{[0]} = {c_{(1)}}^{[-1]}\epsilon ({c_{(1)}}^{[0]}) {c_{(2)}}^{[-1]} \otimes {c_{(2)}}^{[0]} .
\]
For the second equality, just apply $(I\otimes I\otimes \epsilon)$ on (\ref{coproductcoaction}).
\end{proof}

\begin{cor}\label{thisisit} Let $H$ be a bialgebra and $C$ be a left partial $H$-comodule coalgebra, with coaction $\lambda :C\rightarrow H\otimes C$ denoted by $\lambda (c)=c^{[-1]} \otimes c^{[0]}$. Then the map $\psi : C\rightarrow H$ given by $\psi (c) =c^{[-1]} 
\epsilon (c^{[0]})$ is an idempotent in the convolution algebra $\mbox{Hom}_k (C,H)$.
\end{cor}

\begin{proof} If we apply $(I\otimes \epsilon)$ to any of the equalities obtained in \leref{comodulecoalgebralemma}, then we find
\[
\psi (c) =c^{[-1]} \epsilon (c^{[0]}) = {c_{(1)}}^{[-1]}\epsilon ({c_{(1)}}^{[0]}) {c_{(2)}}^{[-1]} \epsilon ({c_{(2)}}^{[0]}) =\psi (c_{(1)})\psi (c_{(2)})=\psi *\psi (c) ,
\]
which shows that $\psi$ is idempotent with relation to the convolution product.
\end{proof}

\subsection{The partial cosmash coproduct}

Let $C$ be a left partial $H$-comodule coalgebra. Consider 
the subspace $\underline{C\cosmash H}$ of the tensor product $C\otimes H$ 
 spanned by elements of the form
\[
c\cosmash h =c_{(1)}\otimes {c_{(2)}}^{[-1]} \epsilon( {c_{(2)}}^{[0]}) h .
\]
By Corollary \ref{thisisit} it is easy to see that 
\[
c\cosmash h =c_{(1)}\cosmash {c_{(2)}}^{[-1]} \epsilon( {c_{(2)}}^{[0]}) h .
\]

\begin{prop} Let $H$ be a Hopf algebra and $C$ be a left partial $H$-comodule coalgebra. Then the space $\underline{C\cosmash H}$ is a coalgebra with the comultiplication given by 
\[
\HDelta (c\cosmash h) =c_{(1)}\cosmash {c_{(2)}}^{[-1]} h_{(1)}\otimes {c_{(2)}}^{[0]} \cosmash h_{(2)} ,
\]
and counit given by
\[
\Hepsilon (c\cosmash h) =\epsilon_C (c) \epsilon_H (h) .
\]
This coalgebra is will be called partial co-smash coproduct.
\end{prop}

\begin{proof} Let us check the counit axioms. First, we have
\begin{eqnarray*}
(I\otimes \Hepsilon )\circ \HDelta  (c\cosmash h) & = & (I\otimes \Hepsilon ) (c_{(1)}\cosmash {c_{(2)}}^{[-1]} h_{(1)}\otimes {c_{(2)}}^{[0]} \cosmash h_{(2)} ) \\
& = & c_{(1)}\cosmash {c_{(2)}}^{[-1]} \epsilon ( {c_{(2)}}^{[0]}) h_{(1)} 
\epsilon ( h_{(2)}) \\
& = & c_{(1)}\cosmash {c_{(2)}}^{[-1]} \epsilon( {c_{(2)}}^{[0]}) h 
 =  c\cosmash h .
\end{eqnarray*}
And on the other hand
\begin{eqnarray*}
(\Hepsilon \otimes I)\circ \HDelta  (c\cosmash h) & = & 
( \Hepsilon \otimes I) (c_{(1)}\cosmash {c_{(2)}}^{[-1]} h_{(1)}\otimes {c_{(2)}}^{[0]} \cosmash h_{(2)} ) \\
& = & \epsilon (c_{(1)})\epsilon ( {c_{(2)}}^{[-1]}) {c_{(2)}}^{[0]} \cosmash 
\epsilon ( h_{(1)}) h_{(2)} \\
& = & \epsilon (c_{(1)} ) c_{(2)} \cosmash h 
 =  c\cosmash h .
\end{eqnarray*}

For the coassociativity, on one hand, we have
\begin{eqnarray*}
& \,  & (\HDelta \otimes I)\circ \HDelta (c\cosmash h)  = 
(\HDelta \otimes I)(c_{(1)}\cosmash {c_{(2)}}^{[-1]} h_{(1)}\otimes {c_{(2)}}^{[0]} \cosmash h_{(2)} ) \\
& = & c_{(1)}\cosmash {c_{(2)}}^{[-1]} {{c_{(3)}}^{[-1]}}_{(1)} h_{(1)}\otimes {c_{(2)}}^{[0]} \cosmash {{c_{(3)}}^{[-1]}}_{(2)} h_{(2)}  \otimes {c_{(3)}}^{[0]} \cosmash h_{(3)} .
\end{eqnarray*}
On the other hand
\begin{eqnarray*}
& \,  & (I\otimes \HDelta )\circ \HDelta (c\cosmash h)  = 
(I\otimes \HDelta )(c_{(1)}\cosmash {c_{(2)}}^{[-1]} h_{(1)}\otimes {c_{(2)}}^{[0]} \cosmash h_{(2)} ) \\
& = & c_{(1)}\cosmash {c_{(2)}}^{[-1]} h_{(1)} \otimes {{c_{(2)}}^{[0]}}_{(1)} \cosmash 
{{{c_{(2)}}^{[0]}}_{(2)}}^{[-1]} h_{(2)} \otimes {{{c_{(2)}}^{[0]}}_{(2)}}^{[0]} \cosmash h_{(3)} \\
& = & c_{(1)}\cosmash {c_{(2)}}^{[-1]} {c_{(3)}}^{[-1]}h_{(1)} \otimes {c_{(2)}}^{[0]} \cosmash {c_{(3)}}^{[0][-1]} h_{(2)} \otimes {c_{(3)}}^{[0][0]} \cosmash h_{(3)} \\
& = & c_{(1)}\cosmash {c_{(2)}}^{[-1]} {c_{(3)}}^{[-1]} \epsilon ({c_{(3)}}^{[0]}) {{c_{(4)}}^{[-1]}}_{(1)} h_{(1)} \otimes  {c_{(2)}}^{[0]} \cosmash {{c_{(4)}}^{[-1]}}_{(2)} h_{(2)} \otimes 
 {c_{(4)}}^{[0]} \cosmash h_{(3)} \\
& = & c_{(1)}\cosmash {c_{(2)}}^{[-1]} \epsilon ({{c_{(2)}}^{[0]}}_{(2)}){{c_{(3)}}^{[-1]}}_{(1)} h_{(1)} \otimes {{c_{(2)}}^{[0]}}_{(1)} \cosmash {{c_{(3)}}^{[-1]}}_{(2)} h_{(2)} \otimes 
 {c_{(3)}}^{[0]} \cosmash h_{(3)} \\
& = & c_{(1)}\cosmash {c_{(2)}}^{[-1]} {{c_{(3)}}^{[-1]}}_{(1)} h_{(1)}\otimes {c_{(2)}}^{[0]} \cosmash {{c_{(3)}}^{[-1]}}_{(2)} h_{(2)}  \otimes {c_{(3)}}^{[0]} \cosmash h_{(3)} ,
\end{eqnarray*}
where in the third equality we used the axiom (i), in the fourth equality, the axiom (iv), in the fifth equality, we used the axiom (i) again and in the sixth equality, only the counit axiom.
\end{proof}

\subsection{Dualities}

\subsubsection*{Comodule coalgebras versus comodule algebras}

\begin{thm}
Let $H$ be a Hopf algebra and $(-,-):A\ot C\to k$ be a non-degenerate dual pairing between the coalgebra $C$ and the algebra $A$. 
Suppose that there are two linear maps $\lambda : C\rightarrow H\otimes C$, and $\rho :A\rightarrow A\otimes H$, given respectively by $\lambda (c)=c^{[-1]}\otimes c^{[0]}$ and $\rho (a)=a^{[0]}\otimes a^{[1]}$ and satisfying the relation 
\begin{equation}
\eqlabel{AcomodCcomod}
\left( a^{[0]} ,c \right) a^{[1]} =c^{[-1]} \left( a, c^{[0]} \right) ,
\end{equation}
for all $a\in A$ and $c\in C$. Then $A$ is a right $H$ comodule algebra if, and only if $C$ is a left $H$ comodule coalgebra.
\end{thm}

\begin{proof} 
Suppose that $A$ is a symmetric right $H$ comodule algebra, and consider $a\in A$ and $c\in C$, then
\beqnast
 c^{[-1]} \otimes c^{[0][-1]} \left( a, c^{[0][0]} \right)  =  c^{[-1]} \otimes \left( a^{[0]}, c^{[0]} \right) a^{[1]} =\left( a^{[0][0]}, c \right)  a^{[0][1]} \otimes a^{[1]} \\
&&\hspace{-10cm} =  \left( a^{[0]}1^{[0]}, c \right)  {a^{[1]}}_{(1)} 1^{[1]} \otimes {a^{[1]}}_{(2)} =\left( a^{[0]} , c_{(1)} \right) \left( 1^{[0]}, c_{(2)} \right)  {a^{[1]}}_{(1)} 1^{[1]} \otimes {a^{[1]}}_{(2)} \\
&&\hspace{-10cm} =  \left( a^{[0]} , c_{(1)} \right) {a^{[1]}}_{(1)} {c_{(2)}}^{[-1]}\left( 1, {c_{(2)}}^{[0]} \right)    \otimes {a^{[1]}}_{(2)} \\
&&\hspace{-10cm} =  {{c_{(1)}}^{[-1]}}_{(1)} \left( a , {c_{(1)}}^{[0]} \right) {a^{[1]}}_{(1)} {c_{(2)}}^{[-1]} \epsilon ({c_{(2)}}^{[0]} )    \otimes {{c_{(1)}}^{[-1]}}_{(2)} \\
&&\hspace{-10cm} =  {{c_{(1)}}^{[-1]}}_{(1)}  {a^{[1]}}_{(1)} {c_{(2)}}^{[-1]} \epsilon ({c_{(2)}}^{[0]} )    \otimes {{c_{(1)}}^{[-1]}}_{(2)} \left( a , {c_{(1)}}^{[0]} \right).
\eqnast
As this equality holds for any $a\in A$, we conclude by the nondegeneracy of the pairing, that 
\[
c^{[-1]} \otimes c^{[0][-1]} \otimes c^{[0][0]} = {{c_{(1)}}^{[-1]}}_{(1)}  {a^{[1]}}_{(1)} {c_{(2)}}^{[-1]} \epsilon ({c_{(2)}}^{[0]} )    \otimes {{c_{(1)}}^{[-1]}}_{(2)} \otimes {c_{(1)}}^{[0]} .
\]
Therefore, C is a symmetric left partial $H$ comodule coalgebra.
\end{proof}

\subsubsection*{Comodule coalgebras versus module coalgebras}

\begin{thm}\thlabel{modulevscomodulecoalg}
Consider a dual pairing of Hopf algebras $\bk{-,-}:H\ot K\to k$ and let $C$ be a left partial $K$-comodule coalgebra.
Then the map
\[
\begin{array}{rccc}
\cdot : & C \ot H & \rightarrow & C\\
\, & c\otimes h & \mapsto & \sum \bk{h,c^{[-1]}} c^{[0]}
\end{array}
\]
turns $C$ into a right partial $H$-module coalgebra. This construction yields a functor from the category of left partial $K$-comodule coalgebras to the category of right partial $H$-module coalgebras. 
If the dual pairing $\bk{-,-}$ is moreover non-degenerate, then the above functor corestricts to an isomorphism of categories between the category of left partial $K$-comodule coalgebras to the category of rational right partial $H$-module coalgebras, which are those for which $c\cdot H$ is a finitely generated $k$-module for all $c\in C$.
\end{thm}

\subsubsection*{Comodule coalgebras versus module algebras}

\begin{thm}\thlabel{comcoalgvsmodalg}
Let $\bk{-,-}:H\ot K\to k$ be a dual pairing between the Hopf algebras $H$ and $K$ and $(-,-):A\ot C\to k$ be a non-degenerate dual pairing between the algebra $A$ and the coalgebra $C$.
If $C$ is a left $K$-comodule coalgebra with coaction $\rho:C\to K\ot C,\ \rho(c)=c^{[-1]}\ot c^{[0]}$, then $A$ is a partial left $H$-module algebra, with action given by 
\[
(h\cdot a, c)=\bk{h,c^{[-1]}}(a,c^{[0]}),
\]
for all $a\in A$, $c\in C$ and $h\in H$. Furthermore
\begin{itemize}
\item Under these conditions, there is a pairing $\lllangle - , - \rrrangle: \underline{A\# H} \otimes \underline{C\cosmash K} \rightarrow k$ between the algebra $\underline{A\# H}$ and the coalgebra $\underline{C\cosmash K}$ given by $\lllangle a\# h, x\cosmash \xi \rrrangle = (a(h_{(1)}\cdot 1_A),x)
\langle h_{(2)},\xi \rangle $;
\item  if the pairing between $H$ and $K$ is moreover non-degenerate then there is a bijective correspondence between the structures of rational left $H$-module algebra on $A$ and the structures of left $K$-comodule coalgebra on $C$. 
\end{itemize}
\end{thm}

\begin{proof}
The first and last statement follows from 
\thref{modulevscomodule}
and \thref{modulevscomodulecoalg}. The details to verify that $\lllangle - , - \rrrangle$ is a dual pairing are left to the reader, we warn however that computations become very technical.
\end{proof}

\section{Conclusions and outlook}\selabel{conclusions}

In this paper, we introduced the dual notions of module algebras and made several algebraic constructions with them. Let us summarize this in the following table
\begin{center}
\begin{tabular}{c|c|c}
& algebra $A$ & coalgebra $C$ \\
\hline 
$H$-module & $H$-module algebra $A$ & $H$-module coalgebra $C$ \\
\hline
$K$-comodule & $K$-comodule algebra $A$ & $K$-comodule coalgebra $C$\\
\end{tabular}\end{center}
If there is a non-degenerate pairing between the algebra $A$ and the coalgebra $C$, then one can use this duality to relate the structures from the right column with those the left column. If the Hopf algebras $H$ and $K$ are in duality, then the associated pairing between them can be used to move from the lower row to the upper row and if the pairing is non-degenerate and one restricts to rational modules one can also move the other way around.

In the general case, when no commutativity or cocommutativity conditions are imposed, each of these notions leads to an algebraic construction as follows (we keep same rows and collumns)
\begin{center}
\begin{tabular}{c|c|c}
& algebra $A$ & coalgebra $C$ \\ \hline
$H$-module & smash product algebra $\ul{A\# H}$ & $C$-ring $\ul{H\ot C}$ \\ \hline
$K$-comodule & $A$-coring $A\ul \ot K$ & cosmash coproduct coalgebra $\ul{C\cosmash K}$\\
\end{tabular}\end{center}
Furthermore, if one adds commutativity conditions, then the obtained structure is even richer. In particular, we found that if $A$ is a commutative algebra, $H$ is a cocommutative Hopf algebra and $K$ is a commutative Hopf algebra, then $\ul{A\# H}$ and $A\ul \ot K$ are Hopf algebroids and the dualities from the first table are lead to a skew pairing between these Hopf algebroids. 

In fact, we believe that the structure is even richer than what we obtained so far. First, we would like to see all dualities that appear in the first table, also to be apparent in the second table. Indications that this is true are given by the duality (without commutativity constraints) between the $A$-coring structure of $A\ul \ot K$ and the $A$-ring structure of $\ul{A\# H}$ (see \leref{splitcoring}) and the coalgebra structure of $\ul{C\cosmash K}$ and the ring structure of $\ul{A\# H}$ (see \thref{comcoalgvsmodalg}). 
Furthermore, we expect that if $C$ is supposed to be a cocommutative coalgebra, $H$ is cocommutative and $K$ is commutative then $\ul{H\ot C}$ and $\ul{C\cosmash K}$ are Hopf coalgebroids, and all dualities also be brought to this level. All this is subject for future investigations.

\section*{Acknowledgments}

One of the authors, E.\ Batista, is supported by CNPq (Ci\^{e}ncia sem Fronteiras), proc n\b{o} 236440/2012-8, and he would like to thank the {\em D\'epartement de Math\'ematique} of the {\em Universit\'e Libre de Bruxelles} for their kind
hospitality.  E.\ Batista is also partially supported by Funda\c{c}\~ao Arauc\'aria, project n. 490/16032. The last author would like to thank the FNRS for the CDR ``Symmetries of non-compact non-commutative spaces: Coactions of generalized Hopf algebras.'' that partially supported this collaboration. We both would like to thank to Marcelo Muniz S. Alves for fruitful discussions along the elaboration of this paper. We would like to thank Felipe Castro and Glauber Quadros for their useful remarks concerning the duality results presented here.

\newpage
\appendix

\section{Some additional proofs}

In this appendix we provide some additional proofs that were omitted in the final journal version of the paper. We leave them here for convenience of the interested reader.

\begin{proof} [ Proof of \leref{Takprod}]
\ul{(2)}. Consider, $a\# h \in \underline{A\# H}$ and $b\in A$, then
\beqnast
b\triangleright (a\# h_{(1)}) \otimes^r_A 1_A \# h_{(2)} & = & (ba\# h_{(1)}) \otimes^r_A 1_A \# h_{(2)} \\
& = & a b(h_{(1)}\cdot 1_A)\# h_{(2)} \otimes^r_A 1_A \# h_{(3)} \\
& = & a(h_{(1)}\cdot 1_A)(h_{(2)}S(h_{(3)})\cdot b)\# h_{(4)} 
\otimes^r_A 1_A \# h_{(5)}\\
& = & (a (h_{(1)}\cdot (S(h_{(2)})\cdot b))\# h_{(3)}) \otimes^r_A 1_A \# h_{(4)} \\
& = & (a \# h_{(1)})\blacktriangleleft (S(h_{(2)})\cdot b) \otimes^r_A 1_A \# h_{(3)} \\
& = & a \# h_{(1)}\otimes^r_A (S(h_{(2)})\cdot b) \blacktriangleright (1_A \# h_{(3)}) \\
& = & a \# h_{(1)}\otimes^r_A  (h_{(3)}\cdot (S(h_{(2)})\cdot b))\# h_{(4)}\\
& = & a \# h_{(1)}\otimes^r_A  (h_{(2)}\cdot 1_A) (h_{(3)}S(h_{(4)})\cdot b)\# h_{(5)}\\
& = & a \# h_{(1)}\otimes^r_A  (h_{(2)}\cdot 1_A)b\# h_{(3)}\\
& = & a \# h_{(1)}\otimes^r_A  b\# h_{(2)}\\
& = & a \# h_{(1)}\otimes^r_A (1_A \# h_{(2)})\triangleleft b .
\eqnast
Hence the image of $\ud_r$ lies in the right Takeuchi product. The proof that $\ud_r$ is an algbra morphim is very similar.
\end{proof}

\begin{proof}[ Proof of \thref{Hoidpartialaction}]
We will show that $\mathcal{S}$ is an anti-algebra morphism. Indeed, taking $a\# h$ and $b\#k$ in $\underline{A\# H}$, we have
\beqnast
\mathcal{S}((a\# h)(b\# k)) & = & \mathcal{S}(a(h_{(1)}\cdot b)\# h_{(2)}k) \\
& = & (S(h_{(3)}k_{(2)})\cdot (a(h_{(1)}\cdot b)))\# S(h_{(2)}k_{(1)}) \\
& = & (S(h_{(4)}k_{(3)})\cdot  a)(S(h_{(3)}k_{(2)})\cdot (h_{(1)}\cdot b))\# S(h_{(2)}k_{(1)}) \\
& = & (S(h_{(4)}k_{(3)})\cdot  a)(S(k_{(2)})S(h_{(3)})h_{(1)}\cdot b)\# S(h_{(2)}k_{(1)}) \\
& = & (S(h_{(4)}k_{(3)})\cdot  a)(S(k_{(2)})S(h_{(2)})h_{(3)}\cdot b)\# S(h_{(1)}k_{(1)}) \\
& = & (S(h_{(2)}k_{(3)})\cdot  a)(S(k_{(2)})\cdot b)\# S(h_{(1)}k_{(1)}) ,
\eqnast
on the other hand
\beqnast
\mathcal{S} (b\# k) \mathcal{S} (a\# h) & = & ((S(k_{(2)})\cdot b)\# S(k_{(1)}))
((S(h_{(2)})\cdot a)\# S(h_{(1)}))\\
& = & (S(k_{(3)})\cdot b)(S(k_{(2)})\cdot (S(h_{(2)})\cdot a))\# S(k_{(1)})S(h_{(1)})\\
& = & (S(k_{(3)})\cdot b)(S(k_{(2)})S(h_{(2)})\cdot a)\# S(h_{(1)}k_{(1)})\\
& = & (S(h_{(2)}k_{(3)})\cdot  a)(S(k_{(2)})\cdot b)\# S(h_{(1)}k_{(1)}).
\eqnast

For the Hopf algebroid axiom (i) consider $a\in A$. Since $s_l=t_l=s_r=t_r$ in our case, we only have to verify two identies:
\beqnast
s_l \circ \ue_l \circ t_r (a) & = & s_l (\ue_l (a\# 1_H))\\
 =  s_l (a(1_H \cdot 1_A)) 
 &=&  s_l (a) =t_r (a) ,
\eqnast
and
\beqnast 
s_r \circ \ue_r \circ t_l (a) & = & s_r(\ue_r (a\# 1_H)) \\
 =  s_r (S(1_H) \cdot a) 
& = & s_r (a) =t_l (a).
\eqnast

The Hopf algebroid axiom (ii) is straightforward.
\end{proof}

\begin{proof}[ Proof of \thref{Hoidpartialcoaction}]
The map $\TS$ is an anti algebra morphism. Indeed, take $a1^{[0]} \otimes h1^{[1]}$ and $b1^{[0]} \otimes k1^{[1]}$ in  $A\underline{\otimes} H$, then
\beqnast
\TS ((a1^{[0]} \otimes h1^{[1]} ) (b1^{[0]} \otimes k1^{[1]})) & = & \TS (ab1^{[0]} \otimes hk1^{[1]} ) \\
& = & a^{[0]}b^{[0]} 1^{[0]} \otimes a^{[1]} b^{[1]} S(k) S(h) 1^{[1]} \\
& = & \left( b^{[0]}1^{[0]} \otimes b^{[1]} S(k) 1^{[1]} \right) \left( a^{[0]}1^{[0']} \otimes a^{[1]} S(h) 1^{[1']} \right) \\
& = & \TS (b1^{[0]} \otimes k1^{[1]} )\TS (a1^{[0']} \otimes h1^{[1']} ) .
\eqnast

The item (i) of the definition of Hopf algebroid (see \seref{algstruc}) can be easily deduced, for example 
\beqnast
t_l \circ \ue_l \circ s_r  (a) & = & t_l (\ue_l (a^{[0]}\otimes a^{[1]})) 
 =  t_l (a^{[0]}\epsilon (a^{[1]}) 
 =  t_l (a) = s_r (a) ;
\eqnast
the other identities are completely analogous.

The item (ii) is consequence of the fact that the left and right coring structure coincide, and then the compatibility simply means the coassociativity.
\end{proof}

\begin{proof}[ Proof of \prref{Cring}]
The multiplication $\mu$ a morphism of $C$ bicomodules: For the left coaction, we have
\beqnast
& \,  & \lambda \circ \mu \left( \sum_i \underline{ h^i \otimes c^i }\otimes 
\underline{ k^i \otimes d^i } \right) \\ 
& = &  \lambda \left( \sum_i \underline{\epsilon (h^i_{(1)} \cdot c^i ) \epsilon (k^i_{(1)} \cdot d^i_{(1)} ) h^i_{(2)}k^i_{(2)} \otimes d^i_{(2)}} \right) \\
& = & \sum_i \epsilon (k^i_{(1)} \cdot d^i_{(1)} )(h^i_{(2)} k^i_{(2)}\cdot d^i_{(2)})\otimes \underline{\epsilon (h^i_{(1)} \cdot c^i )  h^i_{(3)}k^i_{(3)} \otimes d^i_{(3)}} \\
& = & \sum_i (h^i_{(2)} \cdot (k^i_{(1)}\cdot d^i_{(1)}))\otimes 
\underline{\epsilon (h^i_{(1)} \cdot c^i )  h^i_{(3)}k^i_{(2)} \otimes d^i_{(2)}} \\
& = & \sum_i \epsilon (h^i_{(1)} \cdot c^i )(h^i_{(2)} \cdot (k^i_{(1)}\cdot d^i_{(1)}))\otimes \underline{ \epsilon (k^i_{(2)} \cdot d^i_{(2)}) h^i_{(3)}k^i_{(3)} \otimes d^i_{(3)}} \\
& = & \sum_i \epsilon (h^i_{(1)} \cdot c^i_{(1)} )(h^i_{(2)} \cdot c^i_{(2)})\otimes \underline{ \epsilon (k^i_{(1)} \cdot d^i_{(1)}) h^i_{(3)}k^i_{(2)} 
\otimes d^i_{(2)}} \\
& = & \sum_i (h^i_{(1)} \cdot c^i_{(1)})\otimes \underline { \epsilon (h^i_{(2)} \cdot c^i_{(2)} )\epsilon (k^i_{(1)} \cdot d^i_{(1)}) h^i_{(3)}k^i_{(2)} \otimes d^i_{(2)}} \\
& = & \sum_i h^i_{(1)} \cdot c^i_{(1)} \otimes \mu \left( \underline{ h^i_{(2)} \otimes c^i_{(2)} }\otimes \underline{ k^i \otimes d^i } \right) \\
& = & (I\otimes \mu) \circ (\lambda \otimes I) \left( \sum_i \underline{ h^i \otimes c^i }\otimes \underline{ k^i \otimes d^i } \right) .
\eqnast
And for the right coaction,
\beqnast
& \,  & \rho \circ \mu \left( \sum_i \underline{ h^i \otimes c^i }\otimes \underline{ k^i \otimes d^i } \right) \\ 
& = &  \rho \left( \sum_i \underline{\epsilon (h^i_{(1)} \cdot c^i ) \epsilon (k^i_{(1)} \cdot d^i_{(1)} ) h^i_{(2)}k^i_{(2)} \otimes d^i_{(2)}} \right) \\
& = &   \sum_i \underline{\epsilon (h^i_{(1)} \cdot c^i ) \epsilon (k^i_{(1)} \cdot d^i_{(1)} ) h^i_{(2)}k^i_{(2)} \otimes d^i_{(2)}} \otimes d^i_{(3)} \\
& = &   \sum_i \mu \left( \underline{h^i \otimes c^i }\otimes \underline{ k^i \otimes d^i_{(1)}} \right) \otimes d^i_{(3)} \\
& = & (\mu \otimes I)\circ (I\otimes \rho ) \left( \sum_i \underline{ h^i \otimes c^i }\otimes \underline{ k^i \otimes d^i } \right) .
\eqnast

We can prove also that this multiplication is associative. Indeed, on one hand we have,
\beqnast
& \, & \mu \circ (\mu \otimes I ) \left( \sum_i \underline{ h^i \otimes c^i }\otimes \underline{ k^i \otimes d^i } \otimes \underline{ l^i \otimes e^i }\right) \\ 
& = & \mu \left( \sum_i \underline{\epsilon (h^i_{(1)} \cdot c^i ) \epsilon (k^i_{(1)} \cdot d^i_{(1)} ) h^i_{(2)}k^i_{(2)} \otimes d^i_{(2)}} \otimes \underline{ l^i \otimes e^i } \right) \\
& = & \sum_i \underline{ \epsilon (h^i_{(1)} \cdot c^i ) \epsilon (k^i_{(1)} \cdot d^i_{(1)} ) \epsilon (h^i_{(2)}k^i_{(2)} \cdot d^i_{(2)}) \epsilon (l^i_{(1)} \cdot e^i_{(1)} ) h^i_{(3)}k^i_{(3)}l^i_{(2)} \otimes e^i_{(2)} } \\
& = & \sum_i \underline{ \epsilon (h^i_{(1)} \cdot c^i ) \epsilon (h^i_{(2)}\cdot (k^i_{(1)} \cdot d^i)) \epsilon (l^i_{(1)} \cdot e^i_{(1)} ) h^i_{(3)}k^i_{(2)}l^i_{(2)} \otimes e^i_{(2)} } \\
& = & \sum_i \underline{ \epsilon (h^i_{(1)} \cdot c^i ) 
\epsilon (h^i_{(2)}\cdot (k^i_{(1)} \cdot d^i_{(1)})) 
\epsilon (k^i_{(2)}\cdot d^i_{(2)})\epsilon (l^i_{(1)} \cdot e^i_{(1)} ) h^i_{(3)}k^i_{(2)}l^i_{(2)} \otimes e^i_{(2)} } \\
& = & \sum_i \underline{ \epsilon (h^i_{(1)} \cdot c^i_{(1)} ) 
\epsilon (h^i_{(2)}\cdot c^i_{(2)}) 
\epsilon (k^i_{(1)}\cdot d^i )\epsilon (l^i_{(1)} \cdot e^i_{(1)} ) h^i_{(3)}k^i_{(2)}l^i_{(2)} \otimes e^i_{(2)} } \\
& = & \sum_i \underline{ \epsilon (h^i_{(1)} \cdot c^i ) 
\epsilon (k^i_{(1)}\cdot d^i )\epsilon (l^i_{(1)} \cdot e^i_{(1)} ) h^i_{(2)}k^i_{(2)}l^i_{(2)} \otimes e^i_{(2)} } .
\eqnast

On the other hand,
\beqnast
& \, & \mu \circ (I\otimes \mu ) \left( \sum_i \underline{ h^i \otimes c^i }\otimes \underline{ k^i \otimes d^i } \otimes \underline{ l^i \otimes e^i } \right) \\ 
& = & \mu \left( \sum_i \underline{ h^i \otimes c^i }\otimes 
\underline{ \epsilon (k^i_{(1)} \cdot d^i )\epsilon (l^i_{(1)} \cdot e^i_{(1)}) k^i_{(2)}l^i_{(2)}\otimes e^i_{(2)} } \right)\\
& = &  \sum_i \underline{ \epsilon (h^i_{(1)} \cdot c^i ) 
\epsilon (k^i_{(1)} \cdot d^i ) \epsilon (l^i_{(1)} \cdot e^i_{(1)}) 
\epsilon (k^i_{(2)}l^i_{(2)}\cdot e^i_{(2)}) h^i_{(2)}k^i_{(3)}l^i_{(3)}\otimes e^i_{(3)}} \\
& = &  \sum_i \underline{ \epsilon (h^i_{(1)} \cdot c^i ) 
\epsilon (k^i_{(1)} \cdot d^i )  
\epsilon (k^i_{(2)} \cdot (l^i_{(1)}\cdot e^i_{(1)})) h^i_{(2)}k^i_{(3)}l^i_{(2)}\otimes e^i_{(2)}} \\
& = &  \sum_i \underline{ \epsilon (h^i_{(1)} \cdot c^i ) 
\epsilon (k^i_{(1)} \cdot d^i )  
\epsilon (k^i_{(2)} \cdot (l^i_{(1)}\cdot e^i_{(1)})) \epsilon (l^i_{(2)} \cdot e^i_{(2)}) h^i_{(2)}k^i_{(3)}l^i_{(3)}\otimes e^i_{(3)}} \\
& = &  \sum_i \underline{ \epsilon (h^i_{(1)} \cdot c^i ) 
\epsilon (k^i_{(1)} \cdot d^i_{(1)} )  
\epsilon (k^i_{(2)} \cdot d^i_{(2)}) \epsilon (l^i_{(1)} \cdot e^i_{(1)}) h^i_{(2)}k^i_{(3)}l^i_{(2)}\otimes e^i_{(2)}} \\
& = &  \sum_i \underline{ \epsilon (h^i_{(1)} \cdot c^i ) 
\epsilon (k^i_{(1)} \cdot d^i )   \epsilon (l^i_{(1)} \cdot e^i_{(1)}) h^i_{(2)}k^i_{(2)}l^i_{(2)}\otimes e^i_{(2)}} .
\eqnast

Let us check the unit axioms for the $C$-ring, i.e.
\begin{eqnarray*}
\mu \circ (\eta \otimes I)\circ \lambda = I= \mu \circ (I\otimes \eta ) \circ \rho =I
\end{eqnarray*}
For the first identity, take $\underline{h\otimes c}\in \underline{H\otimes C}$,
\beqnast
\mu \circ (\eta \otimes I)\circ \lambda (\underline{h\otimes c}) & = & 
\mu (\underline{1_H \otimes h_{(1)}\cdot c_{(1)}} \otimes \underline{ h_{(2)} \otimes c_{(2)} }) \\
& = & \underline{ \epsilon ( 1_H \cdot (h_{(1)}\cdot c_{(1)})) 
\epsilon (h_{(2)} \cdot c_{(2)} ) h_{(3)} \otimes c_{(3)}} \\
& = & \underline{ \epsilon (h_{(1)}\cdot c_{(1)}) h_{(2)} \otimes c_{(2)}} 
 =  \underline{h\otimes c} .
\eqnast
and for the second identity,
\beqnast
\mu \circ (I \otimes \eta )\circ \rho (\underline{h\otimes c}) & = & 
\mu (\underline{h\otimes c_{(1)}}\otimes \underline{1_H \otimes c_{(2)}}) \\
& = & \underline{ \epsilon (h_{(1)}\cdot c_{(1)})\epsilon (1_H \cdot c_{(2)}) h_{(2)} \otimes c_{(3)} } \\
& = & \underline{ \epsilon (h_{(1)}\cdot c_{(1)})\epsilon (c_{(2)}) h_{(2)} \otimes c_{(3)} } \\
& = & \underline{ \epsilon (h_{(1)}\cdot c_{(1)}) h_{(2)} \otimes c_{(2)} } 
 =  \underline{h\otimes c} .
\eqnast
Therefore, $\underline{H\otimes C}$ is a $C$-ring.
\end{proof}

\end{document}